\documentclass[hidelinks,12pt]{amsart}
\usepackage[dvipsnames]{xcolor}
\usepackage{xcolor,amssymb,enumerate,tikz-cd,hyperref,amsmath,mathrsfs,mathtools,enumitem,changepage,mathdots,amssymb }

\usepackage{eucal}

\calclayout

\author{Alexei Latyntsev}

\theoremstyle{definition}

\newtheorem{theorem}[subsubsection]{Theorem}

\newtheorem{prop}[subsubsection]{Proposition}

\newtheorem{cor}[subsubsection]{Corollary}

\newtheorem{lem}[subsubsection]{Lemma}

\newtheorem*{theorem*}{Theorem}
\newtheorem*{prop*}{Proposition}
\newtheorem*{lem*}{Lemma}
\newtheorem*{cor*}{Corollary}
\newtheorem*{defn*}{Definition}
\newtheorem*{pdefn*}{Proto-Definition}
\newtheorem*{theoremmain*}{Theorem \ref{mainthm}}

\def\End{\mathop{\text{End}}}
\def\Frac{\mathop{\text{Frac}}}
\def\Spec{\mathop{\text{Spec}}}
\def\Sym{\mathop{\text{Sym}}}

\makeatletter
\pgfqkeys{/tikz/commutative diagrams}{
  row sep/.code={\tikzcd@sep{row}{#1}{}},
  column sep/.code={\tikzcd@sep{column}{#1}{}},
  bo row sep/.code={\tikzcd@sep{row}{#1}{between origins}},
  bo column sep/.code={\tikzcd@sep{column}{#1}{between origins}},
  bo column sep/.default=normal,
  bo row sep/.default=normal,
}
\def\tikzcd@sep#1#2#3{% re-defintion of original package macro!
  \pgfkeysifdefined{/tikz/commutative diagrams/#1 sep/#2}%
    {\pgfkeysalso{/tikz/#1 sep={\ifx\\#3\\1*\else1.7*\fi\pgfkeysvalueof{/tikz/commutative diagrams/#1 sep/#2},#3}}}%
    {\pgfkeysalso{/tikz/#1 sep={#2,#3}}}}
\makeatother

\begin{document}

 \title{Cohomological Hall algebras and vertex algebras}
 \maketitle

\begin{adjustwidth}{20pt}{20pt}
\small{\textsc{Abstract}: The moduli stack of representations of a quiver, or coherent sheaves on a proper curve, carries two structures on its cohomology: a Hall algebra and braided vertex coalgebra. We show that they are compatible, by developing a formulation of abelian localisation which works in the derived/singular setting. An application of these techniques gives an explicit formula for the cohomological Hall algebra products.}
\end{adjustwidth}

\section*{Contents}

\begin{itemize}[leftmargin=10pt]
\item[1] \hyperref[Intro]{Introduction}  \enspace\dotfill\enspace \pageref{Intro}
%\begin{enumerate}
%\item[\ref{motivation}] Motivation \enspace\dotfill\enspace \pageref{motivation}
%\item[\ref{results}] Results \enspace\dotfill\enspace \pageref{results}
%\item[\ref{relation}] Relation to other work \enspace\dotfill\enspace \pageref{relation}
%\item[\ref{ack}] Acknowledgements \enspace\dotfill\enspace \pageref{ack}
%\end{enumerate} 

\item[2] \textcolor{black}{\hyperref[coha]{Cohomological Hall algebras}}   \enspace\dotfill\enspace \pageref{coha}
\item[3] \textcolor{black}{\hyperref[va]{Vertex algebras and the cohomology of moduli spaces}}   \enspace\dotfill\enspace \pageref{va}
\item[4] \textcolor{black}{\hyperref[description]{Examples}}   \enspace\dotfill\enspace \pageref{description}
\item[5] \textcolor{black}{\hyperref[euler]{Euler classes}}   \enspace\dotfill\enspace \pageref{euler}
\item[6] \hyperref[note]{Relative case}   \enspace\dotfill\enspace \pageref{note}
\item[7] \textcolor{black}{\hyperref[spec]{Specialisation, deformation and fundamental classes}}   \enspace\dotfill\enspace \pageref{spec}
\item[8] \hyperref[abloc]{Abelian localisation}   \enspace\dotfill\enspace \pageref{abloc}
\item[9] \textcolor{black}{\hyperref[shclass]{Sheaves on classifying spaces}}   \enspace\dotfill\enspace \pageref{shclass}
\item[10] \hyperref[mainsect]{The main theorem}   \enspace\dotfill\enspace \pageref{mainsect} 
\item[11] \hyperref[explicit]{Explicit computations}   \enspace\dotfill\enspace \pageref{explicit}
\item[{}]  
\item[A] \textcolor{black}{\hyperref[sheaves]{Spaces and sheaves}} \enspace\dotfill\enspace \pageref{sheaves}
\item[B] \textcolor{black}{\hyperref[tot]{Total spaces of perfect complexes}}   \enspace\dotfill\enspace \pageref{tot}
\item[C] \hyperref[stratifications]{Stratifications}   \enspace\dotfill\enspace \pageref{stratifications} 
\item[D] \hyperref[stab]{Stabiliser groups}   \enspace\dotfill\enspace \pageref{stab} 
\item[] \textcolor{black}{\hyperref[references]{References}}   \enspace\dotfill\enspace \pageref{references}
\end{itemize}

\section{Introduction} \label{Intro} 

\subsection{} There are (at least) two deep structures coming from moduli spaces of objects in an abelian category, which are expected to be related to string theory.

\begin{center}
\begin{tikzpicture}
\node [] at (0,0) {$\mathcal{A}$};
\node [] at (2,0) {$X$};

\node [] at (1,0) {$\rightsquigarrow$};

\node [rotate=28] at (3.7,0.65) {$\rightsquigarrow$};
\node [rotate=-28] at (3.7,-0.65) {$\rightsquigarrow$};

\node [left] at (3.7,1) {\small{\cite{KS}}};
\node [left] at (3.7,-1) {\small{\cite{J}}};

\node [right] at (5,1) {$\text{H}^\cdot(X)\otimes \text{H}^\cdot(X)\to \text{H}^\cdot(X)$};
\node [right] at (5,-1) {$\text{H}_\cdot(X)\stackrel{Y}{\to} \End\text{H}_\cdot(X)((z))$};

\node [] at (0,-1.5) {\textcolor{white}{a}};

\end{tikzpicture}
\end{center}
\textit{Cohomological Hall algebras (CoHAs)} use short exact sequences in a abelian category $\mathcal{A}$ to define an algebra structure on the cohomology of  its moduli stack of objects $X$:
\begin{center}
\begin{tikzcd}[bo column sep,execute at end picture={
\node [] at ([yshift=0.7cm,xshift=2cm]F) {$\rightsquigarrow$};
\node [right] at ([yshift=0.7cm,xshift=4cm]F) {$\text{H}^\cdot(X)\otimes \text{H}^\cdot(X)\stackrel{p_*q^*}{\to} \text{H}^\cdot(X)$.};
\node [] at ([yshift=1.1cm,xshift=2cm]F) {\small{sect. \ref{coha}}};
}]

   &\text{Ext}\arrow[rd,"p"]\arrow[ld,"q",swap]&\\
X\times X & & |[alias=F]|  X 
\end{tikzcd}
\end{center}
 Here $\text{Ext}$ is the moduli space of short exact sequences in $\mathcal{A}$, and $p$ and $q$ pick out the middle and outer terms, respectively. The corresponding physics notion is an algebra of \textit{BPS states}. 

The second uses a $K$-linear abelian structure on $\mathcal{A}$ to define a holomorphic \textit{vertex algebra} 
\begin{center}
\begin{tikzpicture}

\node [] at (0,-0.07) {$K^\times\to \text{End}(\mathcal{A})$};
\node [] at (4,0) {$\mathcal{A}\times\mathcal{A}\stackrel{\oplus}{\longrightarrow}\mathcal{A}$};

\node [] at (7,-0.1) {$\rightsquigarrow$};
\node [] at (7,0.3) {\small{sect. \ref{va}}};

\node [right] at (9,0) {$\text{H}_\cdot(X)\to \End\text{H}_\cdot(X)((z)),$};

\end{tikzpicture}
\end{center}
then twists it using the extensions in $\mathcal{A}$. Vertex algebras formalise the notion of $2$\textit{d conformal field theory}; they consist of a vector space $V$ of ``states living on $\mathbf{A}^1$" and a map 
$$Y(z)\ :\ V \ \longrightarrow \ \End V((z)),$$
which expresses the effect of inserting another state at point $z$. When independent of $z$ this is just a commutative algebra.

\subsubsection{} CoHAs were discovered in \cite{KS} by Kontsevich and Soibelman, vertex algebras by Borcherds in \cite{B} and the vertex algebra structure on $\text{H}_\cdot(X)$ by Joyce in \cite{J}.

\subsubsection{} Our main theorem relates these two structures: if $\mathcal{A}$ is the category of finite dimensional representations of a finite quiver or coherent sheaves on a smooth proper curve, and $X$ is its moduli stack of objects, then
\begin{theoremmain*} $\text{H}^\cdot(X)$ is a braided super vertex bialgebra.
\end{theoremmain*}
Informally, this theorem says that the two structures interact well with each other; for definitions see section \ref{vadefn}.  The basic idea behind the proof is that both structures are related to Euler classes.

\subsubsection{} Our categories have dimension one: any higher and the CoHA construction does not work as $p$ is not quasismooth, so $p_*$ does not exist. Also implicit in the theorem is the fact that we actually get a \textit{braided} super vertex algebra structure on homology; this is a genuine super vertex algebra for symmetric quivers and elliptic curves. 

\subsubsection{} Throughout $\text{H}^\cdot(X)$ denotes a \textit{rational} cohomology theory over a field $k$, for example de Rham cohomology over $\mathbf{C}$ or $\ell$-adic cohomology over $\mathbf{Q}_\ell$ or singular cohomology over $\mathbf{Q}$. Moreover $X$ is a derived Artin stack; see appendix \ref{sheaves} for details.

\subsection{``Free rank one boson" example: $\mathcal{A}=\text{Vect}$} The moduli space classifying finite dimensional vector spaces is $X=\coprod \text{BGL}_n$. Recall that the classifying space $\text{BGL}_n$ classifies rank $n$ vector bundles, and its cohomology is generated by Chern classes:
$$\text{H}^\cdot(\text{BGL}_n) \ = \ k[c_1,...,c_n].$$
\subsubsection{} The CoHA product can be expressed in terms of the Chern roots
$$(f(s_1,...,s_n),g(t_1,...,t_m)) \ \longmapsto \  \sum_{\sigma} \frac{1}{e(N_\sigma)}f(u_{\sigma 1},...,u_{\sigma n}) g(u_{\sigma 1},...,u_{\sigma m});$$
summing is over all partitions $\sigma$ of $\{1,...,n+m\}$ into two subsets of sizes $n$ and $m$, where
$$e(N_\sigma) \  = \ \prod_{i=1}^n \prod_{j=1}^m (u_{\sigma i} - u_{\sigma j}).$$
For a proof, and an explanation of how $e(N_\sigma)$ is an Euler class of a normal bundle $N_\sigma$, see section \ref{heuristicvect} or \cite{KS}.

\subsubsection{} As a vertex algebra $\text{H}_\cdot(X)$ is a subalgebra of a \textit{lattice vertex algebra} $V_\Lambda$, which is a ``noncommutative quantisation" of the space of maps
$$D\ \longrightarrow\ \mathbf{C}^n/\Lambda$$
from the (formal) disc $D$ into a torus: it has a filtration with associated graded $\text{gr}V_\Lambda\simeq \mathcal{O}\left(\text{Map}(D,\mathbf{C}^n/\Lambda)\right)$, which is an example of a \textit{commutative} vertex algebra (see \cite{AM}).

\subsubsection{} This works because we can enlarge $\mathcal{A}=\text{Vect}$ to the triangulated category $\mathcal{T}$ of perfect complexes, which has heart $\mathcal{A}$. This gives an inclusion 
$$\text{BGL}_n\to \text{Perf}_n \textcolor{white}{aaaaaa} \rightsquigarrow\textcolor{white}{aaaaaa}  \text{H}_\cdot(X)\ \hookrightarrow \ \text{H}^\cdot(\text{Perf})$$
and Joyce's construction makes the right into a lattice vertex algebra:
$$\text{H}^\cdot(\text{Perf}) \ \simeq\ V_\mathbf{Z}.$$
Note that Joyce's construction works equally well for triangulated as for abelian categories. Here $\text{Perf}$ (resp. $\text{Perf}_n$) is the moduli stack classifying (resp. $n$) perfect complexes. Its cohomology is generated by Chern classes of the tautological complex $\gamma$
$$\text{H}^\cdot(\text{Perf}_n) \ = \ k[c_1,c_2,...\ ].$$
Now to view the homology of $\text{Perf}=\coprod \text{Perf}_n$ as a vertex algebra; to zeroeth approximation it is the algebra structure given by the direct sum $\oplus : \text{Perf}_n\times\text{Perf}_m\to \text{Perf}_{n+m}$, in the dual monomial basis it sends $c^\vee_i\cdot c^\vee_j=c^\vee_{i+j}$. Next, the automorphism group of every object in $\mathcal{T}$ containing $k^\times$ gives rise to a map
$$\text{B}\mathbf{G}_m \times \text{Perf} \ \longrightarrow \ \text{Perf}.$$
Multiplying by the generator $t$ of $\text{H}_2(\text{B}\mathbf{G}_m)$ gives a derivation of this algebra, and
$$\text{H}_\cdot(\text{Perf}) \ \longrightarrow \ \End \text{H}_\cdot(\text{Perf})((z))\ \textcolor{white}{aaaaaaaaaa} \alpha \ \longrightarrow \  e^{zt}\alpha$$
defines a vertex algebra ``without poles". To finish, we twist by the complex $\theta=\gamma^\vee\boxtimes \gamma$ on $\text{Perf}^2$, whose fibre above a pair of perfect complexes is the vector space of maps between them: 
$$e^{zt}\alpha \ \ \ \rightsquigarrow \ \ \  Y(\alpha,z) \ = \ \oplus_*e^{zt} \sum_{k\ge 0} z^{\text{rk}\theta-k}  \left(c_k(\theta)\cdot \alpha\otimes (-)\right)$$
where $t$ acts on the first factor. This is an example of a rank one lattice super vertex algebra:
$$\text{H}_\cdot(\text{Perf})\ \simeq\ V_\mathbf{Z}.$$

\subsection{Abelian localisation} The main tool we use and extend in this paper is \textit{abelian localisation}. This is a powerful method for ``turning geometry into combinatorics" by localising cohomological computations onto fixed loci. In the classical setting of a manifold $X$ acted on by a torus $\text{T}$,  Atiyah and Bott explain in \cite{AB} how the equivariant cohomology of $X$ and $X^\text{T}$ are very close to being equal, then prove the \textit{integration formula}
$$\int_X\alpha\ =\ \int_{X^\text{T}}\frac{\alpha\vert_{X^\text{T}}}{e(N)}$$
for equivariant cohomology classes $\alpha$. Here $N$ is the normal bundle to $X^\text{T}$. This is extensively used in Gromov-Witten theory, and there is a folk heuristic relating it to Grothendieck-Riemann-Roch. 

\subsubsection{} Our first task is generalising this to singular schemes. The main obstacle is that the Euler class $e(N)$ becomes significantly subtler since $N$ is no longer a vector bundle. Our answer is the \textit{bivariant Euler class},
$$e(X^\text{T}/X)\ \in\ \text{H}^\cdot(X^\text{T}/X).$$
This lives in the bivariant homology group of the map $i:X^\text{T}\to X$, a notion which was first introduced by Fulton and Macpherson in \cite{FM} (see appendix \ref{sheaves}) as an analogue of ordinary cohomology especially well suited to the analysis of singular spaces. 

We justify our answer in section \ref{euler}, where a Whitney sum formula (section \ref{whitneysingular}) allows us to make sense of statements like
$$E \ = \ (E_0\stackrel{f}{\to} E_1) \ \ \ \rightsquigarrow \ \ \ e(E) \ = \ e(E_0)/e(E_1)$$
where $E$ is a perfect complex (``derived vector bundle") represented by a two-term complex of vector bundles; its classical truncation $E_{cl}=\text{ker}f$ is singular. Further, in section \ref{eqvtloc} we construct a cohomological shadow of the (nonexistent) singular exponential map $N\to X$, called \textit{equivariant specialisation}, which sends
$$e(X^\text{T}/X) \ \longrightarrow \ e(X^T/N).$$
Finally, in the same section we show that these two Euler classes are also compatible with fundamental classes in a suitable sense, which will allow us to compute integrals.

\subsubsection{} To get the above two statements to work, we need to work in the \textit{localised equivariant} setting. If $X$ is a space with a torus action, the equivariant cohomology $\text{H}^\cdot_\text{T}(X)$ is a module over $\text{H}^\cdot_\text{T}(\text{pt})=k[t_1,...,t_n]$, its localisation is
$$\text{H}_\text{T}(X)\otimes_{\text{H}_\text{T}(\text{pt})} \Frac \text{H}_\text{T}(\text{pt}).$$
One recurring point of view throughout is that, on the level of localised equivariant cohomology, singular spaces behave as if they were smooth (likewise bivariant homology, etc.). So killing the singular data contained in the torsion removes any obstruction to Whitney sum, constructing exponential maps on cohomology, and proving abelian localisation.

\subsubsection{} Our second task is generalising to stacks. Most of the work on singular schemes carries through, except that since the cohomology of a finite type stack is not necessarily finite dimensional, it is not as easy to prove that the equivariant cohomology of a stack with free action is torsion (moreover it is not immediately obvious what free should mean).

We satisfy ourselves with proving abelian localisation in the relative case, where $X\to B$ is a representable map of stacks with an action of $T$ (so it is an equivariant map for the trivial action on $B$); this is enough for our purposes. The point of section \ref{note} is to formalise the argument that the equivariant cohomology of $X\setminus X^\text{T}$ is torsion since by the scheme case the fibres of $(X\setminus X^\text{T})/\text{T}\to B$ have torsion cohomology.

\subsubsection{} Having developed these tools, proving the integration formula Theorem \ref{intformula} and using it to compute push-pull maps in section \ref{cohaproducts} follows straightforwardly.

\subsection{History} Cohomological Hall algebras descend from the construction of \textit{Hall algebras} of an abelian category over a finite field. There are close connections between CoHAs and Donaldson-Thomas theory (\cite{DM}, \cite{KV}, \cite{Sz}) and quantum groups (\cite{DM}, \cite{YZ2}). 

Shortly after Borcherds defined vertex algebras, Beilinson and Drinfeld in \cite{BD} discovered the closely related notion of \textit{chiral algebra}. Both notions have since had a profound impact on mathematics (\cite{Lu}, \cite{FG}, \cite{Ga}, \cite{GL}, \cite{KV}, to name a few). 

Our result gives a new perspective to the growing body of work (\cite{KV}, \cite{RSYZ}) exploring the deep relationship between vertex algebras and CoHA-like structures. 

We bypass the use of perfect obstruction theories by using the recent work \cite{K} of Khan on fundamental classes, making our work on abelian localisation substantially cleaner than it might have otherwise been.

\subsection{Summary of sections} \textbf{Section \ref{coha}} recalls what a cohomological Hall algebra is, \textbf{section \ref{va}} recalls what a (braided) vertex algebra is and the construction in \cite{J}, and \textbf{section \ref{description}} describes explicitly the moduli spaces we are interested in.

\textbf{Sections \ref{euler} to \ref{abloc}} form the technical bulk of this paper, whose result is a formulation of abelian localisation which works for certain singular derived stacks. \textbf{Section \ref{euler}} is about the bivariant Euler class, \textbf{section \ref{note}} proves vanishing results for localised equivariant cohomology of certain Artin stacks, \textbf{section \ref{spec}} recaps the fundamental class of a quasismooth map, the smooth specialisation construction then generalises to the singular case, and \textbf{section \ref{abloc}} proves abelian localisation, the integration formula, and a method for computing CoHA-like products in the singular case.

\textbf{Section \ref{shclass}} gives an explicit description of the dg category of sheaves on quotient stacks, the aim being to give a more concrete idea of the objects considered in the paper.

An impatient reader should skip to \textbf{sections \ref{mainsect} and \ref{explicit}}. \textbf{Section \ref{mainsect}} gives a heuristic to compute CoHA products, and is used to prove the main theorem \ref{mainthm} and in \textbf{section \ref{explicit}} to compute CoHA products explicitly for quiver representations and coherent sheaves on curves.

\textbf{Appendix \ref{sheaves}} sets out which categories of spaces and sheaves we are working with and recaps the six functor formalism and bivariant homology, \textbf{appendix \ref{tot}} recalls the contruction of the total space of a perfect complex and \textbf{appendix \ref{stratifications}} records some results about stratified spaces, and  \textbf{appendix \ref{stab}} some results about stabiliser groups.

\subsection{Acknowledgements} \label{ack} I would like to thank Kevin Costello, Adeel Khan, Minhyong Kim, Frances Kirwan and Kevin Lin for many interesting conversations around themes of the paper, and the EPSRC for funding me. I would especially like to thank my supervisor Dominic Joyce, who proposed this project and suggested the idea of using abelian localisation in the first place.

\section{Cohomological Hall algebras} \label{coha}

\subsection{} The cohomology of a space with a little extra structure (which exists for many moduli spaces) carries a ``cohomological Hall" algebra structure, which we will now introduce essentially as in \cite{KS}, however following \cite{K} we make explicit the relation to quasismoothness. 

For the precise meanings of \textit{space} and \textit{cohomology}, see appendix \ref{sheaves}. The reader will not lose anything by focusing on the examples of Artin stacks and their $\ell$-adic or de Rham cohomology, or of topological spaces and singular cohomology.

\subsection{} \label{cohaconst} Take a space $X$ and a correspondence
\begin{center}
\begin{tikzcd}
&\text{Ext}\arrow[rd,"p"]\arrow[ld,"q",swap]&\\
X\times X & & X
\end{tikzcd}
\end{center} 
such that $p$ is proper and quasismooth, so as to give a pushforward map on cohomology
$$p_*\ :\ \text{H}^\cdot(\text{Ext})\ \longrightarrow\ \text{H}^{\cdot }(X),$$
defined in \ref{umkehr}. Quasismoothness is guaranteed if both spaces are smooth. Note that unlike $q^*$ this map is not degree-preserving, and on each component of $\text{Ext}$ it has degree $-2\dim p$. Associated to this data is a \textbf{cohomological Hall algebra} (CoHA) structure
$$p_*q^* \ :\ \text{H}^\cdot(X\times X)\ \longrightarrow\ \text{H}^\cdot(X).$$
This algebra is associative if the composite correspondence
\begin{equation}\label{diag12}
\begin{tikzcd}[bo column sep]
 & &\text{Ext}_3\arrow[rd]\arrow[ld,swap] & &\\
 &\text{Ext}\times X\arrow[rd,"p\times \text{id}"]\arrow[ld,swap]& &\text{Ext}\arrow[rd]\arrow[ld,swap]&\\
X\times X \times X& & X\times X & & X
\end{tikzcd}
\end{equation}
is isomorphic to the correspondence where we replace $\text{Ext}\times X$ by $X\times \text{Ext}$. The structure is unital if there is a singleton connected component $\text{pt} \subseteq X$ such that the restriction of the correspondence to $\text{pt}\times X$ and $X\times \text{pt}$ is
\begin{center}
\begin{tikzcd}
&X\arrow[rd,"\text{id}"]\arrow[ld,"\text{id}",swap]&\\
X & & X
\end{tikzcd}
\end{center} 
In this case, the CoHA identity element is $1\in \text{H}^\cdot(\text{pt})$.

\subsection{} Examples tend to come from $X$ the moduli space of objects in an abelian category $\mathcal{A}$. In this case $\text{Ext}$ is the moduli space classifying short exact sequences in $\mathcal{A}$, and $p,q$ act on points as
\begin{center}
\begin{tikzcd}[bo column sep]
&(0\to\mathcal{E}_1\to\mathcal{E}_2\to\mathcal{E}_3\to 0)\arrow[rd,|->,"p"]\arrow[ld,swap,|->,"q"]&\\
(\mathcal{E}_1,\mathcal{E}_3) & & \mathcal{E}_2
\end{tikzcd}
\end{center} 
For most abelian categories $\mathcal{A}$ the map $p:\text{Ext}\to X$ is unfortunately not quasismooth and so we do not get an algebra structure on cohomology this way. Indeed, its tangent complex is 
$$\mathbf{T}_p\ =\ q^*\text{Hom}(\ ,\ )$$
where $\text{Hom}(\ ,\ )$ is the perfect complex on $X\times X$ whose fibre at $(V,W)$ is the dg vector space $\text{Hom}(V,W)$. This is concentrated in degrees $0$ and $1$, and the map is quasismooth, precisely when $\mathcal{A}$ has dimension at most one. Finally, $\text{Ext}_3$ classifies length three sequences of injections and the correspondence (\ref{diag12}) acts on objects as
\begin{equation}\label{diag13}
\begin{tikzcd}[bo column sep]
 & &(\mathcal{E}\subseteq\mathcal{E}'\subseteq\mathcal{E}'')\arrow[rrdd,|->]\arrow[ld,swap,|->] & &\\
 &(\mathcal{E}\subseteq\mathcal{E}',\mathcal{E}''/\mathcal{E}')\arrow[ld,swap]& &&\\
(\mathcal{E},\mathcal{E}'/\mathcal{E},\mathcal{E}''/\mathcal{E}') & &  & & \mathcal{E}''
\end{tikzcd}
\end{equation}
and for the correspondence using $X\times\text{Ext}$ in place of $\text{Ext}\times X$, the above is the same except that the middle term is 
$$(\mathcal{E}\subseteq\mathcal{E}',\mathcal{E}''/\mathcal{E}')\ \ \  \rightsquigarrow \ \ \ (\mathcal{E}, \mathcal{E}'/\mathcal{E}\subseteq \mathcal{E}''/\mathcal{E})$$
and so the composite maps in (\ref{diag13}) are left unchanged. Thus we have
\begin{prop}
Let $X$ be either the moduli space of representations of a finite quiver $Q$, or of coherent sheaves on a curve $C$. Then the above construction endows its cohomology $\text{H}^\cdot(X)$ with an algebra structure. 
\end{prop}

\subsubsection{} For concreteness we give the dimension zero example, when $\mathcal{A}=\text{Vect}$ is the category of finite dimensional vector spaces. The correspondence is disjoint union over all $n,m\ge 0$ of
\begin{center}
\begin{tikzcd}[bo column sep]
&\text{BP}_{n,m}\arrow[rd,"p"]\arrow[ld,"q",swap]&\\
\text{BGL}_n\times \text{BGL}_m& & \text{BGL}_{n+m}
\end{tikzcd}
\end{center} 
Here the parabolic subgroup $\text{P}_{n,m}\subseteq \text{GL}_{n+m}$ is the stabiliser of a fixed dimension $n$ subspace. The map $q$ and $p$ are induced by the inclusion and quotient maps
$$\text{P}_{n,m}\to \text{GL}_{n+m} \ \ \ \text{ and }\ \ \ \text{P}_{n,m}\to \text{GL}_n\times \text{GL}_m.$$
Since all spaces are smooth, $p$ is quasismooth. Note that $\text{BP}_{n,m}$ carries a tautological short exact sequence of vector bundles
$$0\ \longrightarrow\ \mathcal{E}_n\ \longrightarrow\ \mathcal{E}_{n+m}\ \longrightarrow\ \mathcal{E}_m\ \longrightarrow\ 0$$
Each vector bundle is a pullback of a tautological bundle: the outer ones by $q$, and the middle one by $p$. Since $\text{BP}_{n,m}$ is the Grassmann bundle of $n$-planes inside $\mathcal{E}_{n+m}$, the cohomology $\text{H}^\cdot(\text{BP}_{n,m})$ is a free module over $\text{H}^\cdot(\text{BGL}_{n+m})$, generated as ring over $\text{H}^\cdot(\text{BGL}_{n+m})$ by the Chern classes of $\mathcal{E}_n,\mathcal{E}_m$ subject to the relation $c(\mathcal{E}_{n+m})=c(\mathcal{E}_n)c(\mathcal{E}_m)$. 

The pushforward $p_*$ then picks out the coefficient of the highest degree generator. The pullback $q^*$ is an isomorphism.

\subsection{} The above is the ``off-shell'' version of cohomological Hall algebras appearing in \cite{KS}, with zero potential $W=0$. There are many other variants, for instance using Borel-Moore homology or the cohomology with coefficients in a sheaf of vanishing cycles. These are not relevant to us because as yet there is no analogue of Joyce's vertex algebra construction living on them.

\section{Vertex algebras and the cohomology of moduli spaces} \label{va}

\subsection{} Vertex algebras are a mathematical formulation of two-dimensional chiral conformal field theories appearing in physics, see e.g. \cite{BPV}. They were discovered by Borcherds, and a little while after Beilinson and Drinfeld followed with the closely related notion of a chiral algebra. For an introduction to vertex algebras see \cite{FB}, and to chiral algebras \cite{BD} and \cite{FG}.  In this section we will introduce a class of vertex algebras discovered by Joyce in \cite{J}, whose underlying space is the moduli space of certain abelian or derived categories.

\subsection{Vertex algebras} \label{vadefn}
 A \textit{vertex algebra} consists of a vector space $V$, whose elements $v\in V$ are called \textit{states}, each of which also acts on $V$ as a \textit{field} $Y(v,z)$. By this we mean that there is a linear map
$$Y\ :\ V\ \longrightarrow\ \text{End}V[[z,z^{-1}]]$$
such that each $Y(v,z)v'\in V((z))$, and that these fields \textit{weakly commute}, meaning
$$(z-w)^n(Y(v,z)Y(v',w)-Y(v',w)Y(v,z))\ =\ 0$$
for $n\gg 0$.\footnote{The two summands, which are a priori elements of the different vector spaces $V((w))((z))$ and $V((z))((w))$, are viewed as subspaces of $V[[z^{\pm 1}, w^{\pm 1}]]$.} Two final pieces of data are a \textit{vacuum vector} $|0\rangle\in V$ whose field is $Y(|0\rangle,z)=\text{id}$ and such that every $Y(v,z)|0\rangle\in V[[z]]$ has constant coeffient $v$, and secondly an operator $\partial:V\to V$ annihilating the vacuum vector and such that $Y(\partial v,z)=\partial_z Y(v,z)$.

\subsubsection{} A \textit{super vertex algebra} consists of the same data, except now $V$ is a $\mathbf{Z}/2$ \textit{graded} vector space, $|0\rangle$ and $\partial$ have degrees zero and one, and weak commutativity is replaced by
$$(z-w)^n(Y(v,z)Y(v',w)-(-1)^{|v|\cdot|v'|}Y(v',w)Y(v,z))\ =\ 0.$$

\subsubsection{} A $\mathbf{Z}$ \textit{grading} on a vertex algebra is a $\mathbf{Z}$ grading on $V$ such that writing
$$Y(v,z)\ =\ \sum v_n z^{-n-1}$$
whenever $v$ is homogeneous of degree $|v|$ then the endomorphism $v_n$ has degree $|v|-n-1$. In other words, the above field has degree $|d|$ if we consider $z$ to have degree $-1$. We finally require that $|0\rangle$ has degree zero and $\partial$ has degree one.

\subsubsection{} The field map $Y$ can be rewritten as
$$Y\ :\ V\otimes V\ \longrightarrow\ V[[z,z^{-1}]].$$
We can define similarly a \textit{vertex coalgebra}, which consists of a vector space $V$ together with an operator $\partial:V\to V$, a map
$$Y^\vee\ :\ V\ \longrightarrow\ V\otimes V[[z,z^{-1}]].$$
and a covacuum element $V\to k$, satisfying dual axioms to above (see \cite{Hu} for a precise account). For our purposes all we need to know is that if $V$ is a vertex algebra its dual is a vertex coalgebra,  and taking contragredient duals
$$V\ \longrightarrow\ V^\vee$$
defines an equivalence from the category of $\mathbf{Z}$ graded vertex algebras with finite dimensional graded pieces to $\mathbf{Z}$ graded vertex coalgebras with finite dimensional graded pieces. In our examples $V$ will be the homology and $V^\vee$ the cohomology of a space. See appendix \ref{definitions} for the further notions of \text{super vertex (co)algebra} and \textit{(super) vertex bialgebra}.

\subsubsection{} \label{holvaexplanation} If all the fields only involve positive powers of $z$, a vertex algebra is called \textit{holomorphic}. It follows that the fields $Y(v,z)$ all commute and that 
$$v\cdot v'\ = \ \text{constant term of }Y(v,z)v'$$
defines an algebra structure on $V$, for which $\partial$ is a derivation. Moreover, $Y(v,z)=e^{z\partial}v$, so that the category of holomorphic vertex algebras is equivalent to that of algebras with a derivation.

\subsection{} Let $\mathcal{A}$ be $k$-linear abelian category, and $X$ its moduli space of objects. Then the additive structure on the category induces a map on the moduli space
$$\mathcal{A}\otimes\mathcal{A}\ \stackrel{\oplus}{\longrightarrow}\ \mathcal{A}\ \ \  \rightsquigarrow\ \ \  X\times X\ \stackrel{\oplus}{\longrightarrow}\ X$$
making $X$ into a commutative monoid space. Similarly, the action of $k^\times$ on objects induces 
$$\mathbf{G}_m\otimes\mathcal{A}\ \stackrel{act}{\longrightarrow}\ \mathcal{A}\ \ \  \rightsquigarrow\ \ \  \text{B}\mathbf{G}_m\times X\ \stackrel{act}{\longrightarrow}\ X$$
making $X$ into a module for the group space $\text{B}\mathbf{G}_m$.

\subsection{Joyce's construction: holomorphic case} \label{holva}  Let $X$ be a commutative monoid object in the category of spaces, with a compatible action of $\text{B}\mathbf{G}_m$. Then we get a (co)algebra structure on (co)homology
$$\text{H}^\cdot(X)\ \stackrel{\oplus^*}{\longrightarrow}\ \text{H}^\cdot(X)\otimes \text{H}^\cdot(X)$$
and taking the coefficient of a generator $t\in \text{H}^2(\text{B}\mathbf{G}_m)$ gives a (co)derivation
$$\text{H}^\cdot(X)\ \stackrel{act^*}{\longrightarrow}\ \text{H}^\cdot(\text{B}\mathbf{G}_m)\otimes \text{H}^\cdot(X)\ \stackrel{[t]\otimes \text{id}}{\longrightarrow}\ \text{H}^\cdot(X).$$
As discussed in section \ref{holvaexplanation}, this data is equivalent to that of a holomorphic vertex (co)algebra. Unpacking the equivalence, we arrive at the vertex (co)algebra structure 
$$Y^\vee\ :\ \text{H}^\cdot(X)\ \longrightarrow\ \text{H}^\cdot(X)\otimes \text{H}^\cdot(X)[t]\textcolor{white}{aaaaaaaaaaa}\alpha\ \longmapsto \ act^*_1(\oplus^*\alpha).$$
Here $act_1$ denotes the $\text{B}\mathbf{G}_m$ action on the first factor of $X\times X$.

\subsubsection{} Taking graded duals, we get that the homology $\text{H}_\cdot(X)=\text{H}^\cdot(X)^\vee$ is a holomorphic vertex algebra.

\subsubsection{}The general construction will be a Borcherds bicharacter twist of the above: this twisting takes in a holomorphic vertex algebra and bicharacter and produces a vertex algebra (see \cite{B}). In our case the construction is very explicit.

\subsection{Joyce's construction: general case} \label{general} Let $X$ be a commutative monoid object in the category of spaces, with a compatible action of $\text{B}\mathbf{G}_m$, and
$$\theta\ \in \ \text{Perf}(X\times X)$$
be a perfect complex which
\begin{enumerate}[label = \arabic*)]
\item Respects the monoidal structure on $X$: pulling back by $ (X\times X)\times X\to X\times X$ gives
$$(\oplus\times \text{id})^*\theta\ =\  \theta_{13}\oplus  \theta_{23}$$
where $\theta_{ij}$ is the pullback by the projection $X^3\to X^2$ onto the $ij$th factor, and similarly for $(\text{id}\times\oplus)^*\theta$.

\item Has weight $-1$ in the first factor and $+1$ in the second.\footnote{Given a group action $act:\text{B}\mathbf{G}_m\times Y\to Y$ of the group object $\text{B}\mathbf{G}_m$, the sheaf $\mathcal{F}\in \text{QCoh}(Y)$ or $D^b\text{QCoh}(Y)$ is said to have \textit{weight} $n$ if $act^*\mathcal{F}\ =\ \gamma_1^{\otimes n}\boxtimes \mathcal{F}$.}
\item Is \textit{symmetric} in the sense that $\sigma^*\theta = \theta^\vee[2n]$ where $\sigma:X\times X\to X\times X$ is the map swapping the two factors, and $n\in\mathbf{Z}$. So its Chern classes and rank are symmetric.
\item Its rank on each connected component is even.
\end{enumerate}
The most basic case of the construction is
\begin{prop}
In addition assume that $\theta$ is a vector bundle. Then 
$$Y^\vee\ :\ \text{H}^\cdot(X)\ \longrightarrow\ \text{H}^\cdot(X)\otimes \text{H}^\cdot(X)[t]\textcolor{white}{aaaaaaaaaaa}\alpha\ \longmapsto\ act^*_1(e(\theta)\cdot \oplus^*\alpha)$$
defines a vertex coalgebra structure.
\end{prop}

The only obstruction to repeating the above for general perfect complexes is that the Euler class $e(\theta)$ is not a priori defined. We can sidestep this issue by defining $act^*e(V)$ for a perfect complex $V$ of nonzero weight. Indeed, if $V$ were a vector bundle then by the splitting principle
$$act^*e(V)\ =\ (nt)^{\text{rk}V}\sum (nt)^{-k}c_k(V),$$
and we \textit{define} $act^*e(V)\in \text{H}^\cdot(X)((t))$ for a general perfect complex $V$ of weight $n$ by the above formula. Then

\begin{prop}
For $\theta$ a perfect complex as above, 
\begin{equation}\label{eqn1}
Y^\vee\ :\ \text{H}^\cdot(X)\ \longrightarrow\ \text{H}^\cdot(X)\otimes \text{H}^\cdot(X)((t))\textcolor{white}{aaaaaaaaaaa}\alpha\ \longmapsto\ act^*_1e(\theta) \cdot act^*_1\oplus^*\alpha
\end{equation}
defines a vertex coalgebra structure.
\end{prop}

Finally, the most general construction does not require that the rank of $\theta$ is even. To continue we need the notion of \textit{orientations}, which are functions
$$\varepsilon\ :\ \pi_0(X\times X)\ \longrightarrow\ \mathbf{Z}/2$$
satisfying axioms as explained in \cite{J}.  From these we get signs $\eta \in \text{End} \text{H}^\cdot(X\times X)$, acting as multiplication by $\pm1$ on the cohomology of each connected component $\text{H}^a(X_\alpha\times X_\beta)$, defined by
$$\eta_{\alpha,\beta}\ =\ (-1)^{a\cdot \text{rk}\theta_{\beta,\beta}}\varepsilon_{\alpha,\beta}.$$
There are potentially many choices, and when the rank of $\theta$ is even $\varepsilon=\eta=1$ is one such choice. Finally, we note that cohomology is $\mathbf{Z}$ graded by
$$\deg\text{H}^a(X_\alpha)\ =\ a\ +\ \text{rk}\theta_{\alpha,\alpha}$$ 
which in particular induces a $\mathbf{Z}/2$ grading. Then the foundational theorem in \cite{J} is
\begin{theorem}
For $\theta$ a perfect complex as above, but allowed to have any rank,  
\begin{equation}\label{eqn1}
Y^\vee\ :\ \text{H}^\cdot(X)\ \longrightarrow\ \text{H}^\cdot(X)\otimes \text{H}^\cdot(X)((t))\textcolor{white}{aaaaaaaaaaa}\alpha\ \longmapsto\ \eta\cdot act^*_1e(\theta) \cdot act^*_1\oplus^*\alpha
\end{equation}
defines a super vertex coalgebra structure.
\end{theorem}

\subsubsection{} We get analogous vertex algebra structures on homology $\text{H}_\cdot(X)=\text{H}^\cdot(X)^\vee$ by taking graded duals of the above constructions. This is the point of view taken in \cite{J}.

\subsubsection{} When the perfect complex and orientations are trivial $\theta=0$ and $\varepsilon=1$, we get back the holomorphic construction of section \ref{holva}.

\subsubsection{} If $\theta$ satisfies all conditions except symmetry then the doubled complex $\theta\oplus \sigma^*\theta^\vee$ satisfies all of them. This is closely related to considering the cotangent space $T^*X$, or when $\mathcal{A}$ is the category of representations of a quiver, to doubling the quiver.

\subsection{Braided vertex algebras}  There are many definitions for the notion of braided vertex algebras. The definition we will follow is \cite{J}, where they are called quantum vertex algebras, which in turn is based on earlier definitions in \cite{Li} and \cite{EK}.

\subsubsection{} A \textit{nonlocal vertex algebra} consists of the same data as a vertex algebra, with as before $\text{Y}(|0\rangle , z)=\text{id}$ and the constant term of $Y(v,z)|0\rangle$ being $v$, but now the fields are only required to be \textit{weakly associative}:
$$(z+w)^nY(Y(v,z)v',w)v''\ =\ (z+w)^nY(v,z+w)Y(v',w)v''$$
for $n\gg 0$.\footnote{The two sides are a priori elements of $V((z))((w))$ and $V((z+w))((w))$, which are viewed as subspaces of $V[[z^{\pm 1}, w^{\pm 1}]]$ by expanding negative powers of $z+w$ as a power series in $w/z$.} This is a strictly weaker property than weak commutativity. Nonlocal algebras are called \textit{field algebras} in Kac's terminology.

\subsubsection{} A (\textit{rational}) \textit{Yang-Baxter} operator on vector space $V$ is a linear map
$$S(z)\ :\ V\otimes V\ \longrightarrow\ V\otimes V ((z))$$
satisfying the Yang-Baxter equation:
$$S_{12}(z)\cdot S_{13}(z+w)\cdot S_{23}(w)\ =\ S_{23}(w)\cdot S_{13}(z+w)\cdot S_{12}(z).$$
\subsubsection{} \label{braided} A \textit{braided vertex algebra} is a nonlocal vertex algebra with a compatible Yang-Baxter operator. This means that
$$S(z)(|0\rangle \otimes v)\ =\ |0\rangle \otimes v, \ \ \ S(z)(v\otimes |0\rangle)\ =\ v\otimes |0\rangle,$$
that\footnote{As before, we expand powers of $z+w$ as a power series in $w/z$ and view the equation as taking place in $V[[z^{\pm 1}, w^{\pm 1}]]$.}
$$S(w)\cdot (Y(v,z)\otimes \text{id}_V)\ =\ (Y(v,z)\otimes\text{id}_V)\cdot S_{23}(w) \cdot S_{13}(z+w),$$
and finally that
\begin{equation}\label{eqn13}
Y(v,z)\cdot  S(w) v'\ =\ (-1)^{|v|\cdot |v'|} e^{z\partial} \cdot Y(v,-z)v'.
\end{equation}
When the Yang-Baxter operator is the identity we get back the notion of vertex algebra, because then equation (\ref{eqn13}) along with weak associativity imply the vertex algebra axiom.

\subsubsection{} By dualising the above definition one can define a \textit{nonlocal vertex coalgebra}, which has the same definition as a vertex coalgebra except the weak cocommutativity axiom is replaced by \textit{weak coassociativity}, obtained by dualising the weak associativity. Similarly, a \textit{braided vertex coalgebra} is a nonlocal vertex coalgebra with a Yang-Baxter operator as defined above, satisfying the dual of the compatibility relations in section \ref{braided}.

\subsubsection{} \label{bialgebra} If $V$ is both an algebra $(V,m,1)$ and a braided vertex coalgebra $(V,Y^\vee,|0\rangle^\vee,\partial,S(z))$, we can ask when these structures are compatible. We call it a \textit{braided vertex bialgebra} if the following commutes
\begin{equation}\label{diag15}
\begin{tikzcd}
V\otimes V\arrow[d,"m"]\arrow[r,"Y^\vee\otimes Y^\vee"]& (V\otimes V)\otimes (V\otimes V)((z^{-1})) \arrow[r," S_{23}(z)\sigma"]& (V\otimes V)\otimes (V\otimes V)((z^{-1})) \arrow[d,"m\otimes m"]\\
V\arrow[rr,"Y^\vee"]& &V\otimes V((z^{-1}))
\end{tikzcd}
\end{equation}
where $\sigma$ is the map swapping the second and third factors, if
$$Y^\vee(1)\ =\ 1\otimes 1,$$
if $|0\rangle^\vee m=|0\rangle^\vee\otimes |0\rangle^\vee$, and if $|0\rangle^\vee 1=\text{id}$. Notice the appearance of the Yang-Baxter matrix inside the compatibility diagram (\ref{diag15}). Thus this is \textit{not} the a analogue of the usual notion of quasitriangular bialgebra.

\subsection{Joyce's braided construction} \label{summary}  Let $X$ be a commutative monoid object in the category of spaces, with a compatible action of $\text{B}\mathbf{G}_m$, and $\theta \in  \text{Perf}(X\times X)$ be a perfect complex which
\begin{enumerate}[label = \arabic*)]
\item respects the monoidal structure on $X$, and
\item has weight $-1$ in the first factor and $+1$ in the second.
\end{enumerate}
Thus $\theta$ is not necessarily symmetric. Then 
\begin{theorem} \label{joycethm} 
The Yang-Baxter operator given by multiplication by
$$S(t)\ =\  act^*e(\theta)/act^* e(\sigma^*\theta) \ \in \ \text{H}^\cdot(X\times X)((t)) $$
along with
$$
Y^\vee\ :\ \text{H}^\cdot(X)\ \longrightarrow\ \text{H}^\cdot(X)\otimes \text{H}^\cdot(X)((t))\textcolor{white}{aaaaaaaaaaa}\alpha\ \longmapsto\ act^*_1e(\theta) \cdot act^*_1\oplus^*\alpha
$$
define the structure of a braided vertex coalgebra on $\text{H}^\cdot(X)$.
\end{theorem}

Notice that unlike in the classical construction, the above has no sign correction. However, this has the disadvantage that when $\theta$ \textit{is} symmetric $S(t)=(-1)^{\text{rk}\theta}$ is not the identity so the above does not reduce to a vertex algebra structure. There is some freedom still to choose signs, and we may define a \textit{braided orientation} to be a pair of functions
$$\varepsilon,\delta\ :\ \pi_0(X\times X)\ \longrightarrow\ \mathbf{Z}/2$$
such that the graded duals of
$$S(t)\ =\  \delta\cdot act^*e(\theta)/act^* e(\sigma^*\theta) $$
along with
$$
Y^\vee\ :\ \text{H}^\cdot(X)\ \longrightarrow\ \text{H}^\cdot(X)\otimes \text{H}^\cdot(X)((t))\textcolor{white}{aaaaaaaaaaa}\alpha\ \longmapsto\ \eta \cdot act^*_1e(\theta) \cdot act^*_1\oplus^*\alpha
$$
define a braided vertex coalgebra structure on $\text{H}^\cdot(X)$. It is easy to extract explicit relations $\varepsilon,\delta$ must satisfy in order to be braided orientations.

\subsection{Example} The typical example the moduli space of objects in an abelian or derived category $\mathcal{C}$, and where $\theta = \text{Ext}(\ ,\ )$ the perfect complex whose fibre above the pair of objects $(V,W)$ is $\text{Ext}(V,W)$. Unless the category is even Calabi-Yau the complex will not be symmetric and we will get a \textit{braided} vertex algebra structure.

\subsubsection{} \label{vectex} We give the simplest example, the abelian category $\mathcal{A}$ finite dimensional vector spaces. It is zero dimensional and Calabi-Yau, and its moduli space of objects is 
$$X\ =\ \coprod \text{BGL}_n.$$
A map into $X$ is uniquely specified by what the pullback is of the tautological vector bundle $\gamma$. Therefore the structure maps $\oplus$ and $act$ are \textit{defined} by requiring
$$\oplus^*\gamma\ =\ \gamma\boxplus \gamma,\ \ \ \text{ and } \ \ \  act^*\gamma\ = \ \gamma_1\boxtimes \gamma,$$
where $\gamma_1$ is the tautological line bundle on $\text{B}\mathbf{G}_m$. The perfect complex here is simply $\theta=\gamma^\vee\boxtimes \gamma$.

\subsubsection{} To describe the action of $\oplus$ and $act$ on cohomology, pick a maximal torus $\text{T}_n\subseteq \text{GL}_n$, so that the map $\text{BT}_n\to \text{BGL}_n$ identifies
$$\text{H}^\cdot(\text{BGL}_n)= k[t_1,...,t_n]^{\mathfrak{S}_n}\ \longrightarrow\ k[t_1,...,t_n]\ =\ \text{H}^\cdot(\text{BT}_n).$$
The maps $\oplus$ and $act$ can be lifted to $\text{BT}_n$ in a manner similar to the above, and they induce maps on cohomology 
$$\oplus^*\ :\ k[s_1,...,s_n,t_1,...,t_m]\ \stackrel{\sim}{\longrightarrow}\ k[s_1,...,s_n]\otimes k[t_1,...,t_m],$$
and
$$act^*\ :\ k[t_1,...,t_n]\ \longrightarrow\ k[t]\otimes k[t_1,...,t_n]$$
which sends $t_i\mapsto t\otimes t_i$. Taking symmetric group invariants then recovers these maps for $X$.

\subsubsection{} To explicitly describe the vertex algebra we get, it is easier instead to work with the derived category $\mathcal{C}$ of finite dimensional vector spaces, which has $\mathcal{C}^\heartsuit=\mathcal{A}$. In this case, 
$$X\ =\ \coprod \text{Perf}_n$$
where $\text{Perf}_n$ classifies perfect complexes of rank $n$. Again we have $\theta=\gamma^\vee\boxtimes\gamma$ where $\gamma$ is the tautological perfect complex on $X$.  It is not hard to show that $\text{H}^\cdot(X)\ \simeq\ V_{\mathbf{Z}}$ is the one dimensional lattice vertex algebra. Thus the vertex subalgebra corresponding to the abelian category has basis 
$$\{ b_{-n}^{a_n}\cdots b_{-1}^{a_1}|n\rangle\ :\ n\ge 0, \ a_n\ge 0\}.$$

\subsubsection{} In this example the complex $\theta$ was symmetric. Moving one dimension higher to one dimensional categories, we lose this and so begin encountering braided vertex algebras. These categories will be considered in the next section.

\subsection{} At the moment there is no satisfying explanation of Joyce's constructions at the level of chiral algebras. It would be interesting to relate these constructions to \cite{KV}.

\section{Examples}\label{description} 

\subsection{} This section gives an explicit description of the main geometric objects we are studying, the moduli spaces of objects in one dimensional abelian categories, and discusses their (braided) vertex algebra and CoHA structures. For the readers' convenience we summarise the accounts in \cite{KS} and \cite{He} respectively. 

\subsection{Quiver representations} Let $Q$ be a finite quiver. This is a finite set $|Q|$ and a finite collection of arrows $e:p\to q$ between elements of the set. A \textit{representation} of $Q$ is a vector space attached to each $q\in |Q|$ and associated to each arrow $e:p\to q$ a linear map between the vector space of $p$ and $q$. These arrange to form an abelian category $\mathcal{A}=\text{Rep}Q$, whose Grothendieck group $\mathbf{Z}^{|Q|}$ is labelled by the virtual dimensions of the vector spaces.

\subsubsection{} The moduli space of objects in $\mathcal{A}$ has connected components labelled by the \textit{multidimension} $\gamma\in \mathbf{N}^{|Q|}$ of the representation:
$$M_\gamma\ =\ \prod_{e:p\to q}\text{Hom}(k^{\gamma_p},k^{\gamma_q})/\prod_{q\in |Q|}\text{GL}(k^{\gamma_q}).$$
Indeed, having fixed a multidimension a representation is determined up to isomorphism by the values of the linear maps. Moreover an automorphism of a representation is simply an automorphism of each of the vector spaces. Denote $M_\gamma= V_\gamma/\text{G}_\gamma$. As it is a vector space, $V_\gamma$ is contractible and so $\text{H}^\cdot(M_\gamma)=\text{H}^\cdot(\text{BG}_\gamma)$ is a polynomial algebra.

\subsubsection{} In this paper we will be considering the braided vertex algebra structure arising from the perfect complex $\theta=\text{Hom}(-,-)$. This does not have a particularly explicit description, although if we take instead the symmetrised complex $\theta\oplus \sigma^*\theta^\vee$ on the moduli space $\widehat{M}$ of the \textit{derived} category $D^b(\text{Rep}Q)$, then it is proven in \cite{J}  that

\begin{prop}
The vertex algebra $\text{H}_\cdot(\widehat{M})$ is isomorphic to the lattice vertex algebra of the Grothendieck group of $D^b(\text{Rep}Q)$
$$\text{H}_\cdot(\widehat{M})\ \simeq\ V_{\mathbf{Z}^{|Q|}, \chi}$$
for a particular choice of orientations $\varepsilon$. Here $\chi=\text{rk}(\theta\oplus \sigma^*\theta^\vee)$ is the bilinear form of the lattice, given explicitly by
$$\chi(\gamma,\gamma')\ =\ \sum_{p,q\in|Q|} a_{p,q}\gamma_p\gamma'_q$$
where $a_{p,q}$ is the number of arrows from $p$ to $q$.
\end{prop}

\subsubsection{} The moduli stack of extensions has connected components labelled by pairs of multidimension vectors. Each is a global quotient stack
$$M_{\gamma,\gamma'}\ =\ V_{\gamma,\gamma'}/P_{\gamma,\gamma'}$$
where $V_{\gamma,\gamma'}\subseteq V_{\gamma+\gamma'}$ is the subspace of edge maps preserving $k^\gamma\subseteq k^{\gamma+\gamma'}$, and $P_{\gamma,\gamma'}\subseteq G_{\gamma+\gamma'}$ is the stabiliser that subspace. It is a parabolic subgroup. In the Hall correspondence
\begin{center}
\begin{tikzcd}
&M_{\gamma,\gamma'}\arrow[rd,"p"]\arrow[ld,"q",swap]&\\
M_\gamma\times M_{\gamma'}&&M_{\gamma+\gamma'}
\end{tikzcd}
\end{center} 
the proper map 
$$p\ :\ V_{\gamma,\gamma'}/P_{\gamma,\gamma'}\ \longrightarrow\ V_{\gamma+\gamma'}/P_{\gamma+\gamma'}$$
is induced by the inclusion maps $V_{\gamma,\gamma'}\to V_{\gamma+\gamma'}$ and $P_{\gamma,\gamma'}\to G_{\gamma+\gamma'}$, and the other map 
$$q\ :\ V_{\gamma,\gamma'}/P_{\gamma,\gamma'}\ \longrightarrow\ V_{\gamma}/P_{\gamma}\times V_{\gamma'}/P_{\gamma'}$$
is induced by the quotient maps $V_{\gamma,\gamma'}\to V_\gamma\times V_{\gamma'}$ and $P_{\gamma,\gamma'}\to G_\gamma\times G_{\gamma'}$. The latter is a $BU_{\gamma,\gamma'}$ bundle where $U_{\gamma,\gamma'}\subseteq P_{\gamma,\gamma'}$ is the unipotent radical, and so $q$ is an isomorphism on cohomology. There is an explicit description for the CoHA product discovered in \cite{KS}, and which we rederive in section \ref{explicit}.

\subsection{Coherent sheaves on a curve} Let $C$ be a smooth proper curve over an algebraically closed field. We consider the abelian category $\text{Coh}(C)$ of coherent sheaves. Its moduli space has connected components
$$\text{Coh}\ =\ \coprod \text{Coh}_r^d$$
labelled by the rank $r$ and degree $d$ of the coherent sheaf. 

\subsubsection{} There is a tautological coherent sheaf on $C\times \text{Coh}$, denoted $\mathcal{E}$, with Chern classes
$$c_i(\mathcal{E})\ =\ 1\otimes \alpha_i\ +\ \sum b_{k,i}\otimes \beta_{k,i}\ +\ \sigma\otimes \gamma_i,$$
where $1,b_1,...,b_{2g},\sigma$ is a basis of $\text{H}^\cdot(C)$. Then \cite{He} shows
\begin{prop} For positive rank $r>0$,
$$\text{H}^\cdot(\text{Coh}_r^d)\ =\ \text{Sym}(\text{W}_{C,d})$$
is the free supersymmetric algebra generated by the $\alpha_i,\beta_{k,i},\gamma_i$, which form a basis for the graded vector space $\text{W}_{C,d}$.
\end{prop}

\subsubsection{} In rank zero $r=0$, we first have  an isomorphism
$$\text{Coh}^1_0\ \stackrel{\sim}{\longrightarrow}\ C\times \text{B}\mathbf{G}_m.$$
It takes a degree one rank zero coherent sheaf $\mathcal{T}$ on $X\times C$  to $(\text{Supp}\mathcal{T},\pi_*\mathcal{T})$, where $\text{Supp}\mathcal{T}:X\to C$ and the pushfoward by $\pi:X\times C\to X$ gives a line bundle. That is to say, we have the short exact sequence of coherent sheaves on $\text{Coh}_0^1\times C$
$$0\ \longrightarrow\ \mathcal{O}\ \longrightarrow\ \mathcal{L}\ \longrightarrow\ \mathcal{T}\ \longrightarrow\ 0$$
where $\mathcal{T}$ is the tautological rank zero degree one coherent sheaf, and $\mathcal{L}=\mathcal{O}(-\text{Supp}\mathcal{T})$. Moreover, Heinloth shows 
\begin{prop} For rank zero $r=0$, 
$$\text{H}^\cdot(\text{Coh}_0^d)= \text{Sym}^d(\text{V}_C)$$
where $\text{V}_C$ is the graded vector space $\text{H}^\cdot(C\times \text{B}\mathbf{G}_m)$.
\end{prop}

\subsubsection{} The moduli space of extensions has connected components 
$$\text{Coh}_{r,r'}^{d,d'}$$
labelled by the rank and degree $(r,d),(r',d')$ of the subobject and quotient:
$$0\ \longrightarrow\ \mathcal{E}^d_{r}\ \longrightarrow\ \mathcal{E}_{r+r'}^{d+d'}\ \longrightarrow\ \mathcal{E}_{d'}^{r'}\ \longrightarrow\ 0.$$
In section \ref{explicit}, we will describe the CoHA product. To do so, it is useful to consider the stratification on $\text{Coh}$ given by the length $\ell$ of torsion subsheaf:
$$\text{Coh}^{d}_r\ = \ \coprod_{\ell\ge 0} \text{Coh}^{d,\ell}_r.$$
The closure relations of this stratification is $\overline{\text{Coh}^{d,\ell}_r}\ =\ \coprod_{\ell'\ge \ell}\text{Coh}^{d,\ell'}_r$. All strata are smooth. Sending a coherent sheaf to its torsion subsheaf and torsion-free quotient gives
$$\text{Coh}^{d,\ell}_r\ \longrightarrow\ \text{Coh}_0^\ell\times \text{Bun}_r^{d-\ell}.$$
This is a vector bundle of rank $r\ell$, the zero section being the direct sum map. It follows that the $\ell$th stratum has codimension $r\ell$ inside $\text{Coh}^{d}_r$.

\subsubsection{} The cohomology $\text{H}^\cdot(\text{Coh})$ is a braided vertex algebra, for $\theta=\text{Hom}(-,-)$. Explicitly, we have the Ext complex $\mathcal{E}\text{xt}$ living over $C\times \text{Coh}^2$, whose fibre over $(c,V,W)$ is the fibre above $c\in C$ of the complex of sheaves $\mathcal{E}\text{xt}_{\mathcal{O}_C}(V,W)$. We then define
$$\theta\ =\ \pi_*\mathcal{E}\text{xt}$$
where $\pi:C\times \text{Coh}^2\to \text{Coh}^2$ is the projection, which is a perfect complex because $\pi$ is smooth and proper. Thus this construction works for any smooth proper scheme $C$. The conditions in \ref{joycethm} for $\theta$ are inherited from $\mathcal{E}\text{xt}$. When $C$ is an elliptic curve the category of coherent sheaves is one Calabi-Yau, and consequently $\theta$ is also symmetric.

As before, this braided vertex algebra does not have a very explicit description, but one can show that the cohomology of the moduli space of $D^b(\text{Coh}(C))$ for the doubled complex $\theta\oplus \sigma^*\theta^\vee$ has the structure of a lattice vertex superalgebra, for the lattice $\text{H}^\cdot(C,\mathbf{Z})$ with bilinear form $\text{rk}(\theta\oplus \sigma^*\theta^\vee)$ computable by Grothendieck-Riemann-Roch, see e.g. \cite{Gr}.

\section{Euler classes}\label{euler} 

\subsection{} The Euler class of a vector bundle $E\to X$ is defined using the Gysin sequence
$$\cdots  \ \longrightarrow\ \text{H}^{\cdot -2\text{rk}E}(X)\ \stackrel{\cdot e(E)}{\longrightarrow}\ \text{H}^\cdot(X)\ \longrightarrow\ \text{H}^\cdot(E\setminus X)\ \longrightarrow\ \cdots$$
The first map is multiplication by a cohomology class on $X$, which is \textit{defined} to be the Euler class. The Gysin sequence is the long exact sequence on cohomology of a certain distinguished triangle of sheaves on $X$: the Gysin triangle
$$ i^!k \ \stackrel{}{\longrightarrow}\ i^*k\ \longrightarrow\ i^*j_*j^*k\ \stackrel{+1}{\longrightarrow}$$
Here $i:X\hookrightarrow E$ denotes the zero section and $j:E\setminus X\hookrightarrow E$ its open complement. 

\subsubsection{} Now assume $E\to X$ is no longer a vector bundle but something like a cone bundle or perfect complex in nonnegative degrees. More precisely, assume it has a section $i:X\hookrightarrow E$ which is a closed embedding. Then the Gysin triangle still makes sense, and taking cohomology gives
$$\cdots  \ \longrightarrow\ \text{H}^\cdot(X/E)\ \stackrel{\cdot e(E)}{\longrightarrow}\ \text{H}^\cdot(X)\ \longrightarrow\ \text{H}^\cdot(E\setminus X)\ \longrightarrow\ \cdots$$
where the left entry is now a \textit{bivariant} homology group (see appendix \ref{biv} for a review).  The first map is now right multiplication by a \textit{bivariant} homology class $e(E)\in \text{H}^0(E/X)$, which we \textit{define} to be the bivariant Euler class.

\subsubsection{} \label{bivarianteulerclass} Let $E\to X$ be a map admitting a section $i:X\hookrightarrow E$ which is a closed embedding. Because $i$ is proper it induces a pushforward map on bivariant homology
$$i_* \ : \ \text{H}^\cdot(X/X)\ \longrightarrow\ \text{H}^\cdot(E/X).$$
The \textbf{bivariant Euler class} of $E$ is then defined to be the image of $1\in \text{H}^0(X)$ under this map, and is denoted by $e(E)\in \text{H}^0(E/X)$.

\begin{lem}
If $f:Y\to X$ is any map, $f^*e(E)=e(f^*E)$.
\end{lem}
\begin{proof}This follows by applying the projection formula to the pullback
\begin{center}
\begin{tikzcd}
Y\arrow[r]\arrow[d]& X\arrow[d]\\
f^*E\arrow[r]& E
\end{tikzcd}
\end{center}
\end{proof}

If $E$ is a vector bundle then $\text{H}^\cdot(E/X)=\text{H}^{\cdot+2\text{rk}E}(X)$ and so the bivariant Euler class defines a degree $2\text{rk}E$ cohomology class on $X$, the (non-bivariant) Euler class of $E$. Thus bivariant Euler classes extend the standard notion of Euler class. More precisely, multiplication by the fundamental class gives an isomorphism
$$1_{X/E}\cdot \ :\ \text{H}^\cdot(E/X)\ \stackrel{\sim}{\longrightarrow}\ \text{H}^{\cdot+2\text{rk}E}(X)$$
and $1_{X/E}\cdot e(E)=e(E)$, where the left is the bivariant Euler class and the right is the Euler class of a vector bundle, viewed as a cohomology class on the base.

\subsubsection{} If $i:X\hookrightarrow E$ is a closed embedding which does \textit{not} admit a cosection, we cannot define the Euler class $e(E)$. However, the first map in the Gysin sequence still exists, and we will nevertheless denote it by
$$\cdot e(E)\ :\ \text{H}^\cdot(X/E)\ \longrightarrow\ \text{H}^\cdot(X).$$
When a cosection exists, it is right multiplication by the class $e(E)$.

\subsection{} The prototypical example is of the tautological vector bundle $\gamma=\mathbf{A}^1/\mathbf{G}_m$ over $\text{B}\mathbf{G}_m$. Its Gysin sequence
$$\cdots  \ \longrightarrow\ \text{H}^{\cdot -2}(\text{B}\mathbf{G}_m)\ \stackrel{\cdot e(\gamma)}{\longrightarrow}\ \text{H}^\cdot(\text{B}\mathbf{G}_m)\ \longrightarrow\ \text{H}^\cdot(\mathbf{G}_m/\mathbf{G}_m)\ \longrightarrow\ \cdots$$
can be explicitly written as
$$\cdots  \ \longrightarrow\ k[t]\ \stackrel{\cdot e(\gamma)}{\longrightarrow}\ k[t]\ \longrightarrow\ k\ \longrightarrow\ \cdots$$
Thus the Euler class of $\gamma$ is $t$, rescaling if necessary. A more delicate analysis, e.g. using integral $\ell$-adic cohomology, will show that $e(\gamma^{\otimes n}) = n\cdot e(\gamma)$.

\subsubsection{} More generally, if $\text{G}$ is any reductive group over $\mathbf{C}$, a representation $V$ defines a vector bundle $\mathcal{V}=V/\text{G}$ over $\text{BG}$. Pulling back to the classifying space of a maximal torus $\text{BT}\to\text{BG}$ gives $\text{H}^\cdot(\text{BG}) = \text{H}^\cdot(\text{BT})^\text{W}$. It follows from the above that the Euler class of the vector bundle is
$$e(\mathcal{V})\ =\  \prod_{\lambda\in P}\lambda^{\dim V_\lambda}.$$
as an element of $\text{H}^\cdot(\text{BT})^\text{W} =\mathbf{C}[\mathfrak{t}^*]^\text{W}$. Here $P$ is the weight lattice of $\text{G}$. In particular, it vanishes if and only if zero is a weight of $V$.

\subsection{} \label{whitney}The \textit{Whitney sum formula} says that under a short exact sequence of vector bundles
\begin{equation}\label{ses}
0 \ \longrightarrow\ E_1 \ \longrightarrow\ E\ \longrightarrow \ E_2\ \longrightarrow \ 0
\end{equation}
the Euler class is multiplicative: 
$$e(E)\ =\ e(E_1)\cdot e(E_2).$$ 
We now present a proof which after tweaking will generalise to the bivariant situation. The extension of vector bundles (\ref{ses}) induces a fibre product of spaces
\begin{center}
\begin{tikzcd}
E_1\arrow[r]\arrow[d]&E\arrow[d]\\
X\arrow[r,"i_2"]& E_2
\end{tikzcd}
\end{center}
and the pullback map an bivariant homology gives an isomorphism
$$i_2^*\ :\ \text{H}^\cdot(E/E_2)\ \stackrel{\sim}{\longrightarrow}\ \text{H}^\cdot(E_1/X).$$
Indeed, up to a shift this map is just the pullback $\text{H}^\cdot(E)\stackrel{\sim}{\to} \text{H}^\cdot(E_1)$. Under this isomorphism the bivariant Euler class $e(E_1)\in \text{H}^0(E_1/X)$ corresponds to 
$$\underline{e}(E_1)\ \in\  \text{H}^0(E/E_2).$$

\begin{prop}
$e(E)=\underline{e}(E_1)\cdot e(E_2)$ as elements of $\text{H}^\cdot(E/X)$.
\end{prop}
\begin{proof}
Apply the projection formula to the diagram
\begin{center}
\begin{tikzcd}
E_1\arrow[r,"a"]\arrow[d]&E\arrow[d]& \\
X\arrow[r,"i_2"]& E_2\arrow[r]& X
\end{tikzcd}
\end{center}
which says that as elements of $\text{H}^\cdot(E/X)$,
$$a_*(i_{2}^*\underline{e}(E_1)\cdot 1_X)\ =\ \underline{e}(E_1)\cdot i_{2*}1_X.$$
The left side is $a_*(e(E_1))=e(E)$, and the right side is $\underline{e}(E_1)\cdot e(E_2)$.
\end{proof}

\subsubsection{} \label{whitneysingular} The singular analogue of a vector bundle is a perfect complex concentrated in nonnegative degrees, so that the zero section is a closed embedding. Take a distinguished triangle of such
\begin{equation}\label{perf}
E_1 \ \longrightarrow\ E\ \longrightarrow \ E_2\ \stackrel{+1}{\longrightarrow}
\end{equation}
giving as before pullback square of derived stacks
\begin{center}
\begin{tikzcd}
E_1\arrow[r]\arrow[d]&E\arrow[d]\\
X\arrow[r,"i_2"]& E_2
\end{tikzcd}
\end{center}
If the pullback 
$$i_2^*\ :\ \text{H}^\cdot(E/E_2)\ \longrightarrow\ \text{H}^\cdot(E_1/X).$$
were an isomorphism, we could repeat the previous argument verbatim to get a Whitney sum formula for perfect complexes. But this is almost never the case, and the Whitney sum formula is only true if we work \textit{equivariantly}.

\subsubsection{} Let $X$ be a space with a trivial	 action of $\mathbf{G}_m$. Take a distinguished triangle of perfect complexes on $X/\mathbf{G}_m$, concentrated in nonnegative degrees
\begin{equation}\label{perf}
E_1 \ \longrightarrow\ E\ \longrightarrow \ E_2\ \stackrel{+1}{\longrightarrow}
\end{equation}
So each is a $\mathbf{G}_m$ equivariant perfect complex on $X$. Finally assume that the $\mathbf{G}_m$ weights are all nonzero. Then the pullback map is very close to being an isomorphism: it is an isomorphism \textit{modulo torsion}
$$i_2^*\ :\ \text{H}^\cdot_{\mathbf{G}_m}(E/E_2)_{loc}\ \stackrel{\sim}{\longrightarrow}\ \text{H}^\cdot_{\mathbf{G}_m}(E_1/X)_{loc},$$
and so $e(E_1)$ corresponds to an element $\underline{e}(E_1)\in \text{H}^0_{\mathbf{G}_m}(E/E_2)_{loc}$. Here the subscript $loc$ means the localisation of a $\text{H}^\cdot(\text{B}\mathbf{G}_m)=k[t]$ at its unique prime ideal, i.e. invert $t$. The \textbf{Whitney sum formula} for perfect complexes is then

\begin{prop}
$e(E)=\underline{e}(E_1)\cdot e(E_2)$ as elements of $\text{H}^\cdot_{\mathbf{G}_m}(E/X)_{loc}$.
\end{prop}
\begin{proof}
Assuming that $i_2^*:\text{H}^\cdot_{\mathbf{G}_m}(E/E_2)\to \text{H}^\cdot_{\mathbf{G}_m}(E_1/X)$ is an isomorphism modulo $\text{H}^\cdot(\text{B}\mathbf{G}_m)$ torsion, the argument in (\ref{whitney}) can be repeated verbatim.

To show this claim, consider the pullback diagrams, consisting of perfect complexes over $X$ concentrated in nonnegative degrees
\begin{equation}\label{diag1}
\begin{tikzcd}
E_2[-1]\arrow[r,"\overline{\iota}"]\arrow[d]& E_1\arrow[r]\arrow[d]&X\arrow[d,"i_2"]\\
X\arrow[r,"i"]& E\arrow[r]& E_2
\end{tikzcd}
\end{equation}
Then the following commutes by axiom $A_{23}$ of bivariant homology:
\begin{center}
\begin{tikzcd}
\text{H}^\cdot_{\mathbf{G}_m}(E/E_2)\arrow[r,"i_2^*"]& \text{H}^\cdot_{\mathbf{G}_m}(E_1/X)\\
\text{H}^\cdot_{\mathbf{G}_m}(X/E_2)\arrow[u,"i_*"]\arrow[r,"i_2^*"]& \text{H}^\cdot_{\mathbf{G}_m}(E_2[-1]/X)\arrow[u,"\overline{\iota}_*"]
\end{tikzcd}
\end{center}
We wish to show the upper arrow is an isomorphism modulo torsion. Because of our assumptions on the complexes, $i,\overline{\iota}$ are closed embeddings on whose complement $\mathbf{G}_m$ acts freely, so the vertical maps are isomorphisms modulo torsion. Turn now to the bottom horizontal arrow. The classical truncation of the outer edges of (\ref{diag1}) is
\begin{center}
\begin{tikzcd}
X\arrow[rr]\arrow[d]&\textcolor{white}{aaa}&X\arrow[d,"i_{2}"]\\
X\arrow[rr]&& E_2
\end{tikzcd}
\end{center}
and writing out the definition of pullback, we see that
$$\text{H}^\cdot_{\mathbf{G}_m}(X/E_2)\ \longrightarrow\ \text{H}^\cdot_{\mathbf{G}_m}(E_2[-1]/X)\ =\ \text{H}^\cdot_{\mathbf{G}_m}(X/X)$$ 
agrees with $i_{2*}$, which is an isomorphism modulo torsion for the same reason as $i_*,\overline{\iota}_*$.
\end{proof}

\subsection{}\label{twoterm} The simplest example after vector bundles is when the perfect complex
$$E\ =\ (E_0\ \to \ E_1)$$
is a two-term complex of vector bundles, each with nonzero $\mathbf{G}_m$ weights. In this case the above can be made more explicit. It fits into a distinguished triangle
$$E\ \longrightarrow\ E_0\ \longrightarrow\ E_1\ \stackrel{+1}{\longrightarrow}$$
and thus its total space into a pullback diagram of derived stacks
\begin{center}
\begin{tikzcd}
E\arrow[r]\arrow[d]&X\arrow[d,"i_1"]\\
E_0\arrow[r,"f"]& E_1
\end{tikzcd}
\end{center}
Using the Whitney sum formula we can understand the bivariant Euler class $e(E)$ in terms of the usual Euler classes $e(E_i)$. First, since the product on bivariant homology is skew commutative, the following commutes
\begin{center}
\begin{tikzcd}
\text{H}^\cdot_{\mathbf{G}_m}(X/E)_{loc}\arrow[r,"\cdot e(V)"] \arrow[d,"f^*1_{X/E_1}",swap] & \text{H}^\cdot_{\mathbf{G}_m}(X)_{loc}\arrow[d,"1_{X/E_1}"]\\
\text{H}^{\cdot +2d_1}_{\mathbf{G}_m}(X/E_0)_{loc}\arrow[r,"\cdot \underline{e}(V)"] &\text{H}^{\cdot +2d_1}_{\mathbf{G}_m}(X/E_1)_{loc}
\end{tikzcd}
\end{center}
where $d_i$ denotes the rank of $E_i$. Applying the Whitney sum formula, this becomes
\begin{center}
\begin{tikzcd}
\text{H}^\cdot_{\mathbf{G}_m}(X/E)_{loc}\arrow[rr,"\cdot e(E)"] \arrow[d,"f^*1_{X/E_1}",swap]& & \text{H}^\cdot_{\mathbf{G}_m}(X)_{loc}\arrow[d,"1_{X/E_1}"]\\
\text{H}^{\cdot +2d_1}_{\mathbf{G}_m}(X/E_0)_{loc}\arrow[r,"\cdot e(E_0)"]&\text{H}^{\cdot +2d_1}_{\mathbf{G}_m}(X)_{loc}  &\text{H}^{\cdot +2d_1}_{\mathbf{G}_m}(X/E_1)_{loc} \arrow[l,"\cdot e(E_1)",swap]
\end{tikzcd}
\end{center}
But of course the bivariant Euler classes of $E_i$ are essentially the same thing as the usual Euler classes of $E_i$, so this is

\begin{center}
\begin{tikzcd}
\text{H}^\cdot_{\mathbf{G}_m}(X/E)_{loc}\arrow[rr,"\cdot e(E)"] \arrow[d,"f^*1_{X/E_1}",swap]& & \text{H}^\cdot_{\mathbf{G}_m}(X)_{loc}\arrow[d,"1_{X/E_1}"]\\
\text{H}^{\cdot +2d_1}_{\mathbf{G}_m}(X/E_0)_{loc}& &\text{H}^{\cdot +2d_1}_{\mathbf{G}_m}(X/E_1)_{loc} \\
\text{H}^{\cdot +2d_1-2d_0}_{\mathbf{G}_m}(X)_{loc}\arrow[u,"1_{X/E_0}"]\arrow[r," e(E_0)"] &  \text{H}^{\cdot +2d_1}_{\mathbf{G}_m}(X)_{loc}& \text{H}^\cdot_{\mathbf{G}_m}(X)_{loc}\arrow[l," e(E_1)",swap]\arrow[u,"1_{X/E_1}",swap]
\end{tikzcd}
\end{center}
All arrows here are isomorphisms. Thus in conclusion we have that
$$e(E)\ =\ e(E_0)/e(E_1)$$
in the sense that the following diagram commutes
\begin{center}
\begin{tikzcd}
\text{H}^\cdot_{\mathbf{G}_m}(X/E)_{loc}\arrow[rr,"\cdot e(E)"] \arrow[d,"1_{E/X}",swap]& & \text{H}^\cdot_{\mathbf{G}_m}(X)_{loc}\\
\text{H}^{\cdot +2d_1-2d_0}_{\mathbf{G}_m}(X)_{loc}\arrow[r," e(E_0)"] &  \text{H}^{\cdot +2d_1}_{\mathbf{G}_m}(X)_{loc}& \text{H}^\cdot_{\mathbf{G}_m}(X)_{loc}\arrow[l," e(E_1)",swap]\arrow[u,equals,swap]
\end{tikzcd}
\end{center}

\subsubsection{} More generally, let $E$ is a perfect complex concentrated in degrees $0$ and $1$, equivalently $E\to X/\mathbf{G}_m$ is quasismooth and has a fundamental class $1_{E/X}$. Also assume its $\mathbf{G}_m$ weights are nonzero. Then we get an actual localised cohomology class, also denoted $e(E)\in \text{H}^{2\text{rk}E}_{\mathbf{G}_m}(X)_{loc}$, by
\begin{center}
\begin{tikzcd}
\text{H}^\cdot_{\mathbf{G}_m}(X/E)_{loc}\arrow[r,"\cdot e(E)"] \arrow[d,"1_{E/X}",swap] & \ \text{H}^\cdot_{\mathbf{G}_m}(X)_{loc}\\
\text{H}^{\cdot -2 \text{rk}E}_{\mathbf{G}_m}(X)_{loc}\arrow[r," e(E)"] &\  \text{H}^\cdot_{\mathbf{G}_m}(X)_{loc}\arrow[u,equals,swap]
\end{tikzcd}
\end{center}

\subsubsection{} This can help us understand the bivariant Euler class better. When $E=E_0$ is a vector bundle the above is
\begin{center}
\begin{tikzcd}
\text{H}^{\cdot -2 \text{rk}E}_{\mathbf{G}_m}(X)_{loc}\arrow[r,"e(E_0)"] \arrow[d,"1",swap] & \ \text{H}^\cdot_{\mathbf{G}_m}(X)_{loc}\\
\text{H}^{\cdot -2 \text{rk}E}_{\mathbf{G}_m}(X)_{loc}\arrow[r," e(E_0)"] &\  \text{H}^\cdot_{\mathbf{G}_m}(X)_{loc}\arrow[u,equals,swap]
\end{tikzcd}
\end{center} 
but when $E=E_1[1]$ is a vector bundle concentrated in degree one, 
\begin{center}
\begin{tikzcd}
\text{H}^\cdot_{\mathbf{G}_m}(X)_{loc}\arrow[r,"1"] \arrow[d,"e(E_1)",swap] & \ \text{H}^\cdot_{\mathbf{G}_m}(X)_{loc}\\
\text{H}^{\cdot -2 \text{rk}E}_{\mathbf{G}_m}(X)_{loc}\arrow[r,"1/e(E_1)"] &\  \text{H}^\cdot_{\mathbf{G}_m}(X)_{loc}\arrow[u,equals,swap]
\end{tikzcd}
\end{center} 
Indeed, viewed as a map of sheaves $i^!k\to i^*k$ multiplication by the bivariant Euler class of $E_1$ is the identity because our sheaf theories are blind to the derived structure (and $i_{cl}$ is the identity). Thus to see derived structure and thereby obtain desirable formulas like $e(E)=e(E_0)/e(E_1)$, one really has to use fundamental classes.

Said another way, even though in the second example there is an obvious isomorphism $\text{H}^\cdot_{\mathbf{G}_m}(X/E)_{loc}\simeq \text{H}^\cdot_{\mathbf{G}_m}(X)_{loc}$ given by $i^!k\simeq i^*k$, this is not the correct one, not least because it cannot be extended to more general $E$. Multiplication by the fundamental class $\text{H}^{\cdot}_{\mathbf{G}_m}(X/E)_{loc}\stackrel{\sim}{\to} \text{H}^{\cdot-2\text{rk}E}_{\mathbf{G}_m}(X)_{loc}$ is the correct choice. This is a genuinely equivariant phenomenon, since it may not be an isomorphism before localising.

\section{Relative case} \label{note}

\subsection{} For a space $X$ acted on by a group $\text{G}$, the equivariant cohomology $\text{H}^\cdot_\text{G}(X,\mathcal{F})$ of any sheaf is a module over $\text{H}^\cdot(\text{BG})$. When is this a torsion module?

\subsubsection{} The most basic example is when $X$ is a finite dimensional dg scheme acted on \textit{freely} by $\text{G}$, because in that case

\begin{prop} \label{schtors}
All elements of $\text{H}^\cdot(\text{BG})\subseteq \text{H}^\cdot_{\text{G}}(X)$ in nonzero degree are nilpotent.
\end{prop}
\begin{proof}
As $(X/\text{G})_{cl}=X_{cl}/\text{G}$ it is enough to prove this when $X$ is a classical scheme. As the action is free the stabiliser groups of $X/\text{G}$ are trivial, and corollary $8.4.2$ of \cite{O} implies it is a Deligne-Mumford stack. It follows from $4.39$ of \cite{E} that its cohomology $\text{H}^\cdot_{\text{G}}(X)$ vanishes in degrees above twice its dimension.
\end{proof}

In particular, if $\text{G}$ is a positive dimensional reductive group, so that $\text{H}^\cdot(\text{BG})$ is a polynomial algebra, then it is well known (see for instance $7.2.4$ of \cite{DGai}) that

\begin{cor} \label{schtorscor}
$\text{H}^\cdot_\text{G}(X,\mathcal{F})$ is a torsion $\text{H}^\cdot(\text{BG})$ module for all $\mathcal{F}$.
\end{cor}

After setting up the preliminary technical details, section \ref{torscomp} will show using the above that $\text{H}^\cdot_{\text{G}}(X,\mathcal{F})$ is torsion whenever $\mathcal{F}$ ``comes from'' the free locus of $X$, for $X$ a dg scheme acted on by $\text{G}$. Then section \ref{torsrel} will generalise this to the relative situation when $X$ is not a dg scheme, but instead an Artin stack over a base whose fibres are dg schemes.

\subsection{Sheaves of algebras} A \textit{sheaf of algebras} on $X$ is an monoid object in $\text{Sh}(X)$.  That is, it is $\mathcal{A}\in\text{Sh}(X)$ along with a composition map $m:\mathcal{A}\otimes\mathcal{A}\to\mathcal{A}$ which is associative, and a unit $e:k_X\to \mathcal{A}$, which are compatible. These form a category denoted $\text{Alg}(\text{Sh}(X))$. 

Every sheaf is canonically a sheaf of modules over $k_X$, so $\text{Sh}(X)=\text{Sh}_{k_X}(X)$.

\subsubsection{} \label{adj} The direct and inverse image functors on sheaves give adjoint functors
$$f^*\ :\ \text{Alg}(\text{Sh}(Y))\ \longrightarrow\ \text{Alg}(\text{Sh}(X)), \textcolor{white}{aaaaaaa} f_*\ :\ \text{Alg}(\text{Sh}(X))\ \longrightarrow\ \text{Alg}(\text{Sh}(Y)).$$
Explicitly, if $\mathcal{B}$ is a sheaf of algebras on $Y$, $f^*\mathcal{B}$ is a sheaf of algebras via
$$f^*\mathcal{B}\otimes f^*\mathcal{B}\ =\ f^*(\mathcal{B}\otimes \mathcal{B})\ \stackrel{f^*m}{\longrightarrow}\ f^*\mathcal{B}$$
and if $\mathcal{A}$ is a sheaf of algebras on $X$, $f_*\mathcal{A}$ is a sheaf of algebras via
$$f_*\mathcal{A}\otimes f_*\mathcal{A}\ \longrightarrow f_*(\mathcal{A}\otimes\mathcal{A})\ \stackrel{f_*m}{\longrightarrow}\ f_*\mathcal{A}.$$

\subsubsection{} A sheaf of algebras on a point is the same thing as a dg algebra $A\in\text{Sh}(\text{pt})=\text{Vect}$. Pulling back by the projection $p:X\to \text{pt}$ gives $A_X=p^*A$, the \textit{constant} sheaf of algebras on $X$ with value $A$.

\subsubsection{} The algebra structure on $\text{H}^\cdot(X)=p_*k_X$ is precisely the cup product structure. The algebra structure on $f_*k_X$ should be thought of as the cup product on the fibres of $f$.

\subsection{Sheaves of modules} A \textit{sheaf of modules} for a sheaf of algebras $\mathcal{A}$ on $X$ is simply a module for $\mathcal{A}$, viewed as a monoid in the symmetric monoidal category $\text{Sh}(X)$. Explictly, these are sheaves $\mathcal{M}$ equipped with an action map $a : \mathcal{A}\otimes\mathcal{M}\to \mathcal{M}$ which satisfies various compatibility relations with $m$ and the unit $e$. Write $\text{Sh}_\mathcal{A}(X)$ for the category of these.

\subsubsection{} Replacing $m$ by $a$ in section (\ref{adj}), we get functors
\begin{equation}\label{eqn3}
f^*\ :\ \text{Sh}_\mathcal{B}(Y)\ \longrightarrow\ \text{Sh}_{f^*\mathcal{B}}(X), \textcolor{white}{aaaaaaa} f_*\ :\ \text{Sh}_\mathcal{A}(X)\ \longrightarrow\ \text{Sh}_{f_*\mathcal{A}}(Y).
\end{equation}
Here $\mathcal{A}$ and $\mathcal{B}$ are respectively sheaves of algebras on $X$ and $Y$.

\subsubsection{} For $\mathcal{A}\to \mathcal{A}'$ a map of sheaves of algebras on $X$, there are adjoint functors
$$\text{Res}^{\mathcal{A}'}_{\mathcal{A}}\ : \ \text{Sh}_{\mathcal{A}'}(X)\ \longrightarrow\ \text{Sh}_{\mathcal{A}}(X), \textcolor{white}{aaaaaaa}  \text{Ind}_{\mathcal{A}}^{\mathcal{A}'}\ :\ \text{Sh}_{\mathcal{A}}(X)\ \longrightarrow\ \text{Sh}_{\mathcal{A}'}(X)$$
called \textit{restriction} and \textit{induction} respectively, given by  viewing an $\mathcal{A}'$ module as an $\mathcal{A}$ module, and $\mathcal{A}'\otimes_{\mathcal{A}}\  $. They are compatible in the obvious way with $f^*$ and $f_*$ in (\ref{eqn3}).

\subsubsection{} Let $\mathcal{A}$ and $\mathcal{B}$ be respectively sheaves of algebras on $X$ and $Y$. Then there are maps of sheaves of algebras
$$\mathcal{B}\ \longrightarrow\ f_*f^*\mathcal{B}, \textcolor{white}{aaaaaaa}  \ f^*f_*\mathcal{A}\ \longrightarrow\ \mathcal{A}.$$
Restricting and inducing along these maps gives functors
$$\check{f}_*\ :\ \text{Sh}_{f^*\mathcal{B}}(X)\ \stackrel{f_*}{\longrightarrow}\ \text{Sh}_{f_*f^*\mathcal{B}}(Y)\ \stackrel{\text{Res}}{\longrightarrow}\ \text{Sh}_{\mathcal{B}}(Y), \textcolor{white}{aa} \check{f}^*\ :\ \text{Sh}_{f_*\mathcal{A}}(Y)\ \stackrel{f^*}{\longrightarrow}\ \text{Sh}_{f^*f_*\mathcal{A}}(X)\  \stackrel{\text{Ind}}{\longrightarrow}\ \text{Sh}_{\mathcal{A}}(X).$$
If we forget the module structure $\check{f}_*$ acts as $f_*$, so sometimes we use the same symbol $f_*$ to refer to both. We avoid doing this for $\check{f}^*$; if we forget the module structure it acts as $\mathcal{A}\otimes_{f^*f_*\mathcal{A}}f^*$ rather than $f^*$.

\begin{lem}
These two maps are left and right adjoint to the maps in $(\ref{eqn3})$, respectively. 
\end{lem}
\begin{proof}
We wish to show
\begin{equation}\label{eqn2}
\text{Hom}_{\text{Sh}_\mathcal{B}(Y)}(\mathcal{N},f_*\mathcal{M})\ \simeq\ \text{Hom}_{\text{Sh}_{f^*\mathcal{B}}(X)}(f^*\mathcal{N},\mathcal{M})
\end{equation}
and
\begin{equation}\label{eqn4}
\text{Hom}_{\text{Sh}_{f_*\mathcal{A}}(Y)}(\mathcal{N},f_*\mathcal{M})\ \simeq\ \text{Hom}_{\text{Sh}_{\mathcal{A}}(X)}(f^*\mathcal{N}\otimes_{f^*f_*\mathcal{A}}\mathcal{A},\mathcal{M})\ =\ \text{Hom}_{\text{Sh}_{f^*f_*\mathcal{A}}(X)}(f^*\mathcal{N},\mathcal{M}).
\end{equation}
There are forgetful maps from the left and right sides of (\ref{eqn2}) and (\ref{eqn4}) into the left and right of
$$\text{Hom}_{\text{Sh}(Y)}(\mathcal{N},f_*\mathcal{M})\ \simeq\ \text{Hom}_{\text{Sh}(X)}(f^*\mathcal{N},\mathcal{M})$$
and it is easy to show that one can lift this isomorphism to isomorphisms (\ref{eqn2}) and (\ref{eqn4}).
\end{proof}

\subsubsection{} We now examine what structures the adjoint pair $f^*$ and $\check{f}_*$ have on the level of cohomology. The functor $\check{f}_*$ can be described more explicitly as follows: let $\mathcal{M}$ be a sheaf of $f^*\mathcal{B}$ modules on $X$, then the $\mathcal{B}$ module structure on $f_*\mathcal{M}$ is
\begin{center}
\begin{tikzcd}
\mathcal{B}\otimes f_*\mathcal{M}\arrow[r]\arrow[d] & f_*\mathcal{M}\\
f_*f^*\mathcal{B}\otimes f_*\mathcal{M}\arrow[r]&f_*\mathcal{M}\arrow[u,equals]
\end{tikzcd}
\end{center}
where the bottom arrow is given by the $f^*\mathcal{B}$ module structure on $\mathcal{M}$. Taking cohomology now gives the commuting diagram
\begin{center}
\begin{tikzcd}
\text{H}^\cdot(\mathcal{B})\otimes \text{H}^\cdot(f_*\mathcal{M})\arrow[r]\arrow[d] & \text{H}^\cdot(\mathcal{B}\otimes f_*\mathcal{M})\arrow[r]\arrow[d] & \text{H}^\cdot(f_*\mathcal{M})\\
\text{H}^\cdot(f_*f^*\mathcal{B})\otimes \text{H}^\cdot(f_*\mathcal{M})\arrow[r]& \text{H}^\cdot(f_*f^*\mathcal{B}\otimes f_*\mathcal{M})\arrow[r]&\text{H}^\cdot(f_*\mathcal{M})\arrow[u,equals]
\end{tikzcd}
\end{center}
Similarly, if $\mathcal{N}$ is a sheaf of $\mathcal{B}$ modules on $Y$, then we have
\begin{center}
\begin{tikzcd}
\mathcal{B}\otimes \mathcal{N}\arrow[r]\arrow[d] & \mathcal{N}\arrow[d]\\
f_*f^*\mathcal{B}\otimes f_*f^*\mathcal{N}\arrow[r]&f_*f^*\mathcal{N}
\end{tikzcd}
\end{center}
Taking cohomology as before gives
\begin{lem} \label{torslemma}If $\mathcal{M}$ is a sheaf of $f^*\mathcal{B}$ modules on $X$, the following commutes
\begin{equation}\label{diag7}
\begin{tikzcd}
\text{H}^\cdot(Y,\mathcal{B})\otimes \text{H}^\cdot(Y,f_*\mathcal{M})\arrow[r]\arrow[d] & \text{H}^\cdot(Y,f_*\mathcal{M})\arrow[d,equals]\\
\text{H}^\cdot(X,f^*\mathcal{B})\otimes \text{H}^\cdot(X,\mathcal{M})\arrow[r] & \text{H}^\cdot(X,\mathcal{M})
\end{tikzcd}
\end{equation}
If $\mathcal{N}$ is a sheaf of $\mathcal{B}$ modules on $Y$, the following commutes
\begin{equation}\label{diag11}
\begin{tikzcd}
\text{H}^\cdot(Y,\mathcal{B})\otimes \text{H}^\cdot(Y,\mathcal{N})\arrow[r]\arrow[d] & \text{H}^\cdot(Y,\mathcal{N})\arrow[d]\\
\text{H}^\cdot(X,f^*\mathcal{B})\otimes \text{H}^\cdot(X,f^*\mathcal{N})\arrow[r] & \text{H}^\cdot(X,f^*\mathcal{N})
\end{tikzcd}
\end{equation}
\end{lem}

\subsection{} Let $\mathcal{M}$ be an $\mathcal{A}$ module on $X$. It is called \textbf{cohomologically torsion} if $\text{H}^\cdot(\mathcal{M})$ is a torsion $\text{H}^\cdot(\mathcal{A})$ module. Write $\text{Tors}_\mathcal{A}(X)\subseteq \text{Sh}_\mathcal{A}(X)$ for the full subcategory of cohomologically torsion modules.

\subsubsection{} Maintaining the notation of Lemma \ref{torslemma}, if $\mathcal{M}$ is a cohomologically torsion $f^*\mathcal{B}$ module, then by (\ref{diag7}) it follows that $f_*\mathcal{M}$ is a cohomologically torsion $\mathcal{B}$ module. Thus $\check{f}_*$ restricts to a functor
$$\check{f}_*\ :\ \text{Tors}_{f^*\mathcal{B}}(X)\ \longrightarrow\ \text{Tors}_\mathcal{B}(Y).$$
Exactly the same is true of the adjoint functor $f^*$, on the condition that $\text{H}^\cdot(Y,\mathcal{B})\to \text{H}^\cdot(X,f^*\mathcal{B})$ is injective. In that case, diagram (\ref{diag11}) gives 
$$f^*\ :\ \text{Tors}_{\mathcal{B}}(Y)\ \longrightarrow\ \text{Tors}_{f^*\mathcal{B}}(X)$$
which is left adjoint to $\check{f}_*$. This condition is satisfied for instance if $f$ is a closed embedding.

\subsubsection{} When $\mathcal{B}=\text{B}_Y$ is a constant sheaf we have $\text{H}^\cdot(\mathcal{B})=\text{H}^\cdot(X)\otimes \text{B}$. Writing $\text{Tors}_\text{B}(X)\subseteq \text{Sh}_\mathcal{B}(X)$ for the full subcategory of those sheaves whose cohomology are torsion $\text{B}$ modules, we have

\begin{prop}\label{torsprop}The functor $\check{f}_*$ restricts to a functor
$$\check{f}_*\ : \ \text{Tors}_{\text{B}}(X)\ \longrightarrow\ \text{Tors}_{\text{B}}(Y)$$
and similarly if $\text{H}^\cdot(Y,\mathcal{N})\to \text{H}^\cdot(X,f^*\mathcal{N})$ is injective, $f^*$ restricts to
$$f^*\ : \ \text{Tors}_{\text{B}}(Y)\ \longrightarrow\ \text{Tors}_{\text{B}}(X).$$
\end{prop}
\begin{proof}
Compose diagrams (\ref{diag7}) and (\ref{diag11}) with the map $\text{B}\to \text{H}^\cdot(\mathcal{B})$ to give 
\begin{center}
\begin{tikzcd}
\text{B}\otimes \text{H}^\cdot(Y,f_*\mathcal{M})\arrow[r]\arrow[d,equals] & \text{H}^\cdot(Y,f_*\mathcal{M})\arrow[d,equals]&& \text{B}\otimes \text{H}^\cdot(Y,\mathcal{N})\arrow[r]\arrow[d] & \text{H}^\cdot(Y,\mathcal{N})\arrow[d]\\
\text{B}\otimes \text{H}^\cdot(X,\mathcal{M})\arrow[r] & \text{H}^\cdot(X,\mathcal{M})&& \text{B}\otimes \text{H}^\cdot(X,f^*\mathcal{N})\arrow[r] & \text{H}^\cdot(X,f^*\mathcal{N})
\end{tikzcd}
\end{center}
\end{proof}

\subsection{} \label{torscomp} Our main application of the above concerns Euler classes. Take complementary closed and open embeddings of dg schemes
\begin{center}
\begin{tikzcd}
Z\ \arrow[r]& X& \arrow[l,swap] U
\end{tikzcd}
\end{center}
which are equivariant under the action of a group $\text{G}$, hence give maps on quotient stacks
\begin{center}
\begin{tikzcd}
Z/\text{G}\ \arrow[r,"i"]& X/\text{G}& U/\text{G}\arrow[l,swap, "j"] 
\end{tikzcd}
\end{center}
The cohomology of any sheaf on these spaces is naturally a $\text{H}^\cdot(\text{BG})$ module. Let
$$\mathcal{F}\ \in \ \text{Tors}_{\text{H}^\cdot(\text{BG})}(U/\text{G})$$
be one such giving a torsion module, e.g. by Proposition \ref{schtors} any sheaf works if the action on $U$ is free. Then Proposition \ref{torsprop} implies
\begin{prop}\label{propclosedopen}
The cohomology of $i^*j_*\mathcal{F}$ is a torsion $\text{H}^\cdot(\text{BG})$ module.
\end{prop}
Considering the Mayer Vietoris sequence
$$j_!\mathcal{F}\ \longrightarrow\ j_*\mathcal{F}\ \longrightarrow\ i_*i^*j_*\mathcal{F}\ \stackrel{+1}{\longrightarrow}$$
and noting that the assumption on $\mathcal{F}$ is precisely that $\text{H}^\cdot(j_*\mathcal{F})$ is torsion, we arrive at
\begin{cor}\label{j!tors}
The cohomology of $j_*\mathcal{F}$ and $j_!\mathcal{F}$ are torsion $\text{H}^\cdot(\text{BG})$ modules.
\end{cor}

Another consequence is that if $\mathcal{G}$ is a sheaf on $X/\text{G}$ whose pullback $j^*\mathcal{G}$ has torsion cohomology, applying cohomology to the Mayer Vietoris sequence
$$i^!\mathcal{G}\ \longrightarrow\ i^*\mathcal{G}\ \longrightarrow\ i^*j_*j^*\mathcal{G}\ \stackrel{+1}{\longrightarrow}$$
and noting that the last entry is torsion by Proposition \ref{propclosedopen}, we get
\begin{cor}
On cohomology, modulo $\text{H}^\cdot(\text{BG})$ torsion the map $i^!\mathcal{G}\to i^*\mathcal{G}$ is an isomorphism.
\end{cor}
When $\mathcal{G}$ the constant sheaf this map is precisely multiplication by the (bivariant) Euler class, so for instance if the action on $U$ is free the equivariant Euler class of $Z$ in $X$ is a unit modulo torsion.

\subsection{Relative setting} \label{torsrel} Proposition \ref{schtors} is a statement about (dg) schemes, but it holds more generally. Before getting to this, we first need results on freeness. 

\subsubsection{} First to fix our class of maps. We want to include maps like $\text{Ext}\to X$. We consider
$$ f \ :\ X\ \longrightarrow\ Y,$$
representable (or relative Deligne-Mumford) maps of finite type between Artin stacks which are locally of finite type over a field $k$. Recall that $f$ being representable (resp. Deligne-Mumford) means that for any scheme (resp. Deligne-Mumford stack) $V$ mapping into $Y$, the pullback 
\begin{center}
\begin{tikzcd}
U\arrow[r]\arrow[d]& V\arrow[d]\\
X\arrow[r]& Y
\end{tikzcd}
\end{center}
is a scheme (resp. Deligne Mumford stack).

\begin{lem}
For $f$ as above, $R^qf_*k=0$ for $q\gg 0$.
\end{lem}
\begin{proof}Let $V$ be a smooth atlas of $Y$, and take
\begin{center}
\begin{tikzcd}
U\arrow[r,"\overline{f}"]\arrow[d,"\overline{\pi}"]& V\arrow[d,"\pi"]\\
X\arrow[r,"f"]& Y
\end{tikzcd}
\end{center}
Then because $\pi$ is smooth, by smooth base change
$$\pi^*f_*k\ =\ \overline{f}_*\overline{\pi}^*k\ =\ \overline{f}_*k$$
and it is enough to show that $R^q\overline{f}_*k$ vanishes for high enough $q$. Because $\overline{f}$ is a map between schemes (or Deligne-Mumford stacks), a fibre of $(R^q\overline{f}_*k)_v$ is the limit of cohomologies of neighbourhoods of ${\overline{f}}^{-1}(v)$ in $U$. 

 Because $Y$ is locally of finite type, so is $V$ since smooth maps are locally of finite type. Thus we can pick a neighhourhood $V_v$ of $v$ in $V$, which is finite type over $k$. In particular we consider $U_v=f^{-1}(V_v)$ and 
$$\overline{f}\vert_{U_v}\ : \ U_v\ \longrightarrow\ V_v,$$
which is a finite type map between finite type schemes over $k$, and as such $R^q\overline{f}_*k$ vanishes when $q$ is above twice the dimension of $U_v$.
\end{proof}

This is the relative version of the fact that the cohomology of finite dimensional schemes vanishes in high enough degrees.

\subsubsection{} Let $f:X\to Y$ be a representable map of derived Artin stacks, and $\text{G}$ a smooth group scheme over $Y$ acting on $X$ with {\'e}tale relative stabiliser groups (see appendix \ref{stab}). Then by Proposition \ref{dmrelprop} the map
$$h\ :\ X/\text{G}\ \longrightarrow\ Y$$
is relative Deligne Mumford.

\begin{prop}
All elements in $\text{H}^\cdot(\text{BG})\subseteq \text{H}^\cdot_\text{G}(X)$ of positive degree are nilpotent.
\end{prop}
\begin{proof}
This result is true for the fibres of $h$ because it is relative Deligne Mumford so $R^qh_*k$ vanishes for high enough $q$. To get it globally we will now apply the Leray spectral sequence to $h$:
\begin{equation}\label{lerayss}
E_2^{p,q}\ =\ \text{H}^p(Y,R^qh_*k)\ \Longrightarrow\ \text{H}^{p+q}(X/\text{G},k).
\end{equation}
Note that the Leray spectral sequence, like any Grothendieck spectral sequence, is functorial in the element of the derived category, here $k$. See for instance \cite{Del}. Given  a cohomology class $\alpha$ in $\text{H}^a(\text{BG})$ of degree $a$, cup product with that class gives a map $k\to k[a]$. Thus for each $r\ge 2$ we get a map
$$\alpha_r\ :\ E_r^{p,q}\ \longrightarrow\ E_r^{p,q+a}$$
which commutes with the differential $d_r$, and taking cohomology gives $\alpha_{r+1}$. When $r=2$ this is a map
$$\alpha_2\ :\ \text{H}^p(Y, R^qh_*k)\ \longrightarrow\ \text{H}^p(Y, R^{q+a}h_*k)$$
and when $r=\infty$ this is cup product with $\alpha$
$$\alpha_\infty\ :\ \text{H}^{p+q}(X/\text{G},k)\ \longrightarrow\ \text{H}^{p+q+a}(X/\text{G},k).$$
The point is that when $r=2$ the action of $\alpha$ is visibly nilpotent if it has positive degree, because for large enough $a$ the sheaf $R^{q+a}h_*k$ will vanish. This implies the same for $r=\infty$, i.e. that the action of positive degree elements of $\text{H}^\cdot(\text{BG})$ on $\text{H}^\cdot_\text{G}(X)$ is nilpotent.
\end{proof}

As before, if $\text{G}$ is a positive dimensional reductive group, so that the cohomology of $\text{H}^\cdot(\text{BG})$ is a polynomial algebra, this implies

\begin{cor}
The cohomology $\text{H}^\cdot(X,\mathcal{F})$ of any sheaf $\mathcal{F}$ on $X/\text{G}$ is a torsion $\text{H}^\cdot(\text{BG})$ module.
\end{cor}

\subsubsection{} \label{compleqvt} In particular, let 
\begin{center}
\begin{tikzcd}
Z\arrow[r] & X & U \arrow[l]
\end{tikzcd}
\end{center}
be complementary open and closed embeddings of derived Artin stacks with representable maps to $Y$. Assume each three is acted upon by $\text{G}$, a smooth group scheme over $Y$, with the stabilisers of the action on $U$ being etale. Then we can take
\begin{center}
\begin{tikzcd}
Z/\text{G}\ \arrow[r,"i"]& X/\text{G}& U/\text{G}\arrow[l,swap, "j"] 
\end{tikzcd}
\end{center}
 complementary open and closed embeddings over $Y$. Since the map $U/\text{G}\to Y$ is relative Deligne Mumford, arguing exactly as in section \ref{torscomp} we have
\begin{prop}
The cohomology of $i^*j_*\mathcal{F}$ is torsion $\text{H}^\cdot(\text{BG})$ module.
\end{prop}

From this it follows as before that

\begin{cor}\label{j!tors}
The cohomology of $j_*\mathcal{F}$ and $j_!\mathcal{F}$ are torsion $\text{H}^\cdot(\text{BG})$ modules.
\end{cor}

Moreover, if $\mathcal{G}$ is a sheaf on $X/\text{G}$ whose pullback $j^*\mathcal{G}$ has torsion cohomology, 
\begin{cor}
On cohomology, modulo $\text{H}^\cdot(\text{BG})$ torsion the map $i^!\mathcal{G}\to i^*\mathcal{G}$ is an isomorphism.
\end{cor}

\subsection{} \label{pullbacklem} We finish this section with a lemma which will be useful later. Let 
$$\overline{X}\ \stackrel{s}{\longrightarrow}\ X\ \stackrel{t}{\longleftarrow}\ X^\circ \ \ \ \text{ and }\ \ \ Z\ \stackrel{i}{\longrightarrow}\ X\ \stackrel{j}{\longleftarrow}\ U$$
be two complementary open and closed embeddings as in section \ref{compleqvt}. Pulling back
\begin{center}
\begin{tikzcd}
\overline{Z}/\text{G}\arrow[r,"\overline{\iota}"]\arrow[d,"\overline{s}"] & \overline{X}/\text{G}\arrow[d,"s"]& \overline{U}/\text{G}\arrow[l,"\overline{\jmath}",swap]\arrow[d,"\underline{s}"]\\
Z/\text{G}\arrow[r,"i"] & X/\text{G}& U/\text{G}\arrow[l,"j",swap]\\
Z^\circ/\text{G}\arrow[r,"\underline{i}"]\arrow[u,"\overline{t}",swap] & X^\circ/\text{G}\arrow[u,"t",swap] &U^\circ/\text{G}\arrow[l,swap,"\underline{j}"]\arrow[u,swap,"\underline{t}"] 
\end{tikzcd}
\end{center}
gives a pullback map
$$i^*\ :\ \text{H}^\cdot_{\text{G}}(\overline{X}/X)\ \longrightarrow\ \text{H}^\cdot_{\text{G}}(\overline{Z}/Z).$$
Then
\begin{lem}$i^*$ is an isomorphism  modulo torsion.
\end{lem}
\begin{proof}
Now, the map $i^*$ is defined by Mayer-Vietoris in the horizontal direction
$$s^!j_!j^*k\ \longrightarrow\ s^!k\ \stackrel{i^*}{\longrightarrow} \ s^!i_*i^*k\ \stackrel{+1}{\longrightarrow}$$
so it is enough to show that $\text{H}^\cdot_{\text{G}}(\overline{X},s^!j_!k)$ is torsion. The Mayer-Vietoris sequence in the vertical direction gives
$$s_*s^!j_!k\ \longrightarrow\ j_!k\ \longrightarrow\ t_*t^!j_!k\ \stackrel{+1}{\longrightarrow}$$
It follows from corollary \ref{j!tors} that the cohomology of both $j_*k$ and $t_*t^!j_!k=t_*\underline{j}_!k$ are torsion $\text{H}^\cdot(\text{BG})$ modules, because the actions of $\text{G}$ on $U$ and on $U^\circ$ are free.
\end{proof}

\section{Specialisation, deformation and fundamental classes}\label{spec} 

\subsection{} \label{spnat} In this section we will be considering commutative diagrams $\mathcal{D}$ of spaces of the shape
\begin{center}
\begin{tikzcd}
\textcolor{white}{a} \bullet \textcolor{white}{a}\arrow[r,"\overline{\jmath}"]&\textcolor{white}{a} \bullet \textcolor{white}{a} & \textcolor{white}{a} \bullet \textcolor{white}{a}\arrow[l,"\overline{\iota}",swap]\\
\textcolor{white}{a} \bullet \textcolor{white}{a}\arrow[u,"f"]\arrow[r,"j"]& \textcolor{white}{a} \bullet \textcolor{white}{a}\arrow[u,"\widetilde{f}"]& \textcolor{white}{a} \bullet \textcolor{white}{a}\arrow[u,"f_0"]\arrow[l,"i",swap]
\end{tikzcd}
\end{center}
where the rows are complementary open and closed embeddings. To this data we can associate the natural transformation
$$\text{sp}_f\ :\ j_*f^!\overline{\jmath}^!\ \Longrightarrow\ i_*f_0^!\overline{\iota}^![1],$$
defined by taking the boundary term in the Gysin sequence, $j_*j^!\Rightarrow i_*i^![1]$, and precomposing with $\widetilde{f}^!$. We will sometimes adopt the shorthand 
$$\text{sp}_f\ :\ \mathcal{J}^\mathcal{D}\ \Longrightarrow\ \mathcal{I}^\mathcal{D}.$$
Given two such diagrams
\begin{center}
\begin{tikzcd}
\textcolor{white}{a} \bullet \textcolor{white}{a}\arrow[r,"\overline{\jmath}"]&\textcolor{white}{a} \bullet \textcolor{white}{a} & \textcolor{white}{a} \bullet \textcolor{white}{a}\arrow[l,"\overline{\iota}",swap]\\
\textcolor{white}{a} \bullet \textcolor{white}{a}\arrow[u,"f"]\arrow[r,"j"]& \textcolor{white}{a} \bullet \textcolor{white}{a}\arrow[u,"\widetilde{f}"]& \textcolor{white}{a} \bullet \textcolor{white}{a}\arrow[u,"f_0"]\arrow[l,"i",swap]\\
\textcolor{white}{a} \bullet \textcolor{white}{a}\arrow[u,"g"]\arrow[r,"\underline{j}"]& \textcolor{white}{a} \bullet \textcolor{white}{a}\arrow[u,"\widetilde{g}"]& \textcolor{white}{a} \bullet \textcolor{white}{a}\arrow[u,"g_0"]\arrow[l,"\underline{i}",swap]
\end{tikzcd}
\end{center}
then we have
$$\text{sp}_{fg}\ =\ \text{sp}_g\cdot \widetilde{f}^!\ =\ \widetilde{g}^!\cdot \text{sp}_f$$
as natural transformations $\underline{j}_*g^!f^!\overline{j}^!\Rightarrow \underline{i}_*g_0^!f_0^!\overline{\iota}^![1]$.

\subsubsection{} \label{gmexample} The basic example is for the diagram
\begin{center}
\begin{tikzcd}
\mathbf{G}_m\arrow[r,"j"]& \mathbf{A}^1& \arrow[l,swap,"i"]\text{pt}\\
\mathbf{G}_m\arrow[r,"j"]\arrow[u,equals]& \mathbf{A}^1\arrow[u,equals]& \arrow[l,swap,"i"]\text{pt}\arrow[u,equals]
\end{tikzcd}
\end{center}
On the level of cohomology specialisation gives
$$\text{H}^\cdot(\text{sp})\ :\ \text{H}^\cdot(\mathbf{G}_m)\ \longrightarrow\ \text{H}^{\cdot -1}(\text{pt})$$
which picks out the coefficient of a generator $\gamma\in \text{H}^1(\mathbf{G}_m)$.

\subsection{Deformation to the normal bundle/complex} \label{defbun} Let $s:Z\to X$ be a closed embedding of smooth spaces. In topology, the exponential map would allow us to compare it with the zero section to the normal bundle, $s_0:Z\to N$. But although the exponential map does not actually exist in algebraic geometry, $s$ and $s_0$ can still be compared using the \textit{deformation to the normal bundle}, denoted $D$.

It is a flat $\mathbf{A}^1$-valued family of spaces over $X$ admitting a map from $Z$. Above $0\in \mathbf{A}^1$ it is $s$ and elsewhere it is $s_0$. It fits into a diagram of pullback squares
\begin{equation}\label{diag8}
\begin{tikzcd}
X\times\mathbf{G}_m \arrow[r,"\widehat{\jmath}"]&X\times \mathbf{A}^1 &X\times\text{pt}\arrow[l,swap,"\widehat{\iota}"]\\
X\times\mathbf{G}_m \arrow[r,"\overline{\jmath}"]\arrow[u,"\pi"]&D\arrow[u,"\widetilde{\pi}"] & N\arrow[l,"\overline{\iota}",swap]\arrow[u,"\pi_0"]\\
Z\times\mathbf{G}_m\arrow[u,"s\times \text{id}"]\arrow[r,"j"]& Z\times\mathbf{A}^1\arrow[u,"\widetilde{s}"]& Z\times\text{pt} \arrow[u,"s_0"]\arrow[l,"i",swap]
\end{tikzcd}
\end{equation}
where the horizontal arrows are complementary open and closed embeddings, and the bottom horizontal arrows are closed embeddings. See \cite{F} for more details.

\subsubsection{} Notice that the trivial family $D_0=X\times\mathbf{A}^1$ also satisfies all the above, but with $N$ replaced by $X$. To obtain the deformation to the normal cone from this, first blow $D_0$ up along $Z\times \text{pt}$. The fibre above zero now contains $N$ as an open subspace, with complement the blowup of $X$ along $Z$. Removing this closed subspace then gives $D$. It follows that $D$ is smooth since both $Z$ and $X$ are.

\subsubsection{} This construction can be generalised in two different ways. Neither enforce any assumptions on the spaces, only on the map $s$. In both cases the normal complex $\mathbf{N}_{Z/X}$ is a perfect complex, so we can consider its total space.

\begin{theorem}
Let $s:Z\to X$ be a map of derived Artin stacks. If either $s$ is quasismooth, or a closed embedding which is homotopically of finite presentation,\footnote{Equivalently, it has finite presentation and its normal complex $\mathbf{N}_{Z/X}$ is perfect.} then there is a derived Artin stack $D$ fitting into a diagram (\ref{diag8}), where $N=\mathbf{N}_{Z/X}$. It is called the \textit{deformation to the normal complex}.

In the first case all vertical arrows are quasismooth, and in the second the bottom row of vertical arrows are closed embeddings homotopically of finite presentation.
\end{theorem}
\begin{proof}
If $s$ is quasismooth this is Theorem $1.3$ of \cite{K}. The second case is to appear in the upcoming paper \cite{HKR}.
\end{proof}

\subsection{Fundamental classes} For any quasismooth map $f:X\to Y$ between derived Artin stacks of virtual dimension $d$, one can define its \textit{virtual fundamental class} $1_{X/Y}\in H^{-2d}(X/Y)$, introduced in \cite{K}. Despite the fact that bivariant cohomology
$$\text{H}^\cdot(X/Y)\ =\ \text{H}^\cdot(X_{cl}/Y_{cl})$$
does not see the derived structure, the fundamental class is sensitive to it. When $f$ is smooth, $\text{H}^\cdot(X/Y)= \text{H}^{\cdot +2d}(X)$ and the fundamental class corresponds to $1\in \text{H}^\cdot(X)$.

\subsubsection{}

The main theorem of \cite{K} is:
\begin{theorem} Let $f:X\to Y$ be a quasismooth map of derived stacks. Then it has a \textit{fundamental class}, an element $1_{X/Y}\in \text{H}^\cdot(X/Y)$. These satisfy the following properties:
\begin{enumerate}[label = \arabic*)]
\item The construction is stable under base change. If $Z\to W$ is quasismooth, then under any pullback of derived stacks
\begin{center}
\begin{tikzcd}
X\arrow[r]\arrow[d]&Y\arrow[d,"g"]\\
Z\arrow[r]& W
\end{tikzcd}
\end{center}
we have $1_{X/Y}=g^*1_{Z/W}$. We restress that we have taken the \textit{homotopy} pushout; this is necessary for instance to guarantee that $X\to Y$ is quasismooth at all.
\item  The construction is stable under products. If 
$$X\ \longrightarrow\ Y\ \longrightarrow\ Z$$
are two quasismooth maps then the product of $1_{X/Y}$ and $1_{Y/Z}$ is $1_{X/Z}$.
\end{enumerate}
\end{theorem}

To construct it, consider the top half of the diagram (\ref{diag8}). Writing $p:Y\times\mathbf{A}^1\to Y$ for the projection, we get
$$\text{H}^\cdot(\text{sp}_\pi^D p^!k)\ :\ \text{H}^{\cdot +1}(Y\times\mathbf{G}_m /\text{pt})\ \longrightarrow\ \text{H}^\cdot(N/Y).$$
Of course we can repeat this for the trivial family too, giving another map:
$$\text{H}^\cdot(Y/Y)\ \longleftarrow\ \text{H}^{\cdot +1}(Y\times\mathbf{G}_m /\text{pt})\ \longrightarrow\ \text{H}^{\cdot }(N/Y).$$
As we have seen in \ref{gmexample} the left map simply picks out the coefficient of $\gamma\in H^{-1}(\mathbf{G}_m/\text{pt})$, and so admits a section $1\otimes \gamma$. Finally, because the map is quasismooth $N$ is a smooth of dimension $-d$ over $X$, and so $\text{H}^\cdot(X/Y)\simeq \text{H}^{\cdot +2d}(N/Y)$. The fundamental class is then the image of the identity under
$$\text{H}^\cdot(Y/Y)\ \stackrel{1\otimes \gamma}{\longrightarrow}\ \text{H}^{\cdot +1}(Y\times\mathbf{G}_m /\text{pt})\ \longrightarrow\ \text{H}^{\cdot }(N/Y)\ \simeq\ \text{H}^{\cdot +2d}(X/Y).$$

\subsection{Specialisation} If $s:Z\to X$ is a closed embedding of smooth manifolds, the exponential map induces
$$\text{H}^\cdot(Z/X)\ \simeq\ \text{H}^\cdot(Z/N).$$
This isomorphism actually survives to the general case, and is called \textit{specialisation}. Thus although the exponential map does not exist except in the topological setting, its shadow on the level of cohomology does.

\begin{prop}\label{fundspec} Let $s:Z\to X$ be a closed embedding of smooth spaces. There is an isomorphism of sheaves, called \textbf{specialisation}
$$\mathfrak{sp}_s\ :\ s^!k\ \stackrel{\sim}{\longrightarrow}\ s_0^!k.$$
 It is compatible with Euler classes, in the sense that 
\begin{equation}\label{diag9}
\begin{tikzcd}
\text{H}^\cdot(Z)\arrow[d,equals] & \text{H}^\cdot(Z/X)\arrow[opacity=0]{d}[opacity=1]{\sim}\arrow[d,swap,"\text{H}^\cdot(\mathfrak{sp}_s)"]\arrow[l,swap,"e(X)"]& \text{H}^\cdot(Z)\arrow[d,equals]\arrow[l,swap,"\cdot 1_{Z/X}"]\\
\text{H}^\cdot(Z)& \text{H}^\cdot(X/N)\arrow[l,swap,"e(N)"]& \text{H}^\cdot(Z)\arrow[l,swap,"\cdot 1_{Z/N}"]
\end{tikzcd}
\end{equation}
commutes. 
\end{prop}
\begin{proof}

To construct it, we apply the formalism of subsection \ref{spnat} to the bottom half of (\ref{diag8})
\begin{center}
\begin{tikzcd}
X\times\mathbf{G}_m \arrow[r,"\overline{\jmath}"]&D & N\arrow[l,"\overline{\iota}",swap]\\
Z\times\mathbf{G}_m\arrow[u,"s\times \text{id}"]\arrow[r,"j"]& Z\times\mathbf{A}^1\arrow[u,"\widetilde{s}"]& Z\times\text{pt} \arrow[u,"s_0"]\arrow[l,"i",swap]
\end{tikzcd}
\end{center}
 to give
$$i^*\text{sp}_s^Dk\ :\ i^*j_*(s\times \text{id})^!k\ \longrightarrow\ s_0^!\overline{\iota}^!k[1].$$
Because all spaces are smooth $\overline{\iota}^!k\simeq k[-2]$, and picking once an for all a generator of $H^2(\text{pt}/\mathbf{A}^1)$ fixes an isomorphism. The above can then be viewed as a map
$$i^*\text{sp}_s^Dk\ :\ i^*j_*(s\times \text{id})^!k\ \longrightarrow\ s_0^!k[-1]$$
which on cohomology gives
$$\text{H}^\cdot(i^*\text{sp}_s^D)\ :\ \text{H}^\cdot(Z/X)\otimes \text{H}^\cdot(\mathbf{G}_m)\ \longrightarrow\ \text{H}^{\cdot -1}(Z/N).$$
We can repeat this replacing $D$ by the trivial family $X\times\mathbf{A}^1$. Since the left side of the above equation only depends on $s$ and not on the family $D$, we now get two maps
$$\text{H}^{\cdot -1}(Z/X) \ \stackrel{\text{H}^\cdot(i^*\text{sp}_s^{triv})}{\longleftarrow} \ \text{H}^\cdot(Z/X)\otimes \text{H}^\cdot(\mathbf{G}_m)\ \stackrel{\text{H}^\cdot(i^*\text{sp}_s^D)}{\longrightarrow}\ \text{H}^{\cdot -1}(Z/N).$$
As we have seen in \ref{gmexample} the left map simply picks out the coefficient of $\gamma\in H^1(\mathbf{G}_m)$, so admits a section, namely $\text{id}\otimes\gamma$. The specialisation map on cohomology is then
$$\text{H}^\cdot(\mathfrak{sp}_s)\ :\ \text{H}^{\cdot -1}(Z/X) \ \stackrel{\text{id}\otimes \gamma}{\longrightarrow} \ \text{H}^\cdot(Z/X)\otimes \text{H}^\cdot(\mathbf{G}_m)\ \stackrel{\text{H}^\cdot(\text{sp}_s^D)}{\longrightarrow}\ \text{H}^{\cdot -1}(Z/N).$$
Then Lemma $1.6.14$ of \cite{Ay} proves this is an isomorphism; this is true by construction if $X=N$ is a vector bundle. Notice that $\text{id}\otimes \gamma$ lifts to a map of sheaves, giving a section of $\text{sp}_s^{triv}$. Thus the specialisation map also lifts to a map of sheaves
$$\mathfrak{sp}_s\ :\ s^!k\ \stackrel{\text{id}\otimes \gamma}{\longrightarrow}\ i^*j_*(s\times \text{id})^!k\ \stackrel{\text{sp}_s^D}{\longrightarrow}\ s^!_0k.$$
It is independent of the choice we made for the generator of $H^2(\text{pt}/\mathbf{A}^1)$ because upon rescaling the choice by $\lambda$, both $sp^D_s$ and $\text{sp}_s^{triv}$ also rescale by $\lambda$, and so $\mathfrak{sp}_s$ is unchanged.

We now turn to showing specialisation is compatible with Euler classes. Applying the natural transformation $\widetilde{s}^!\to \widetilde{s}^*$ gives a commuting diagram
\begin{center}
\begin{tikzcd}
s^!k\arrow[r]\arrow[d] & i^*j_*(s\times\text{id})^!k[1]\arrow[r]\arrow[d]& s_0^!\overline{\iota}^!k[2]\arrow[d]\\
s^*k\arrow[r] & i^*j_*(s\times\text{id})^*k[1]\arrow[r] & s_0^*\overline{\iota}^!k[2]
\end{tikzcd}
\end{center}
and taking cohomology thereof gives
\begin{center}
\begin{tikzcd}
\text{H}^\cdot(Z/X)\arrow[r,"\text{H}^\cdot(\mathfrak{sp}_s)"]\arrow[d,"e(X)"]& \text{H}^\cdot(Z/N)\arrow[d,"e(N)"]\\
\text{H}^\cdot(Z)\arrow[r,equals]& \text{H}^\cdot(Z)
\end{tikzcd}
\end{center}
Finally we establish compatibility with fundamental classes under the assumption that $s$ is quasismooth. Note that $\widetilde{\pi}$ is quasismooth, so we can consider its fundamental class $1_{D/X\times\mathbf{A}^1}$. Viewing it as a morphism $k_D\to \widetilde{\pi}^!k_{X\times\mathbf{A}^1}$, we have
$$j^!1_{\widetilde{\pi}}\ =\ 1_{X\times\mathbf{G}_m/X\times\mathbf{G}_m}, \ \ \ i^!1_{\widetilde{\pi}}\ =\ 1_{N/X}.$$
Now we can consider the commuting diagram
\begin{center}
\begin{tikzcd}
i_*i^*\mathcal{J}^Dk_D\arrow[r]\arrow[d]& i_*i^*\mathcal{J}^D\widetilde{\pi}^!k_{X\times\mathbf{A}^1}\arrow[d]\\
\mathcal{I}^Dk_D\arrow[r]& \mathcal{I}^D\widetilde{\pi}^!k_{X\times\mathbf{A}^1}
\end{tikzcd}
\end{center}
Here the horizontal arrows are the cup products with $1_{\widetilde{\pi}}$, and the vertical arrows are $\text{sp}^D_s$. Note that $\text{sp}^D_s\widetilde{\pi}^!=\text{sp}^{triv}_{s}$ by subsection \ref{spnat}, and the bottom arrow is the cup product with $1_{N/X}$. The top arrow is the cup product with $1_{X\times\mathbf{G}_m/X\times\mathbf{G}_m}$ and so is the identity. Thus on cohomology the above diagram is
\begin{center}
\begin{tikzcd}
 \text{H}^\cdot(i^*j_*(s\times \text{id})^!k)\arrow[r,equals] \arrow[d,"\text{H}^\cdot(\text{sp}_s^D)"]& \text{H}^\cdot(i^*j_*(s\times \text{id})^!k)\arrow[d,"\text{H}^\cdot(\text{sp}^{triv}_s)"]\\
\text{H}^{\cdot -1}(X/N)\arrow[r,"\cdot 1_{N/X}"]&\text{H}^{\cdot -1}(X/Y)
\end{tikzcd}
\end{center}
This shows that $\cdot 1_{N/X}$ is inverse to $\text{H}^\cdot(\mathfrak{sp}_s^D)$, proving the proposition.

\end{proof}

\subsection{Equivariant specialisation} \label{eqvtloc} For a general closed embedding $s:Z\to X$ the vector spaces $\text{H}^\cdot(Z/X)$ and $\text{H}^\cdot(Z/N)$ are usually not isomorphic. This is not surprising since a neighbourhood of $Z$ in $X$ is usually nothing like a vector bundle on $Z$, e.g. even in the topological setting there is no exponential map.

If however we work $\mathbf{G}_m$ \textit{equivariantly} then modulo torsion things behave as in the smooth setting. More precisely, let $s$ be a $\mathbf{G}_m$ equivariant closed embedding, on the closed piece $Z$ the action being trivial and on the open piece $X\setminus Z$ the action being free. We need to also assume that $s$ is homotopically of finite presentation, so that its normal complex $\mathbf{N}=\mathbf{N}_{Z/X}$ is a perfect complex. In that case,

\begin{prop} \label{locspecprop} There is an isomorphism of vector spaces
$$\text{H}^\cdot(\mathfrak{sp}_s^{loc})\ :\ \text{H}^\cdot_{\mathbf{G}_m}(Z/X)_{loc}\ \stackrel{\sim}{\longrightarrow}\ \text{H}^\cdot_{\mathbf{G}_m}(Z/\mathbf{N})_{loc},$$
which is induced by a map of sheaves called \textbf{localised specialisation}  
$$\mathfrak{sp}_s^{loc} \ :\  \mathcal{I}^{triv}_sk\ \longrightarrow\ \mathcal{I}_s^Dk.$$ 
It respects Euler classes and (if $s$ is quasismooth) fundamental classes, in the sense that 
\begin{equation}\label{diag10}
\begin{tikzcd}
\text{H}^\cdot_{\mathbf{G}_m}(Z)_{loc}\arrow[d,equals] & \text{H}^\cdot_{\mathbf{G}_m}(Z/X)_{loc}\arrow[opacity=0]{d}[opacity=1]{\sim}\arrow[opacity=0]{l}[opacity=1]{\sim}\arrow[d,swap,"\mathfrak{sp}_s^{loc}"]\arrow[l,swap,"e(X)"]& \text{H}^\cdot_{\mathbf{G}_m}(Z)_{loc}\arrow[d,equals]\arrow[l,swap,"\cdot 1_{Z/X}"]\\
\text{H}^\cdot_{\mathbf{G}_m}(Z)_{loc}& \arrow[opacity=0]{l}[opacity=1]{\sim}\text{H}^\cdot_{\mathbf{G}_m}(Z/\mathbf{N})_{loc}\arrow[l,swap,"e(\mathbf{N})"]& \text{H}^\cdot_{\mathbf{G}_m}(Z)_{loc}\arrow[l,swap,"\cdot 1_{Z/\mathbf{N}}"]
\end{tikzcd}
\end{equation}
commutes. 
\end{prop}
\begin{proof} By exactly the same construction as before, we get a map of sheaves on $(Z\times\mathbf{A}^1)/\mathbf{G}_m$
$$\mathfrak{sp}_s^{loc}\ :\ \mathcal{I}_s^{triv}k\ \stackrel{\text{id}\otimes \gamma}{\longrightarrow}\ i_*i^*j_*(s\times\text{id})^!k[-1]\ \stackrel{\text{sp}_s^D}{\longrightarrow}\ \mathcal{I}_s^Dk.$$
This gives a map on vector spaces
$$\text{H}^\cdot_{\mathbf{G}_m}(\mathfrak{sp}_s^{loc})_{loc}\ :\ \text{H}_{\mathbf{G}_m}^{\cdot -1}(Z/X)_{loc}\ \longrightarrow\ \text{H}_{\mathbf{G}_m}^\cdot(Z/X)_{loc}\otimes \text{H}^\cdot(\mathbf{G}_m)\ \longrightarrow\ \text{H}^{\cdot +1}_{\mathbf{G}_m}(Z/D)_{loc}.$$
Now, by Lemma \ref{pullbacklem} the pullback map
$$\text{H}^\cdot_{\mathbf{G}_m}(\mathbf{N}/D)_{loc}\ \stackrel{\sim}{\longrightarrow}\ \text{H}^\cdot_{\mathbf{G}_m}(Z/Z\times\mathbf{A}^1)_{loc}$$
is an isomorphism. Thus a choice of generator of $\text{H}^2_{\mathbf{G}_m}(\text{pt}/\mathbf{A}^1)$ gives a class $1^{loc}_{\mathbf{N}/D}\in \text{H}^2_{\mathbf{G}_m}(\mathbf{N}/D)_{loc}$, cupping with which induces an isomorphism $\text{H}^{\cdot +1}(Z/D)_{loc}\simeq \text{H}^{\cdot -1}(Z/\mathbf{N})_{loc}$. Composing this with the above then gives
$$\text{H}^\cdot(\mathfrak{sp}_s^{loc})\ :\  \text{H}^\cdot(Z/X)_{loc}\ \longrightarrow\ \text{H}^\cdot(Z/\mathbf{N})_{loc}.$$
The same argument as in the smooth case gives the compatibility with Euler classes and fundamental classes, and shows that the specialisation map is independent of the choice of $1^{loc}_{\mathbf{N}/D}$. Compatibility with Euler classes shows that equivariant specialisation is an isomorphism on cohomology, because so too is the cup product with Euler classes.
\end{proof}

Informally, this proposition says that a closed embedding and its normal cone have the same bivariant homology, up to ``small" corrections terms which are torsion $\text{H}^\cdot(\text{B}\mathbf{G}_m)$ modules.

\subsubsection{} \label{eqvtlocrelative} If $s$ is not quasismooth but the spaces are quasismooth over some base, specialisation is still related to fundamental classes in the obvious way. That is, take
\begin{center}
\begin{tikzcd}
Z\arrow[rd,"\overline{f}"]\arrow[d,"s"]&\\
X\arrow[r,"f"]& Y
\end{tikzcd}
\end{center}
where $f,\overline{f}$ are quasismooth and $s$ is as before a closed embedding, homotopically of finite presentation. In particular, $\mathbf{N}\to Z$ is quasismooth by looking at the distinguished triangle of tangent complexes, and hence also so is $\mathbf{N}\to Y$. Then

\begin{prop}  The following commutes
\begin{center}
\begin{tikzcd}
\text{H}^\cdot_{\mathbf{G}_m}(Z/Y)_{loc}\arrow[d,equals]& \text{H}^\cdot_{\mathbf{G}_m}(Z/X)_{loc}\arrow[l,swap,"\cdot 1_{X/Y}"]\arrow[d,"\mathfrak{sp}_s^{loc}"] \\
\text{H}^\cdot_{\mathbf{G}_m}(Z/Y)_{loc}& \text{H}^\cdot_{\mathbf{G}_m}(Z/\mathbf{N})_{loc}\arrow[l,swap,"\cdot 1_{\mathbf{N}/Y}"] &
\end{tikzcd}
\end{center}
\end{prop}
\begin{proof}
We argue exactly as in \ref{fundspec}, but instead of (\ref{diag8}) we use the diagram
\begin{center}
\begin{tikzcd}
Y\times\mathbf{G}_m\arrow[r]& Y\times\mathbf{A}^1 & Y\times\text{pt}\arrow[l] \\
X\times\mathbf{G}_m\arrow[r,"j"]\arrow[u,"f\times \text{id}"] & D\arrow[u]& \mathbf{N}\arrow[l,swap,"i"]\arrow[u]\\
Z\times\mathbf{G}_m\arrow[r]\arrow[u,"s\times\text{id}"]& Z\times\mathbf{A}^1\arrow[u]& Z\times\text{pt}\arrow[l]\arrow[u]
\end{tikzcd}
\end{center}
Note that the top row of vertical arrows are all quasismooth. Writing $\widetilde{1}=1_{D/Y\times\mathbf{A}^1}$, we have
$$j^!\widetilde{1}\ =\ 1_{X\times\mathbf{G}_m/Y\times\mathbf{G}_m}, \ \ \ i^!\widetilde{1}\ =\ 1_{\mathbf{N}/Y}$$
and arguing as in \ref{fundspec} get the diagram
\begin{center}
\begin{tikzcd}
 \text{H}^\cdot(i^*j_*(s\times \text{id})^!k)_{loc}\arrow[r,"\cdot 1_{X/Y}"] \arrow[d,"\text{sp}_s^Dk"]& \text{H}^\cdot(i^*j_*(\overline{f}\times \text{id})^!k)_{loc}\arrow[d,"sp^{triv}_s f^!k"]\\
\text{H}^{\cdot -1}(Z/\mathbf{N})_{loc}\arrow[r,"\cdot 1_{\mathbf{N}/Y}"]&\text{H}^{\cdot -1}(Z/Y)_{loc}
\end{tikzcd}
\end{center}
This is almost the diagram in the proposition, except that we the order of $\cdot 1_{X/Y}$ and $sp^{triv}_s f^!k$ needs to be swapped. More precisely, the following commutes:
\begin{center}
\begin{tikzcd}
 \text{H}^\cdot(i^*j_*(s\times \text{id})^!k)_{loc}\arrow[r,"\cdot 1_{X/Y}"] \arrow[d,"\text{sp}_s^{triv}k"]& \text{H}^\cdot(i^*j_*(\overline{f}\times \text{id})^!k)_{loc}\arrow[d,"sp^{triv}_s f^!k"]\\
\text{H}^{\cdot -1}(Z/X)_{loc}\arrow[r,"\cdot 1_{X/Y}"]&\text{H}^{\cdot -1}(Z/Y)_{loc}
\end{tikzcd}
\end{center}
which proves the proposition.
\end{proof}

\subsection{Examples} The situation is uninteresting when all spaces are smooth, because the bivariant groups are plain cohomology groups, fundamental classes are the identity and the specialisation map an isomorphism on the nose. The simplest interesting examples are the linear ones, i.e. coming from perfect complexes/cone bundles.

\subsubsection{} Let $E$ be a perfect complex on space $Z$; recall that
\begin{enumerate}[label = \arabic*)]
\item $E\to Z$ is a quasismooth map iff $E$ is concentrated in degrees $(-\infty,1]$.
\item $Z\to E$ is a closed embedding iff $E$ is concentrated in degrees $[0,\infty)$.
\end{enumerate}
For instance, assume that
$$E\ =\ (E_{\le 0}\ \stackrel{f}{\to} \ E_1)$$
where $E_{\le 0}$ is a perfect complex in degrees $\le 0$, and $E_1$ is a vector bundle. In this case we can construct the fundamental class $1_{E/Z}$ directly, bypassing the general construction. Consider the pullback
\begin{center}
\begin{tikzcd}
E\arrow[r]\arrow[d]& Z\arrow[d]\\
E_{\le 0}\arrow[r,"f"]& E_1
\end{tikzcd}
\end{center}
then we have
$$1_{E/Z} \ =\ 1_{E/E_0}\cdot 1_{E_0/Z}\ =\ f^*1_{Z/E_1}\cdot 1_{E_0/Z}\ =\ \ =\ f^*1_{E_1/Z}^{-1}\cdot 1_{E_0/Z}$$
defined completely in terms of fundamental classes of smooth morphisms, which have a very direct definition.

\subsection{Umkehr maps} \label{umkehr} Thinking about fundamental classes gives a cleaner formulation of pushforward maps on cohomology. If $f:X\to Y$ quasismooth of virtual dimension $d$ and proper, we get two maps
$$\text{H}^\cdot(X)\ \stackrel{\cdot 1_{X/Y}}{\longrightarrow}\ \text{H}^{\cdot -2d}(X/Y)\ \stackrel{f_*}{\longrightarrow}\ \text{H}^{\cdot -2d}(Y)$$
respectively. The pushforward map on cohomology is their composite, which will be referred to as $f_*$ when there is no risk of confusion with the pushforward map $f_*$ on bivariant homology.

\begin{lem}Given a pullback of derived stacks
\begin{center}
\begin{tikzcd}
X\arrow[r,"\overline{f}"]\arrow[d,"\overline{g}"]& Y\arrow[d,"g"]\\
Z\arrow[r,"f"]& W
\end{tikzcd}
\end{center}
such that $f$ (and hence $\overline{f}$) is proper and quasismooth, then $$g^*f_*\ = \ \overline{f}_*\overline{g}^*$$
as maps $\text{H}^\cdot(Z)\to \text{H}^\cdot(Y)$.
\end{lem}
\begin{proof}
For a class $\alpha\in \text{H}^\cdot(Z)$, we have 
$$g^*f_*(\alpha\cdot 1_{Z/W})\ \stackrel{A_{23}}{=}\ \overline{f}_*\overline{g}^*(\alpha\cdot 1_{Z/W})\ \stackrel{A_{13}}{=} \overline{f}_*(\overline{g}^*(\alpha)\cdot \overline{g}^*1_{Z/W})\ =\ \ \overline{f}_*(\overline{g}^*(\alpha)\cdot 1_{X/Y}).$$
\end{proof}

\section{Abelian localisation} \label{abloc}

\subsection{} \textit{Abelian localisation} is a collection of results saying that the torus equivariant cohomology of a space and its fixed locus are approximately the same
$$\text{H}^\cdot_\text{T}(X)\ \approx\ \text{H}^\cdot_\text{T}(X^\text{T}).$$
Here $X$ is a space acted upon by a split torus $\text{T}=\mathbf{G}_m^n$. This can drastically simplify computations in $\text{H}^\cdot_\text{T}(X)$ because the fixed locus is typically a much simpler space than $X$ itself.

\subsection{Smooth case} \label{smoothab} Let $X$ be a smooth (dg) scheme of finite type acted upon by a split torus $\text{T}$. More generally than just fixed loci we can consider $\text{T}$ equivariant closed embeddings
$$i\ :\ \overline{X}\ \longrightarrow\ X$$
of smooth schemes on whose complement the stabilisers have dimension less than $\text{dim}\text{T}$, e.g. the action is free. Write $d$ for its codimension. Atiyah and Bott's abelian localisation then says

\begin{theorem} (\cite{AB}, $3.5$) \label{smoothabthm}The pullback and pushforward maps on cohomology
$$  \text{H}^\cdot_{\text{T}}(\overline{X})\ \stackrel{i_*}{\longrightarrow}\ \text{H}^{\cdot +2d}_{\text{T}}(X)\textcolor{white}{aaaaaaa} \text{H}^\cdot_{\text{T}}(X)\ \stackrel{i^*}{\longrightarrow}\  \text{H}^\cdot_{\text{T}}(\overline{X})$$ 
become isomorphisms of $\text{H}^\cdot(\text{BT})$ modules after localising at the maximal ideal of $\text{H}^\cdot(\text{BT})$.
\end{theorem}

\subsubsection{} We carefully spell out how to get Atiyah and Bott's differential geometric proof to work in the world of algebraic geometry. There are three steps:
\begin{enumerate}[label = \alph*)]
\item show that $\text{H}^\cdot_\text{T}(X\setminus \overline{X})$ dies when localised,
\item prove the theorem in the special case that $X$ is a vector bundle over $\overline{X}$, then
\item use the exponential map to relate the cohomology of $X$ with the normal bundle $N=N_{\overline{X}/X}$ of $\overline{X}$ in $X$.
\end{enumerate}

\subsubsection{} It is worth seeing how the proof works in the simplest case
$$\mathbf{G}_m\ \curvearrowright\ V$$
of the scaling action on a vector space. In this case the equivariant cohomologies $\text{H}^\cdot_{\mathbf{G}_m}(0)$ and $\text{H}^\cdot_{\mathbf{G}_m}(V)$ are literally the same because $V$ equivariantly contracts onto its fixed locus, so $i^*$ is an isomorphism on the nose.

To understand $i_*$, recall from appendix \ref{sheaves} that it fits into the Mayer-Vietoris sequence
$$\cdots\ \longrightarrow\  \text{H}^{\cdot -2\text{rk}V}_{\mathbf{G}_m}(0)\ \stackrel{i_*}{\longrightarrow}\ \text{H}^\cdot_{\mathbf{G}_m}(V)\ \longrightarrow\ \text{H}^\cdot_{\mathbf{G}_m}(V\setminus 0)\ \longrightarrow\ \cdots$$
The first two terms are $k[t]$, but the last term is
$$\text{H}^\cdot_{\mathbf{G}_m}(V\setminus 0)\ =\ \text{H}^\cdot(\mathbf{P}V)\ =\ k[t]/t^{\text{rk}V+1}.$$
In other words, it is a torsion module and so $i_*$ is an isomorphism modulo torsion: it is multiplication by $t^{\text{rk}V}$.

\subsubsection{} \label{vectbund} More generally, if $V$ is a vector bundle over a scheme $\overline{X}$ of finite type, the cohomology $\text{H}^\cdot(\textbf{P}V)$ of its projectivisation is finite dimensional and hence a torsion module. Alternatively, using Corollary \ref{schtorscor} allows us to drop the finite type assumption. Thus the above proof works here also, showing b) for $\text{T}=\mathbf{G}_m$.

\subsubsection{} The composition 
$$\text{H}_{\mathbf{G}_m}^{\cdot-2\text{rk}V}(\overline{X})\ \stackrel{i_*}{\longrightarrow}\ \text{H}_{\mathbf{G}_m}^{\cdot}(V)\ \stackrel{i^*}{\longrightarrow}\ \text{H}_{\mathbf{G}_m}^{\cdot}(\overline{X}) $$
is multiplication by the Euler class $e(V)$ of the vector bundle $V/\mathbf{G}_m$ over $\overline{X}/\mathbf{G}_m$. In particular, these maps being isomorphisms modulo torsion means that
$$e(V)\ \in \ \text{H}^\cdot_{\mathbf{G}_m}(\overline{X})_{loc}\ :=\ \text{H}^\cdot_{\mathbf{G}_m}(\overline{X})\otimes_{\text{H}^\cdot(\text{B}\mathbf{G}_m)}\Frac \text{H}^\cdot(\text{B}\mathbf{G}_m)$$
is a unit. 

\subsubsection{Proof of Theorem \ref{smoothabthm}} \label{comparison} We first show a), that
$$\text{H}^\cdot_\text{T}(X\setminus \overline{X})$$
is a torsion $\text{H}^\cdot(\text{BT})$ module. When $\text{T}=\mathbf{G}_m$ this follows because the action is free so the quotient is a Deligne-Mumford stack, which has finite dimensional cohomology (see Corollary \ref{schtorscor}). Thus for general tori $X\setminus \overline{X}$ has an open cover by open subschemes whose equivariant cohomology is torsion: take the complement of the fixed point locus of characters $\mathbf{G}_m\subseteq \text{T}$. By assumption no point is stabilised by all of $\text{T}$, so this is a cover. Thus for a fine enough {\v C}ech cover, all terms in the {\v C}ech complex computing $\text{H}^\cdot_\text{T}(X\setminus \overline{X})$ will be torsion $\text{H}^\cdot(\text{BT})$ modules. This proves a).

Combining a) with the preceding discussion proves b).

For c), we note that the exponential map does not exist in algebraic geometry, but the \textit{specialisation} map on bivariant homology is the next best thing:
\begin{equation}\label{eqn6}
\text{H}^\cdot(\mathfrak{sp})\ :\  \text{H}^\cdot_{\mathbf{G}_m}(\overline{X}/X)\ \stackrel{\sim}{\longrightarrow}\text{H}^\cdot_{\mathbf{G}_m}(\overline{X}/N)
\end{equation}
In section \ref{spec} we construct this map, and show that it interlaces the Euler classes
\begin{center}
\begin{tikzcd}
\text{H}^\cdot_{\mathbf{G}_m}(\overline{X}/X)\arrow[r,"\cdot e(X)"]\arrow[d,"\text{H}^\cdot(\mathfrak{sp})"]\arrow[opacity=0]{d}[opacity=1,swap]{\sim}& \text{H}_{\mathbf{G}_m}^*(\overline{X})\arrow[d,equals]\\
\text{H}^\cdot_{\mathbf{G}_m}(\overline{X}/N)\arrow[r,"\cdot e(N)"]& \text{H}_{\mathbf{G}_m}^*(\overline{X})
\end{tikzcd}
\end{center}
In particular, because we have already proven abelian localisation for vector bundles we know that $\cdot e(N)$ is an isomorphism modulo torsion, therefore so is $\cdot e(X)=i^*i_*$. But now the Gysin sequence
$$\cdots\ \longrightarrow\  \text{H}^{\cdot -2d}_{\mathbf{G}_m}(\overline{X})\ \stackrel{i_*}{\longrightarrow}\ \text{H}^\cdot_{\mathbf{G}_m}(\overline{X})\ \longrightarrow\ \text{H}^\cdot_{\mathbf{G}_m}(X\setminus \overline{X})\ \longrightarrow\ \cdots$$
in combination with a) implies that $i_*$ too is an isomorphism modulo torsion, and therefore so is $i^*$. This completes the proof of Theorem \ref{smoothabthm}.

\subsubsection{} The specialisation map is an algebraic geometric shadow of the exponential map/tubular neighbourhood theorem on cohomology.

\subsubsection{} Why schemes, and not general Artin stacks? One difficulty is in proving a) when the ambient space $X$ already has stabiliser groups. We skirt this issue by proving a \textit{relative} version where an Artin stack $X$ admits a fibration whose fibres are schemes.

\subsection{} Atiyah and Bott's paper was based on previous work on \textit{integration formulae}, expressing integrals on $X$ in terms of integrals on the fixed locus. Indeed, as $e(X)$ is a unit in the ring $\text{H}^\cdot_{\text{T}}(\overline{X})_{loc}$, it follows from $i_*=i^*\cdot e(X)$ that
$$\text{id}\ =\ i_*\frac{i^*(-)}{e(X)}$$
as endomorphisms of $\text{H}^\cdot_{\text{T}}(X)_{loc}$. In particular, if $\alpha$ is an equivariant cohomology class on $X$ then pushing forward to a point gives
$$\int_X\alpha\ =\ \int_{\overline{X}} \frac{i^*\alpha}{e(X)}$$
as elements of $\text{H}^\cdot(\text{BT})_{loc}$, the localised equivariant cohomology of a point. Below we will see how to localise integrals of genuine (non-equivariant, non-localised) cohomology classes into fixed loci, in specific examples related to CoHA products.

\subsection{} A natural example of a smooth space with the action of a torus is a flag variety $\text{G}/\text{B}$, for $\text{G}$ a reductive group with maximal torus $\text{T}$. It admits the ``Bruhat'' stratification into affine spaces, labelled by elements of the Weyl group
$$\text{G}/\text{B}\ =\ \coprod_{w\in \text{W}}\text{B}w\text{B}/\text{B}\ \simeq\ \coprod_{w\in \text{W}}\mathbf{A}^{\ell(w)}$$
The action of $\text{T}$ has one fixed point per stratum, so according to abelian localisation we should have
$$\text{H}^\cdot_{\text{T}}(\text{G}/\text{B})_{loc}\ \simeq \ \bigoplus_{w\in\text{W}} \Frac \text{H}^\cdot(\text{BT}).$$
Picking a generic character $\mathbf{G}_m\subseteq\text{T}$, this is the same for the $\mathbf{G}_m$ action. We can verify this explicitly: it is well known that each stratum contributes an element of degree $2\ell(w)$ to  $\text{H}^\cdot(\text{G}/\text{B})$, and together these form a basis.  Thus the pullback map on equivariant cohomology $i^*:\text{H}^\cdot_{\mathbf{G}_m}(\text{G}/\text{B})\to \text{H}^\cdot_{\mathbf{G}_m}((\text{G}/\text{B}^\text{T})$ is the inclusion
$$i^*\ :\ \bigoplus_{w\in \text{W}} t^{\ell(w)}\mathbf{C}[t]\ \longrightarrow\ \bigoplus_{w\in W}\mathbf{C}[t]$$
which visibly becomes an isomorphism after localising. 

\subsubsection{}Other natural examples include partial flag varieties or smooth proper toric varieties. More generally for any smooth projective variety we can apply the Białynicki-Birula theorem, which says that the attracting sets of the fixed points give a equivariant stratification by affine spaces.

\subsection{Singular case} \label{ablocsing} The above generalises fairly straightforwardly to the singular case, working with bivariant homology. What requires more work is turning this into useful statements about honest cohomology groups.

To begin, let $X$ be a (dg) scheme of finite type acted upon by a split torus $\text{T}$, and consider $\text{T}$ equivariant closed embeddings
$$i\ :\ \overline{X}\ \longrightarrow\ X$$
of smooth schemes on whose complement the stabilisers have dimension less than $\text{dim}\text{T}$, e.g. the action is free. In addition, we require the mild assumption is that this map is \textit{homotopically of finite presentation}, which in particular means that its normal complex $\mathbf{N}$ is a perfect complex so that we can consider its total space and hope to compare it to $X$.

\begin{theorem} \label{singabthm}The pullback and pushforward maps on cohomology
$$  \text{H}^\cdot_{\text{T}}(\overline{X}/X)\ \stackrel{i_*}{\longrightarrow}\ \text{H}^{\cdot}_{\text{T}}(X)\textcolor{white}{aaaaaaa} \text{H}^\cdot_{\text{T}}(X)\ \stackrel{i^*}{\longrightarrow}\  \text{H}^\cdot_{\text{T}}(\overline{X})$$ 
become isomorphisms of $\text{H}^\cdot(\text{BT})$ modules after localising at the maximal ideal of $\text{H}^\cdot(\text{BT})$.
\end{theorem}
\begin{proof}
Similar to Theorem \ref{smoothabthm}'s proof. The proof of a) and b) carry over verbatim. We only need to give an analogue of specialisation \ref{comparison}, which should compare $X$ and $\mathbf{N}$ in a neighbourhood of $\overline{X}$. There will certainly not be an isomorphism of bivariant homology groups as in (\ref{eqn6}), instead there is a \textit{localised specialisation} map, an isomorphism of \textit{localised} bivariant homologies
$$\text{H}^\cdot(\mathfrak{sp}_{loc})\ :\  \text{H}^\cdot_{\mathbf{G}_m}(\overline{X}/X)_{loc}\ \stackrel{\sim}{\longrightarrow}\text{H}^\cdot_{\mathbf{G}_m}(\overline{X}/\mathbf{N})_{loc}.$$ 
This interlaces the Euler classes as before
\begin{center}
\begin{tikzcd}
\text{H}^\cdot_{\mathbf{G}_m}(\overline{X}/X)_{loc}\arrow[r,"\cdot e(X)"]\arrow[d,"\text{H}^\cdot(\mathfrak{sp}_{loc})"]\arrow[opacity=0]{d}[opacity=1,swap]{\sim}& \text{H}_{\mathbf{G}_m}^*(\overline{X})_{loc}\arrow[d,equals]\\
\text{H}^\cdot_{\mathbf{G}_m}(\overline{X}/\mathbf{N})_{loc}\arrow[r,"\cdot e(\mathbf{N})"]& \text{H}_{\mathbf{G}_m}^*(\overline{X})_{loc}
\end{tikzcd}
\end{center}
and so we can repeat the rest of the proof of Theorem \ref{smoothabthm} verbatim.
\end{proof}

\subsection{Relative setting} The above theorems \ref{smoothabthm} and \ref{singabthm} hold in a limited way for (dg) Artin stacks: those which admit a map to some base $B$, with fibres (dg) schemes as above.

More precisely, we can consider a closed embedding between Artin stacks over some base $B$ whose structure maps to $B$ are representable:
$$i\ :\ \overline{X}\ \longrightarrow\ X.$$
Moreover, letting $\text{T}$ be a trivial group scheme over $B$ with fibre a split torus, assume that $\text{T}$ acts on the above stacks, trivially on $\overline{X}$, such that the closed embedding $i$ is equivariant. 

Finally, we would like to require that the action of $\text{T}$ on $X\setminus \overline{X}$ is free, so that its equivariant cohomology is a torsion. Since the notion of free actions on stacks can be subtle, we sidestep this issue and simply require that the torus action on the fibres of $X\setminus \overline{X}\to B$, which are schemes, are free. Then by section \ref{torsrel} it follows that the equivariant cohomology $\text{H}^\cdot_{\text{T}}(X\setminus \overline{X})$ is a torsion $\text{H}^\cdot(\text{BT})$ module. The proofs of theorems \ref{smoothabthm} and \ref{singabthm} then proceed exactly as above.

\subsection{} \label{singint} As before we can deduce an ``integration formula'' from the above
$$\text{id}\ =\ i_*\frac{i^*(-)}{e(X)}$$
where the right side is viewed as a sequence of maps
$$\text{H}^\cdot_{\text{T}}(\overline{X})_{loc}\ \stackrel{\sim}{\longrightarrow}\ \text{H}^\cdot_{\text{T}}(X)_{loc} \ \stackrel{\sim}{\longrightarrow}\ \text{H}^\cdot_{\text{T}}(\overline{X}/X)_{loc}\ \stackrel{\sim}{\longrightarrow}\ \text{H}^\cdot_{\text{T}}(X)_{loc}.$$
In practice this will not be as useful because it involves bivariant homology, which is harder to work with than ordinary cohomology.

If we want to get an integration formula for honest cohomology classes, we need to be in a situation where we can integrate cohomology classes over $\overline{X}$ and $X$. We can integrate over a map if it is quasismooth, so admits a fundamental class, and proper. Thus let us consider
\begin{equation}\label{eqn5}
\begin{tikzcd}
&\overline{X}\arrow[d,"i"]\arrow[rdd,bend right = -30, "\overline{f}"]&\\
&X\arrow[rd,"f"]&\\
&& Y
\end{tikzcd}
\end{equation}
where the rightwards maps are both quasismooth and proper. Assume moreover that this diagram is $\text{T}$ equivariant for the trivial action on $Y$. It follows that the normal complex $\mathbf{N}_i$ of the closed embedding $i$ is concentrated in degrees $0$ and $1$, so defines an actual (non-bivariant) cohomology class, which we abusively call
$$e(\mathbf{N}_i)\ \in\ \text{H}^\cdot_{\text{T}}(\overline{X})_{loc}.$$
It is by definition the product of the bivariant Euler class with the fundamental class of $\mathbf{N}_i\to \overline{X}$ which exists because the map is quasismooth, precisely because .

\begin{theorem} \label{intformula} For any $\alpha\in \text{H}^\cdot_{\text{T}}(X)$, we have the \textit{integration formula}
\begin{equation}\label{eqn10}
f_*\alpha\ =\ \overline{f}_*\frac{i^*\alpha}{e(\mathbf{N}_i)}.
\end{equation} 
\end{theorem}

\begin{proof}To avoid the possibility of confusion, we spell out (\ref{eqn10}) in a commutative diagram:
\begin{equation}\label{eqn11}
\begin{tikzcd}
\text{H}^\cdot_{\text{T}}(\overline{X})_{loc}\arrow[r,"1/e(\mathbf{N}_i)"]&\text{H}^{\cdot -2c}_{\text{T}}(\overline{X})_{loc}\arrow[r,"\overline{f}_*"]& \text{H}^{\cdot -2d}_{\text{T}}(Y)_{loc}\\  
\text{H}^\cdot_{\text{T}}(\overline{X})\arrow[u,hook]&\text{H}^\cdot_{\text{T}}(X)\arrow[r,"f_*"]\arrow[l,swap,"i^*"]&\text{H}^{\cdot -2d}_{\text{T}}(Y)\arrow[u,hook]
\end{tikzcd}
\end{equation}
The vertical maps are injections because the action on $\overline{X}$ and $Y$ is trivial. To build diagram (\ref{eqn11}), we begin with the commuting diagram
\begin{center}
\begin{tikzcd}
&\text{H}^\cdot_{\text{T}}(\overline{X}/X)_{loc}\arrow[r,"\cdot 1_{X/Y}"]\arrow[opacity=0]{ld}[opacity=1]{\sim} \arrow[ld, "\cdot e(X)",swap]\arrow[opacity=0]{d}[opacity=1]{\sim}\arrow[d,swap,"i_*"]&\text{H}^{\cdot -2d}_{\text{T}}(\overline{X}/Y)_{loc}\arrow[d,"i_*",swap]\arrow[rd,"\overline{f}_*"]& \\

\text{H}^\cdot_{\text{T}}(\overline{X})_{loc}&\text{H}^\cdot_{\text{T}}(X)_{loc}\arrow[r,"\cdot 1_{X/Y}"]\arrow[opacity=0]{l}[opacity=1]{\sim}\arrow[l,swap,"i^*"]&\text{H}^\cdot_{\text{T}}(X/Y)_{loc}\arrow[r,"f_*"]& \text{H}^{\cdot -2d}_{\text{T}}(Y)_{loc}\\

\text{H}^\cdot_{\text{T}}(\overline{X})\arrow[u]&\text{H}^\cdot_{\text{T}}(X)\arrow[r,"\cdot 1_{X/Y}"]\arrow[l,swap,"i^*"]\arrow[u]&\text{H}^\cdot_{\text{T}}(X/Y)\arrow[r,"f_*"]\arrow[u]& \text{H}^{\cdot -2d}_{\text{T}}(Y)\arrow[u]
\end{tikzcd}
\end{center}
The bottom row of vertical arrows are the inclusions of $\text{H}^\cdot(\text{B}\text{T})$ modules into their localisations. This is almost (\ref{eqn11}), except that it is in terms of the bivariant Euler class $e(X)$ rather than the ordinary localised cohomology class $e(\mathbf{N}_i)$ defined above. So to finish we need a comparison 

\begin{center}
\begin{tikzcd}
\text{H}^\cdot_{\text{T}}(\overline{X})_{loc}\arrow[d,equals]& \text{H}^{\cdot }_{\text{T}}(\overline{X})_{loc}\arrow[opacity=0]{d}[opacity=1,swap]{\sim}\arrow[l,"\cdot e(\mathbf{N}_i)",swap]\arrow[r,"\cdot 1_{\overline{X}/Y}"]& \text{H}^{\cdot -2d}_{\text{T}}(\overline{X}/Y)_{loc}\arrow[d,equals] \\

\text{H}^\cdot_{\text{T}}(\overline{X})_{loc}\arrow[d,equals]& \text{H}^{\cdot }_{\text{T}}(\overline{X}/\mathbf{N}_i)_{loc}\arrow[d,"\mathfrak{sp}_i^{loc}"]\arrow[opacity=0]{d}[opacity=1,swap]{\sim}\arrow[l,"\cdot e(\mathbf{N}_i)",swap]\arrow[r,"\cdot 1_{\mathbf{N}_i/Y}"] \arrow[u, "\cdot 1_{\mathbf{N}_i/\overline{X}}",swap] \arrow[opacity=0]{u}[opacity=1]{\sim} & \text{H}^{\cdot -2d}_{\text{T}}(\overline{X}/Y)_{loc}\arrow[d,equals] \\

\text{H}^\cdot_{\text{T}}(\overline{X})_{loc}& \text{H}^{\cdot }_{\text{T}}(\overline{X}/X)_{loc}\arrow[r,"\cdot 1_{X/Y}"]\arrow[l,swap,"\cdot e(X)"]& \text{H}^{\cdot -2d}_{\text{T}}(\overline{X}/Y)_{loc}
\end{tikzcd}
\end{center}
The bottom row commutes by properties of the equivariant specialisation map (see section \ref{eqvtloc}), the top left square is the definition of the localised cohomology class $e(\mathbf{N}_i)\in \text{H}^\cdot_{\text{T}}(\overline{X})_{loc}$, and the top right square commutes because the fundamental class of a composition of two quasismooth maps is the product of their fundamental classes.
\end{proof}

\subsection{CoHA products} \label{cohaproducts} We can use the integration formula to compute CoHA-like products on \textit{non-equivariant} cohomology, like so. Take a $\text{T}$ equivariant diagram of spaces and the action on all the spaces except $X$ being trivial:
\begin{center}
\begin{tikzcd}
&\overline{X}\arrow[d,"i"]\arrow[rdd,bend right = -30, "\overline{f}"]\arrow[ldd,bend right = 30, "\overline{g}",swap]&&\\
&X\arrow[rd,"f"]\arrow[ld,"g",swap]&&\\
Z&& Y
\end{tikzcd}
\end{center}
$i$ being a closed embedding and $f$ (therefore also $\overline{f}$) proper. Assume that $f$ and $\overline{f}$ are quasismooth, the simplest instance of which being if all the spaces are smooth. The right side of this diagram is then as in section \ref{singint}, and 
\begin{prop} \label{cohaproposition}As maps from $\text{H}^\cdot(Z)$ to $\text{H}^\cdot(Y)$,
\begin{equation}\label{cohaprop}
f_*g^*\ =\ \overline{f}_*\frac{\overline{g}^*}{e(\mathbf{N}_i)}.
\end{equation}
\end{prop}
Again, to avoid confusion we should spell out where (\ref{cohaprop}) is taking place. Because the actions on $Z$ and $Y$ are trivial, there is a map
\begin{equation}\label{mapsinclusion}
\text{Maps}(\text{H}^\cdot_{\text{T}}(Z)_{loc},\text{H}^\cdot_{\text{T}}(Y)_{loc})\ \longrightarrow\ \text{Maps}(\text{H}^\cdot(Z),\text{H}^{\cdot}(Y))
\end{equation}
given by 
\begin{equation}\label{mapsinclusiondiag}
\begin{tikzcd}
\text{H}^\cdot_{\text{T}}(Z)_{loc}\arrow[r]& \text{H}^\cdot_{\text{T}}(Y)_{loc}\arrow[d]\\
\text{H}^\cdot(Z)\arrow[r,dashed]\arrow[u] & \text{H}^\cdot(Y)
\end{tikzcd}
\end{equation}
The upwards arrow corresponds to the map $k\to k(t)$ sending $1\mapsto t^0$, and the downwards arrow to the map $k(t)\to k$ picking out the $t^0$ coefficient. The above proposition then says that the image of $\overline{f}_*\frac{\overline{g}^*}{e(\mathbf{N}_i)}$  under the map (\ref{mapsinclusion}) is $f_*g^*$.

\begin{proof} The equation (\ref{cohaprop}) is true as maps $\text{H}^\cdot_{\text{T}}(Z)_{loc}\to \text{H}^\cdot_{\text{T}}(Y)_{loc}$ by the integration formula. Thus we are left to show that the image under (\ref{mapsinclusion}) of $f_*g^*$, viewed as a map of localised equivariant cohomology groups, is $f_*g^*$. To do this we expand (\ref{mapsinclusiondiag}) in this case 
\begin{equation}\label{diag6}
\begin{tikzcd}
\text{H}^\cdot_{\text{T}}(Z)_{loc}\arrow[r,"g^*"]& \text{H}^\cdot_{\text{T}}(X)_{loc}\arrow[r,"f_*"]& \text{H}^{\cdot-2d}_{\text{T}}(Y)_{loc}\arrow[dd,bend left = 40, dashed, shift left = 2]\\
\text{H}^\cdot_{\text{T}}(Z)\arrow[r,"g^*"]\arrow[d,twoheadrightarrow]\arrow[u]& \text{H}^\cdot_{\text{T}}(X)\arrow[r,"f_*"]\arrow[d]\arrow[u]& \text{H}^{\cdot-2d}_{\text{T}}(Y)\arrow[d,twoheadrightarrow]\arrow[u]\\
\text{H}^\cdot(Z)\arrow[r,"g^*"]\arrow[u,bend left = 40, dashed, shift left = 2] & \text{H}^\cdot(X)\arrow[r,"f_*"]& \text{H}^{\cdot-2d}(Y)
\end{tikzcd}
\end{equation}
and show that it commutes. The bottom two squares commute because they come from the diagram of spaces
\begin{center}
\begin{tikzcd}
 Z\arrow[d] & X\arrow[l,"g",swap]\arrow[d]\arrow[r,"f"] & Y\arrow[d]\\
 Z/\text{T} & X/\text{T} \arrow[l,"g",swap]\arrow[r,"f"]  & Y/\text{T}
\end{tikzcd}
\end{center}
and the right square is a fibre product. The top two squares commute because localisation of modules is functorial. The upwards dotted arrow corresponds to the map $k\to k[t]$ sending $1\mapsto 1$, and the downwards to $k(t)\to k$ picking out the $t^0$ coefficient. Thus (\ref{diag6}) is indeed of the form (\ref{mapsinclusiondiag}) and commutes, proving the proposition.
\end{proof}

\subsection{Functorial integration formula} \label{funct} For completeness we include an account of the functorial version of abelian localisation. We will only discuss the smooth case, but the singular case can be dealt with as above.

Consider a commutative diagram of smooth schemes
\begin{center}
\begin{tikzcd}
\overline{X}\arrow[r,"\overline{\pi}"]\arrow[d,"\overline{\iota}"]& \overline{Y}\arrow[d,"i"]\\
X\arrow[r,"\pi"]&Y
\end{tikzcd}
\end{center}
where the vertical arrows are $\text{T}$ equivariant closed embeddings on whose complement the action is free, and the horizontal arrows are proper. The we have

\begin{prop}\label{functprop} In the above setting, 
$$\frac{i^*\pi_*(-)}{e(\mathbf{N}_i)}\ =\ \overline{\pi}_*\left(\frac{\overline{\iota}^*(-)}{e(\mathbf{N}_{\overline{\iota}})}\right)$$
as maps from $\text{H}^\cdot_{\text{T}}(X)_{loc}$ to $\text{H}^\cdot_{\text{T}}(\overline{Y})_{loc}$. 
\end{prop}

\begin{proof}[Proof sketch]
For clarity we spell out both sides of the equation in a diagram
\begin{center}
\small{
\begin{tikzcd}
\text{H}^\cdot_{\text{T}}(X)_{loc}\arrow[d,equals] \arrow[r,"1_{X/Y}"] &  \text{H}^\cdot_{\text{T}}(X/Y)_{loc}\arrow[r,"\pi_*"]& \text{H}^\cdot_{\text{T}}(Y)_{loc}\arrow[r,"i^*"] & \text{H}^\cdot_{\text{T}}(\overline{Y})_{loc} & \text{H}^\cdot_{\text{T}}(\overline{Y}/Y)_{loc}\arrow[d,equals]\arrow[l,"\cdot e(Y)",swap]\arrow[opacity=0]{l}[opacity=1]{\sim}\\

\text{H}^\cdot_{\text{T}}(X)_{loc}\arrow[r,"\overline{\iota}^*"]& \text{H}^\cdot_{\text{T}}(\overline{X})_{loc} & \text{H}^\cdot_{\text{T}}(\overline{X}/X)_{loc}\arrow[l,"\cdot e(X)",swap]\arrow[opacity=0]{l}[opacity=1]{\sim} \arrow[r,"1_{X/Y}"] & \text{H}^\cdot_{\text{T}}(\overline{X}/Y)_{loc}\arrow[r,"\overline{\pi}_*"] & \text{H}^\cdot_{\text{T}}(\overline{Y}/Y)_{loc}
\end{tikzcd}
}
\end{center}
Because $\cdot e(X)=i^*i_*$, this commutes if the following does 
\begin{center}
\begin{tikzcd}
\text{H}^\cdot_{\text{T}}(X)\arrow[d,equals] \arrow[r,"1_{X/Y}"] &  \text{H}^\cdot_{\text{T}}(X/Y)\arrow[r,"\pi_*"]& \text{H}^\cdot_{\text{T}}(Y) &\text{H}^\cdot_{\text{T}}(\overline{Y}/Y)\arrow[d,equals]\arrow[l,"i_*",swap]\\
\text{H}^\cdot_{\text{T}}(X)& \text{H}^\cdot_{\text{T}}(\overline{X}/X)\arrow[l,"\overline{\iota}_*",swap] \arrow[r,"1_{X/Y}"] & \text{H}^\cdot_{\text{T}}(\overline{X}/Y)\arrow[lu,"\overline{\iota}_*",swap]\arrow[r,"\overline{\pi}_*"] & \text{H}^\cdot_{\text{T}}(\overline{Y}/Y)
\end{tikzcd}
\end{center}
In the above the right side trivially commutes. The left side commutes because pushforward and product in bivariant cohomology commute, axiom $A_{12}$ of bivariant homology. To go from the first diagram to the statement of the proposition, repeat the arguments in the proof of the integration theorem.
\end{proof}

\subsubsection{} The functorial integration formula is related to a folk conjecture connecting the Grothendieck Riemann Roch formula and abelian localisation. See \cite{Liu} or \cite{At} for more details. Briefly, the idea is that if $\overline{X},\overline{Y}$ are smooth schemes and
$$\overline{\pi}\ :\ \overline{X}\ \longrightarrow\ \overline{Y}$$
is a proper map between them, this should fit into the formalism of section \ref{funct}, where
$$\pi\ =\ \mathcal{L}\overline{\pi} \ :\ \mathcal{L}\overline{X}\ \longrightarrow\ \mathcal{L}\overline{Y}$$
is the associated map on free loop spaces. These admit a $\mathbf{G}_m$ action given by loop rotation, for which the fixed loci are the space of constant loops $\overline{X}$ and $\overline{Y}$. Then Proposition \ref{functprop} is meant to be the Grothendieck Riemann Roch formula, with the Euler class of $\overline{X}\to \mathcal{L}\overline{X}$ (which is a priori not defined because the spaces are infinite dimensional) replaced with the Todd class of $\overline{X}$.

\section{Sheaves on classifying spaces} \label{shclass} 

\subsection{} This section is a brief note describing the category of sheaves on quotient stacks. It is logically independent of the rest of this paper and can be skipped on a first reading. We focus on classifying spaces and try to write things down very concretely, in order to give as explicit as possible a picture of the objects we are dealing with before localising.

\subsection{} Let $\text{G}$ be any connected algebraic group. We recall from section $7$ of \cite{DGai}

\begin{prop}There is an equivalence of stable $\infty$-categories
\begin{equation}\label{descent} 
\text{Sh}(\text{BG})\ \simeq\ B\text{-mod}
\end{equation}
where the right side is the category of dg modules over the dg algebra $\text{B}=\text{H}^\cdot(\text{G})^\vee$.
\end{prop}
\begin{proof}
The idea is to apply descent to the map $\pi:\text{pt}\to \text{BG}$. Consider 
$$\pi^!\ :\ \text{Sh}(\text{BG})\ \longrightarrow\ \text{Sh}(\text{pt})\ =\ \text{Vect}.$$
Since $\pi^!$ is conservative and $\pi_!,\pi^!$ are continuous functors, the Lurie-Barr-Beck Theorem identifies $\text{Sh}(\text{BG})$ with modules in $\text{Vect}$ for the monad $\pi^!\pi_!$, which is therefore equivalent to $B\text{-mod}$ for $\text{B}=\pi^!\pi_!k$. To compute $\text{B}$ more precisely, apply smooth base change\footnote{We need to use here that the base field is perfect or of characteristic zero, so that $G$ is smooth.} to
\begin{center}
\begin{tikzcd}
\text{G}\arrow[r,"\overline{\pi}"]\arrow[d,"\overline{\pi}"]&\text{pt}\arrow[d,"\pi"]\\
\text{pt}\arrow[r,"\pi"]& \text{BG}
\end{tikzcd}
\end{center}
which gives $\text{B}=\pi^!\pi_!k = \overline{\pi}_!\overline{\pi}^!k=(\overline{\pi}_*\overline{\pi}^*k)^\vee = \text{H}^\cdot(\text{G})^\vee$.
\end{proof}

Under the equivalence (\ref{descent}), the constant sheaf corresponds to the trivial $\text{B}$-module, $k$. In particular, the cohomology functor
$$\text{H}^\cdot(\text{BG},-)\ :\ \text{Sh}(\text{BG})\ \longrightarrow\ \text{Vect}$$
is identified with
$$\text{Hom}_\text{B}(k,-)\ :\ \text{B}\text{-mod}\ \longrightarrow\ \text{Vect}.$$
Moreover, more can be said, since the cohomology of any sheaf is automatically a module over $\text{B}^!=\text{H}^\cdot(\text{BG})$. Thus the cohomology functor lifts to
\begin{equation}\label{eqn15}
\text{B}\text{-mod}\ \longrightarrow\ \text{B}^!\text{-mod}
\end{equation}
which is an equivalence upon restricting to certain subcategories. This is an instance of \textit{Koszul duality}, for the Koszul dual dg algebras $\text{B}$ and $\text{B}^!$. Notice that
$$\text{B}^!\ \simeq\ \text{Hom}_\text{B}(k,k), \ \ \ \text{B}\ \simeq\ \text{Hom}_{\text{B}^!}(k,k).$$
The same story holds for quotient stacks. If $\text{G}$ acts on a scheme $X$,  

\begin{prop}There is an equivalence of stable $\infty$-categories
$$\text{Sh}(X/\text{G})\ \simeq\ \mathcal{B}\text{-mod}_{\text{Sh}(X)}$$
where the right side are modules in $\text{Sh}(X)$ for the constant sheaf of algebras $\mathcal{B}=\text{H}^\cdot(\text{G})^\vee_X$.
\end{prop}
\begin{proof}
Exactly the same proof as before, applied to the pullback
\begin{center}
\begin{tikzcd}
\text{G}\times X\arrow[r,"a"]\arrow[d,"p"]&X\arrow[d,"\pi"]\\
X\arrow[r,"\pi"]& X/\text{G}
\end{tikzcd}
\end{center}
where $a$ and $p$ are the action and projection maps, respectively. Note that $\mathcal{B}=(p_*a^*k)^\vee=\text{H}^\cdot(\text{G})^\vee_X$.
\end{proof}

\subsection{} For the rest of this section we will consider $\text{G}=\mathbf{G}_m$, so that $\text{B}=\text{H}^\cdot(\mathbf{G}_m)^\vee=k[\varepsilon]/\varepsilon^2$ with $d\varepsilon=0$. We use the grading conventions that the differential $d$ has degree $1$ and $\varepsilon$ has degree $-1$.

\subsubsection{} An object of $\text{Sh}(\text{B}\mathbf{G}_m)=k[\varepsilon]/\varepsilon^2\text{-mod}$ is thus represented by a chain complex $\text{M}$ with a square-zero endomorphism $\varepsilon$ satisfying $d\varepsilon+\varepsilon d=0$.  Using a homogeneous basis of $\text{M}$ this can often be drawn in pictures: 
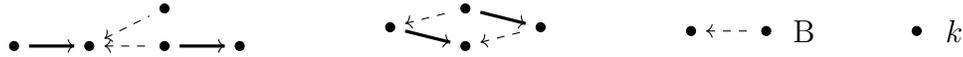
\begin{center}
\begin{tikzpicture}
\draw [fill] (1,0) circle [radius=0.06];
\draw [fill] (2,0) circle [radius=0.06];
\draw [fill] (3,0) circle [radius=0.06];
\draw [fill] (3,0.5) circle [radius=0.06];
\draw [fill] (4,0) circle [radius=0.06];

\draw [fill] (10,0.2) circle [radius=0.06];
\draw [fill] (11,0.2) circle [radius=0.06];
\draw [<-,dashed] (10.2,0.2)--(10.8,0.2);
\node [] at (11.5,0.2) {$\text{B}$};

\draw [fill] (13,0.2) circle [radius=0.06];
\node [] at (13.5,0.2) {$k$};

\draw [->] (1.2,0)--(1.8,0);
\draw [->] (3.2,0)--(3.8,0);
\draw [-,very thick] (1.2,0)--(1.8,0);
\draw [-,very thick] (3.2,0)--(3.8,0);
\draw [<-,dashed] (2.2,0)--(2.8,0);
\draw [<-,dashed] (2.2,0.5*0.2)--(2.8,0.8*0.5);

\draw [fill] (6,0.25) circle [radius=0.06];
\draw [fill] (7,0) circle [radius=0.06];
\draw [fill] (8,0.25) circle [radius=0.06];
\draw [fill] (7,0.5) circle [radius=0.06];

\draw [->] (6.2,0.25*0.8)--(6.8,0.25*0.2);
\draw [-,very thick] (6.2,0.25*0.8)--(6.8,0.25*0.2);
\draw [<-,dashed] (7.2,0.25*0.2)--(7.8,0.25*0.8);
\draw [<-,dashed] (6.2,0.25+0.2*0.25)--(6.8,0.25+0.8*0.25);
\draw [->] (7.2,0.25+0.8*0.25)--(7.8,0.25+0.2*0.25);
\draw [-,very thick] (7.2,0.25+0.8*0.25)--(7.8,0.25+0.2*0.25);
\end{tikzpicture}
\end{center}
Each dot represents a homogenenous basis vector of $\text{M}$ with its horizontal position being its degree. The right solid arrows represent the nonzero action of $d$, and the left dotted arrows the nonzero action of $(-1)^{\deg}\varepsilon$. The functor $\text{Sh}(\text{B}\mathbf{G}_m)\to \text{Vect}$ is conservative, so in particular a map $\text{M}\to \text{N}$ is an equivalence in $\text{Sh}(\text{B}\mathbf{G}_m)$ if and only if it is a quasiisomorphism of chain complexes.

\subsubsection{} In this framework, cohomology is easy to compute. There is a resolution of the constant sheaf by free modules
$$\cdots\ \stackrel{\cdot\varepsilon}{\longrightarrow}\ \text{B}[2]\ \stackrel{\cdot\varepsilon}{\longrightarrow}\ \text{B}[1]\ \stackrel{\cdot\varepsilon}{\longrightarrow}\ \text{B}\ \longrightarrow\ k$$
More precisely, there is a cofibrant resolution in the projective model structure in the model category underlying $\text{Sh}(\text{B}\mathbf{G}_m)$
$$\text{P}\ =\ \text{Tot}(\ \  \cdots\ \stackrel{\cdot\varepsilon}{\longrightarrow}\ \text{B}[2]\ \stackrel{\cdot\varepsilon}{\longrightarrow}\ \text{B}[1]\ \stackrel{\cdot\varepsilon}{\longrightarrow}\ \text{B})\ \longrightarrow\ k$$
where $\text{Tot}$ denotes the single complex associated to a double complex. In pictures, the module $\text{P}$ and the resolution is
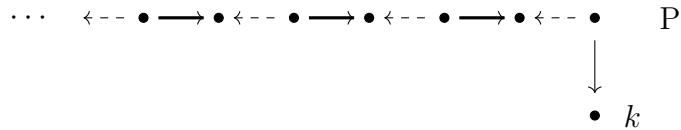
\begin{center}
\begin{tikzpicture}
\draw [fill] (1,0) circle [radius=0.06];
\draw [fill] (2,0) circle [radius=0.06];
\draw [fill] (3,0) circle [radius=0.06];
\draw [fill] (4,0) circle [radius=0.06];
\draw [fill] (5,0) circle [radius=0.06];
\draw [fill] (6,0) circle [radius=0.06];
\draw [fill] (7,0) circle [radius=0.06];
\draw [fill] (7,-1.3) circle [radius=0.06];
\draw [->] (1.2,0)--(1.8,0);
\draw [->] (3.2,0)--(3.8,0);
\draw [->] (5.2,0)--(5.8,0);
\draw [-,very thick] (1.2,0)--(1.8,0);
\draw [-,very thick] (3.2,0)--(3.8,0);
\draw [-,very thick] (5.2,0)--(5.8,0);
\draw [<-,dashed] (0.2,0)--(0.8,0);
\draw [<-,dashed] (2.2,0)--(2.8,0);
\draw [<-,dashed] (4.2,0)--(4.8,0);
\draw [<-,dashed] (6.2,0)--(6.8,0);

\draw [->] (7,-0.3)--(7,-1);

\node [] at (-0.5,0) {$\cdots$};
\node [] at (8,0) {$\text{P}$};
\node [] at (7.5,-1.3) {$k$};
\end{tikzpicture}
\end{center}
In particular, the cohomology groups $\text{H}^n(\text{B}\mathbf{G}_m,\text{M})$ are equal to the chain homotopy classes of degree preserving maps $\text{P}[-n]\to \text{M}$, or equivalently, 
\begin{equation}\label{eqn14}
\text{H}^\cdot(\text{B}\mathbf{G}_m,\text{M})\ =\ \text{H}^\cdot\text{Tot}(\text{M}\stackrel{\cdot\varepsilon}{\longrightarrow}\ \text{M}[-1]\stackrel{\cdot\varepsilon}{\longrightarrow}\ \text{M}[-2]\stackrel{\cdot\varepsilon}{\longrightarrow}\ \cdots\ \ ).
\end{equation}
It is easy therefore to compute 
$$\text{H}^\cdot(\text{B}\mathbf{G}_m,k)\ =\ k[t],\ \ \ \text{H}^\cdot(\text{B}\mathbf{G}_m,\text{B})\ =\ k[1].$$
\subsubsection{} For an example coming from geometry, consider
$$\text{B}\mathbf{G}_m\ \stackrel{i}{\longrightarrow}\ \mathbf{A}^n/\mathbf{G}_m \ \stackrel{j}{\longleftarrow}\ (\mathbf{A}^n\setminus 0)/\mathbf{G}_m\ =\ \mathbf{P}^{n-1}$$
the zero section inside the sum of $n$ tautological line bundles, and its complement. In working with Euler classes we use the Gysin sequence
$$i^!k\ \longrightarrow\ i^*k\ \longrightarrow\ i^*j_*k \ \stackrel{+1}{\longrightarrow}$$
To begin with, we can identify $i^*j_*k$ with the $\text{B}$ module concentrated in degrees $0,1,...,2n-1$:
\begin{center}
\begin{tikzpicture}
\draw [fill] (1*1,0) circle [radius=0.06*1];
\draw [fill] (2*1,0) circle [radius=0.06*1];
\draw [fill] (3*1,0) circle [radius=0.06*1];

\draw [fill] (6*1,0) circle [radius=0.06*1];
\draw [fill] (7*1,0) circle [radius=0.06*1];
\draw [fill] (8*1,0) circle [radius=0.06*1];

\draw [<-,dashed] (1.2*1,0)--(1.8*1,0);
\draw [<-,dashed] (3.2*1,0)--(3.8*1,0);
\draw [->] (2.2*1,0)--(2.8*1,0);
\draw [-,very thick] (2.2*1,0)--(2.8*1,0);

\draw [<-,dashed] (5.2*1,0)--(5.8*1,0);
\draw [->] (6.2*1,0)--(6.8*1,0);
\draw [-,very thick] (6.2*1,0)--(6.8*1,0);
\draw [<-,dashed] (7.2*1,0)--(7.8*1,0);

\node [] at (4.5*1,0) {$\cdots$};

\node [] at (8.75,0) {$\text{M}_n$};
\node [] at (-0,0) {$\textcolor{white}{i^*j_*k}$};

\draw [white] (1*1,-0.2) circle [radius=0.06*1];
\draw [white] (1*1,0.4) circle [radius=0.06*1];
\end{tikzpicture}
\end{center}
Indeed, by Koszul duality (\ref{eqn15}) is is enough to show that its cohomology is $\text{H}^\cdot(\mathbf{P}^{n-1})$ as a $k[t]$ module. But the cohomology of this module can be computed very explicitly, for instance when $n=2$ the double complex on the right side of (\ref{eqn14}) is 
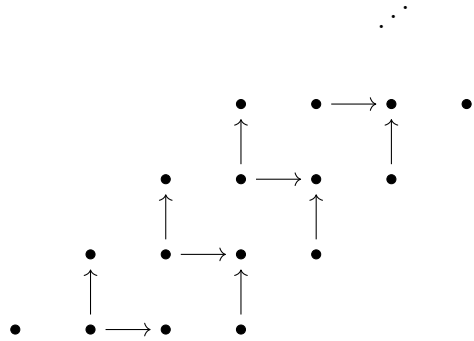
\begin{center}
\begin{tikzpicture}
\draw [white] (1*1,-0.5) circle [radius=0.06*1];

\draw [fill] (1*1,0) circle [radius=0.06*1];
\draw [fill] (2*1,0) circle [radius=0.06*1];
\draw [fill] (3*1,0) circle [radius=0.06*1];
\draw [fill] (4*1,0) circle [radius=0.06*1];

\draw [fill] (1*1+1*1,+1*1) circle [radius=0.06*1];
\draw [fill] (2*1+1*1,+1*1) circle [radius=0.06*1];
\draw [fill] (3*1+1*1,+1*1) circle [radius=0.06*1];
\draw [fill] (4*1+1*1,+1*1) circle [radius=0.06*1];

\draw [fill] (1*1+2*1,+2*1) circle [radius=0.06*1];
\draw [fill] (2*1+2*1,+2*1) circle [radius=0.06*1];
\draw [fill] (3*1+2*1,+2*1) circle [radius=0.06*1];
\draw [fill] (4*1+2*1,+2*1) circle [radius=0.06*1];

\draw [fill] (1*1+3*1,+3*1) circle [radius=0.06*1];
\draw [fill] (2*1+3*1,+3*1) circle [radius=0.06*1];
\draw [fill] (3*1+3*1,+3*1) circle [radius=0.06*1];
\draw [fill] (4*1+3*1,+3*1) circle [radius=0.06*1];

\draw [->] (2.2*1+3*1,+3*1)--(2.8*1+3*1,+3*1);
\draw [->] (2.2*1+2*1,+2*1)--(2.8*1+2*1,+2*1);
\draw [->] (2.2*1+1*1,+1*1)--(2.8*1+1*1,+1*1);
\draw [->] (2.2*1,0)--(2.8*1,0);
\draw [->] (2*1,0.2*1)--(2*1,0.8*1);
\draw [->] (4*1,0.2*1)--(4*1,0.8*1);

\draw [->] (2*1+1*1,0.2*1+1*1)--(2*1+1*1,0.8*1+1*1);
\draw [->] (4*1+1*1,0.2*1+1*1)--(4*1+1*1,0.8*1+1*1);

\draw [->] (2*1+2*1,0.2*1+2*1)--(2*1+2*1,0.8*1+2*1);
\draw [->] (4*1+2*1,0.2*1+2*1)--(4*1+2*1,0.8*1+2*1);

\node [] at (4/0.7+0.3,3/0.7) {$\iddots$};
\end{tikzpicture}
\end{center}
the cohomology of whose associated single complex has two generators, in degree zero and two. We can express those two cohomology classes maps from $\text{P}$ into this module, 
\begin{center}
\begin{tikzpicture}
\draw [fill] (5,0) circle [radius=0.06];
\draw [fill] (6,0) circle [radius=0.06];
\draw [fill] (7,0) circle [radius=0.06];
\draw [->] (5.2,0)--(5.8,0);
\draw [-,very thick] (5.2,0)--(5.8,0);
\draw [<-,dashed] (4.2,0)--(4.8,0);
\draw [<-,dashed] (6.2,0)--(6.8,0);

\draw [-,very thick] (5.2+3,-1.3)--(5.8+3,-1.3);
\draw [->] (5.2+3,-1.3)--(5.8+3,-1.3);
\draw [<-,dashed] (4.2+3,-1.3)--(4.8+3,-1.3);
\draw [<-,dashed] (6.2+3,-1.3)--(6.8+3,-1.3);

\draw [fill] (4+3,-1.3) circle [radius=0.06];
\draw [fill] (5+3,-1.3) circle [radius=0.06];
\draw [fill] (6+3,-1.3) circle [radius=0.06];
\draw [fill] (7+3,-1.3) circle [radius=0.06];

\draw [->] (4+3,-0.3)--(4+3,-1);

\node [] at (3.5,0) {$\cdots$};
\node [] at (8,0) {$\text{P}$};
\node [] at (6,-1.3) {$\text{M}_2$};

\draw [fill] (9+3,0) circle [radius=0.06];
\draw [fill] (9+4,0) circle [radius=0.06];
\draw [fill] (9+5,0) circle [radius=0.06];
\draw [fill] (9+6,0) circle [radius=0.06];
\draw [fill] (9+7,0) circle [radius=0.06];
\draw [->] (9+3.2,0)--(9+3.8,0);
\draw [-,very thick] (9+3.2,0)--(9+3.8,0);
\draw [->] (9+5.2,0)--(9+5.8,0);
\draw [-,very thick] (9+5.2,0)--(9+5.8,0);
\draw [<-,dashed] (9+2.2,0)--(9+2.8,0);
\draw [<-,dashed] (9+4.2,0)--(9+4.8,0);
\draw [<-,dashed] (9+6.2,0)--(9+6.8,0);

\draw [-,very thick] (9+2+4.2,-1.3)--(9+2+4.8,-1.3);
\draw [->] (9+2+4.2,-1.3)--(9+2+4.8,-1.3);
\draw [<-,dashed] (9+2+3.2,-1.3)--(9+2+3.8,-1.3);
\draw [<-,dashed] (9+2+5.2,-1.3)--(9+2+5.8,-1.3);

\draw [fill] (9+2+3,-1.3) circle [radius=0.06];
\draw [fill] (9+2+4,-1.3) circle [radius=0.06];
\draw [fill] (9+2+5,-1.3) circle [radius=0.06];
\draw [fill] (9+2+6,-1.3) circle [radius=0.06];

\draw [->] (9+2+3,-0.3)--(9+2+3,-1);

\node [] at (9+1.5,0) {$\cdots$};
\node [] at (9+8,0) {$\text{P}[-2]$};
\node [] at (9+2+7,-1.3) {$\text{M}_2$};
\end{tikzpicture}
\end{center}
so that composing with $t:\text{P}[-2]\to\text{P}$ takes the first class to the second, therefore $\text{Hom}_\text{B}(k,\text{M}_n)$  is indeed $\text{H}^\cdot(\mathbf{P}^1)=k[t]/t^2$. Notice that composing with $t^2:\text{P}[-4]\to \text{P}$ gives a nullhomotopic map. The same argument works for all $n$. To finish, $i^*k$ is the constant sheaf, and therefore $i^!k$ is the fibre (cocone) of the map $k\to i^*j_*k$ which in pictures is the module concentrated in degrees $0,1,...,2n$:
\begin{center}
\begin{tikzpicture}
\draw [fill] (0*1,0) circle [radius=0.06*1];
\draw [fill] (1*1,0) circle [radius=0.06*1];
\draw [fill] (2*1,0) circle [radius=0.06*1];
\draw [fill] (3*1,0) circle [radius=0.06*1];

\draw [fill] (6*1,0) circle [radius=0.06*1];
\draw [fill] (7*1,0) circle [radius=0.06*1];
\draw [fill] (8*1,0) circle [radius=0.06*1];

\draw [<-,dashed] (1.2*1,0)--(1.8*1,0);
\draw [<-,dashed] (3.2*1,0)--(3.8*1,0);
\draw [->] (2.2*1,0)--(2.8*1,0);
\draw [-,very thick] (2.2*1,0)--(2.8*1,0);
\draw [->] (0.2*1,0)--(0.8*1,0);
\draw [-,very thick] (0.2*1,0)--(0.8*1,0);

\draw [<-,dashed] (5.2*1,0)--(5.8*1,0);
\draw [->] (6.2*1,0)--(6.8*1,0);
\draw [-,very thick] (6.2*1,0)--(6.8*1,0);
\draw [<-,dashed] (7.2*1,0)--(7.8*1,0);

\node [] at (4.5*1,0) {$\cdots$};

\node [] at (8.75,0.05) {$i^!k$};
\node [] at (-0,0) {$\textcolor{white}{i^!k}$};

\draw [white] (1*1,-0.2) circle [radius=0.06*1];
\draw [white] (1*1,0.4) circle [radius=0.06*1];
\end{tikzpicture}
\end{center}
which recovers the expected result $i^!k\simeq k[-2n]$.

\subsection{} Examples involving singular spaces are generally more interesting. For instance, a heuristic we have encountered throughout this paper is that localising the equivariant bivariant cohomology, $\text{H}^\cdot_{\mathbf{G}_m}(X/Y)_{loc}$, forgets the singularities. But the sheaf $f^!k$ (whose cohomology gives bivariant cohomology) will remember this often subtle information.

\subsubsection{} Take for example an affine cone $C$, smooth except at the origin:
$$\text{B}\mathbf{G}_m\ \stackrel{i}{\longrightarrow}\ C/\mathbf{G}_m\ \stackrel{j}{\longleftarrow}\ (C\setminus 0)/\mathbf{G}_m\ =\  \mathbf{P}C.$$
Then $\mathbf{P}C$ is smooth, it can be any smooth projective variety, but there is a potentially complicated singular point at the origin, measured to a certain extent by $i^!k$.

For instance if $C$ is a union of $n$ distinct lines through the origin then 
$$i^!k\ =\ \text{fib}(k\ \longrightarrow\ B[-1]^{n})\ \simeq\ k[-2]\oplus B[-2]^{n-1}.$$
Moving one dimension up, if $\mathbf{P}C$ is a smooth projective curve of genus $g$ then as a $k[t]$ module $\text{H}^\cdot(\mathbf{P}C) =k[t]/t^2\oplus k[-1]^{2g}$. This is because the degree of any projective embedding of such a curve is nonzero, which precisely means that the pullback of $\mathcal{O}(1)$ to $\mathbf{P}C$ has nonvanishing first chern class. From this it follows that
$$i^!k\ =\ \text{fib}(k\ \longrightarrow\ \text{M}_2\oplus B[-1]^{2g})\ \simeq\ k[-4]\oplus B[-2]^{2g}.$$
More examples are easy to construct as soon as one knows the action of $c_1(\mathcal{O}(1))$ on $\mathbf{P}C$. This gives a description of the ``correction term" to the Euler class $i^!k\to i^*k$ being an isomorphism, which vanishes upon localising.

\subsubsection{} The sheaf $i^!k$ is still a coarse measure of the singularity. Take for instance the du Val singularities, which are GIT quotients
$$\mathbf{A}^2//\Gamma\ =\ \Spec \mathbf{C}[x,y]^{\Gamma}$$
for $\Gamma$ be a finite subgroup of $\text{SL}(2,\mathbf{C})$. These correspond to ADE Lie algebras, and each is a hypersurface in $\mathbf{A}^3$, smooth away from the origin. We have
$$\text{B}\mathbf{G}_m\ \stackrel{i}{\longrightarrow}\ (\mathbf{A}^2//\Gamma)/\mathbf{G}_m\ \stackrel{j}{\longleftarrow}\ (\mathbf{A}^2\setminus 0)/(\Gamma\times \mathbf{G}_m) \ =\ \mathbf{P}^1/\Gamma.$$
However, $i^!k=k[-2]$ as in the smooth case. This is because
$$\text{H}^\cdot(\mathbf{P}^1/\Gamma)\ =\ \text{H}^\cdot(\mathbf{P}^1).$$
Indeed, because $\Gamma$ is a finite group $\text{H}^\cdot(\mathbf{P}^1/\Gamma)$ identifies with the invariants $\text{H}^\cdot(\mathbf{P}^1)^\Gamma$. But $\Gamma$, acting as a subgroup of the connected group $\text{SL}(2,\mathbf{C})$, acts trivially on cohomology, giving the above. Thus the sheaf $i^!k$ (and hence bivariant homology) does not see the singularity.

\section{The main theorem}  \label{mainsect}

\subsection{} Let $X$ be the moduli stack classifying finite dimensional representations of a finite quiver or coherent sheaves on a smooth proper curve. On its cohomology there is
\begin{enumerate}
\item[{a)}] a cohomological Hall algebra product $m$ by section \ref{coha}, and
\item[{b)}] a braided super vertex coalgebra  $Y^\vee$ by section \ref{va}. 
\end{enumerate}
In this section we show they are compatible, i.e. that
\begin{equation}\label{diag4}
\begin{tikzcd}
\text{H}^\cdot(X\times X)\arrow[r,"Y^\vee\otimes Y^\vee"]\arrow[d,"m "] &\text{H}^\cdot(X^2\times X^2)((z))\arrow[r,"S(z)\sigma^*"] & \text{H}^\cdot(X^2\times X^2)((z))\arrow[d,"m\otimes m "]\\
\text{H}^\cdot(X)\arrow[rr,"Y^\vee"] && \text{H}^\cdot(X^2)((z))
\end{tikzcd}
\end{equation}
commutes. Here the Yang-Baxter matrix $S(z)$ and the swap map $\sigma$ act on the second and third factors. In the language of section \ref{bialgebra}, the above is our main theorem
\begin{theoremmain*} $\text{H}^\cdot(X)$ is a braided super vertex bialgebra.
\end{theoremmain*} 

The basic idea of the proof, sketched in section \ref{heuristic}, uses \textit{split loci} to compute CoHA products using abelian localisation. Since the spaces involved are in general singular, we must use our above work on abelian localisation. We use this in section \ref{main} to prove the main theorem \ref{mainthm}, and extract explicit formulas for CoHA products in section \ref{explicit}.

\subsection{} \label{heuristicvect} It may be clarifying to first explain the idea of section \ref{heuristic} in the simplest case, the abelian category $\mathcal{A}=\text{Vect}$ of finite dimensional vector spaces.  The moduli space of objects in $\mathcal{A}$ is
$$X \ =\ \coprod_{n\ge 0}\text{BGL}_n$$
and the CoHA product structure on its cohomology is induced by pull-push along
\begin{center}
\begin{tikzcd}
&\text{BP}_{n,m}\arrow[ld,swap]\arrow[rd]&\\
\textcolor{white}{a}\text{BGL}_n\times \text{BGL}_m\textcolor{white}{a}&&
\text{BGL}_{n+m}\\
\end{tikzcd}
\end{center}
Recall that a map into $\text{BGL}_k$ corresponds to a rank $k$ vector bundle $V_k$ and a map into $\text{BP}_{n,m}$ to $V_n\subseteq V_{n+m}$, a rank $n+m$ vector bundle with a chosen subbundle of rank $n$. 

Actually no information is lost in postcomposing by the injection $\oplus^*$
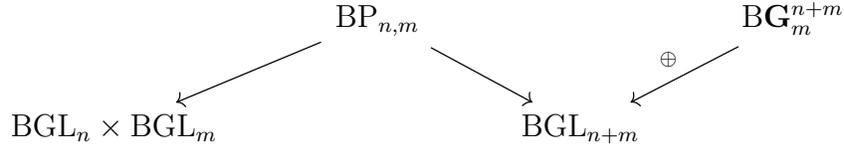
\begin{center}
\begin{tikzcd}
&\text{BP}_{n,m} \arrow[ld]\arrow[rd] && \text{B}\mathbf{G}_m^{n+m}\arrow[ld,"\oplus",swap]\\
\textcolor{white}{a}\text{BGL}_n\times \text{BGL}_m\textcolor{white}{a}&&\text{BGL}_{n+m}&\\
\end{tikzcd}
\end{center}
Here $\oplus$ sends an $n$ tuple of line bundles to its direct sum, and $\oplus^*$ identifies $\text{H}^\cdot(\text{BGL}_k)$ with the $\mathfrak{S}_k$ invariants inside $\text{H}^\cdot((\text{B}\mathbf{G}_m)^k)=k[t]^{\otimes k}$. 

Next, take the pullback
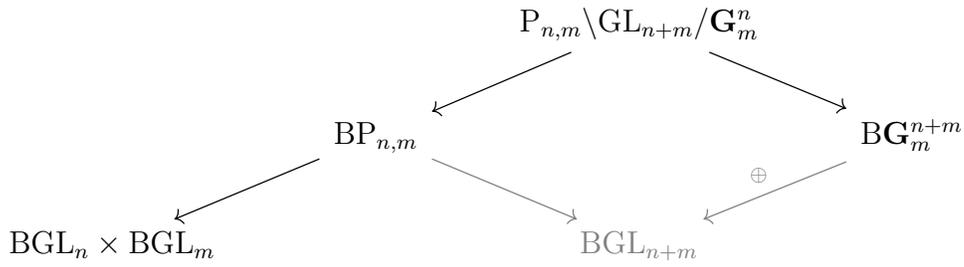
\begin{center}
\begin{tikzcd}
& & \text{P}_{n,m}\backslash\text{GL}_{n+m}/\mathbf{G}_m^n \arrow[rd]\arrow[ld,swap]&&\\
&\text{BP}_{n,m} \arrow[ld,swap]\arrow[rd,gray]&& \text{B}\mathbf{G}_m^{n+m}\arrow[ld,"\textcolor{gray}{\oplus}",swap,gray]\\
\textcolor{white}{a}\text{BGL}_n\times \text{BGL}_m\textcolor{white}{a}&&\textcolor{gray}{\text{BGL}_{n+m}}&\\
\end{tikzcd}
\end{center}
Our composition (pull-push-pull) is equal to pull-pull-push along the upper maps. The above fibre product is a smooth stack classifying $n+m$ tuples of line bundles with a chosen rank $n$ subbundle:
$$V_n \ \stackrel{\iota}{\hookrightarrow} \ L_1\oplus\cdots \oplus L_{n+m}.$$
But there is now an obvious action of the rank $n+m$ torus on this space! It scales the entries of $\iota$, and its fixed locus classifies the above but for  \textit{split} subbundles
$$L_{k_1}\oplus \cdots \oplus L_{k_n} \ \hookrightarrow \ L_1\oplus\cdots \oplus L_{n+m}.$$
Having fixed an $n$ element subset $\sigma\subseteq [n+m]=\{1,...,n+m\}$, this stack is simply $\text{B}\mathbf{G}_m^{n+m}$, and so the fixed locus $i$ is simply a union of these:
\begin{center}
\begin{tikzcd}
&&\coprod\text{B}\mathbf{G}_m^{n+m}\arrow[d,"\textcolor{gray}{i}",gray]\arrow[rdd,bend right = -25]\arrow[lldd,bend right = 30,swap]&&\\
&&\textcolor{gray}{\text{P}_{n,m}\backslash\text{GL}_{n+m}/\mathbf{G}_m^n}\arrow[rd,gray]\arrow[lldd,swap,gray]&&\\
\textcolor{white}{a}\text{B}\mathbf{G}_m^n\times \text{B}\mathbf{G}_m^m \arrow[d]&\textcolor{white}{\text{Ext}}&& \text{B}\mathbf{G}_m^{n+m}\arrow[d]\\
\textcolor{white}{a}\text{BGL}_n\times \text{BGL}_m\textcolor{white}{a}&&&\text{BGL}_{n+m}\\
\end{tikzcd}
\end{center}
The top left and right arrows are extremely simple: they are just the identity on each component. So by abelian localisation, pull-push along the top maps is just
$$k[s_1,...,s_n]\otimes k[t_1,...,t_m] \ \longrightarrow \ k[u_1,..., u_{n+m}]$$
\begin{equation}\label{eqn16}
(f(s_1,...,s_n),g(t_1,...,t_m)) \ \mapsto \  \sum_{\sigma} \frac{1}{e(N_i)}f(u_{\sigma 1},...,u_{\sigma n}) g(u_{\sigma 1},...,u_{\sigma m}).
\end{equation}
We view every size $n$ subset $\sigma\subseteq [n+m]$ as a pair of jointly surjective order-preserving maps $[n],[m]\to [n+m]$, and $e(N_i)$ denotes the Euler class of the normal bundle of $i$, which one can show is
$$e(N_i) \ = \ \prod_{i\in [n]} \prod_{j\in [m]} (u_{\sigma i} - u_{\sigma j}).$$
Restricting to symmetric group invariants gives the above formula for the CoHA product
$$\text{H}^\cdot(\text{BGL}_n)\otimes \text{H}^\cdot(\text{BGL}_m)\ \longrightarrow \ \text{H}^\cdot(\text{BGL}_{n+m}).$$
This is essentially the only example where all involved spaces are smooth, which meant we could apply the classical form of abelian localisation without implicitly using our results.

\subsection{Heuristic to compute CoHA products}  \label{heuristic} \label{explicitsketch}

We recall the CoHA product on $\text{H}^\cdot(X)$ is defined by push-pull $\underline{p}_*a^*$ along the correspondence
\begin{center}
\begin{tikzcd}
&\text{Ext}\arrow[ld,"a",swap]\arrow[rd,"\underline{p}"]&\textcolor{white}{\text{Ext}\times_X X^s}\\
\textcolor{white}{a}X\times X\textcolor{white}{a}&&X\\
\end{tikzcd}
\end{center}
The idea is to use the \textit{split locus}
$$\underline{q}\ :\ X^s\ \longrightarrow\ X.$$
What this is exactly will depend on the details of the abelian category and the situation; one possibility is to consider direct sums of rank one objects. An important point is that it should give an injection on cohomology. We then get a diagram
\begin{center}
\begin{tikzcd}
&&\text{Ext}^s\arrow[d,"i"]\arrow[rdd,bend right = -30, "\overline{p}"]\arrow[llddd,bend right = 40, "\overline{q}",swap]&&\\
&&\text{Ext}\times_X X^s\arrow[rd,"p"]\arrow[ld,"q",swap]&&\\
&\text{Ext}\arrow[ld,"a",swap]\arrow[rd,"\underline{p}"]&& X^s\arrow[ld,"\underline{q}",swap]\\
\textcolor{white}{a}X\times X\textcolor{white}{a}&&X&\\
\end{tikzcd}
\end{center}
and because $\underline{q}^*$ is an injection, it is enough for us to compute $\underline{q}^*\underline{p}_*a^* = p_*q^*a^*$. To continue, notice that $\text{Ext}\times_XX^s$ parametrises exact sequences with a splitting of the middle term
$$0\ \longrightarrow\ \mathcal{E}_1\ \stackrel{\alpha}{\longrightarrow}\ \textstyle{\bigoplus} \mathcal{E}_{2,i}\ \stackrel{\beta}{\longrightarrow}\ \mathcal{E}_3\ \longrightarrow\ 0.$$
This admits an action of the torus $T=\mathbf{G}_m^n$ given by multiplying the middle term, with a torus element $\tau \in T$ acting on the above by 
$$0\ \longrightarrow\ \mathcal{E}_1\ \stackrel{\tau\alpha}{\longrightarrow}\ \textstyle{\bigoplus} \mathcal{E}_{2,i}\ \stackrel{\beta \tau^{-1}}{\longrightarrow}\ \mathcal{E}_3\ \longrightarrow\ 0.$$
This splits as a direct sum of exact sequences if and only if it is fixed under the $T$ action; write $\text{Ext}^s$ for the fixed locus. Moreover, all the maps in
\begin{center}
\begin{tikzcd}
&&\text{Ext}^s\arrow[d,"i"]\arrow[rdd,bend right = -30, "\overline{p}"]\arrow[llddd,bend right = 40, "\overline{q}",swap]&&\\
&&\text{Ext}\times_X X^s\arrow[rd,"p"]\arrow[lldd,"aq",swap]&&\\
&\textcolor{white}{\text{Ext}}&& X^s\\
\textcolor{white}{a}X\times X\textcolor{white}{a}&&&\\
\end{tikzcd}
\end{center}
are $T$ equivariant where the other spaces are endowed with trivial $T$ actions. Thus we are in a place to apply the integration formula (\ref{cohaprop}), which says that
$$\underline{q}^*\underline{p}_*a^*\ =\ p_*q^*a^* \ =\ \overline{p}_*\frac{\overline{q}^*}{e(\mathbf{N}_i)}.$$
The point is that $\overline{p}_*$ is usually far simpler to work out than $p_*$, is approximately just a product of a number of rank one $\underline{p}$'s. Moreover, the K theory class of the normal complex $\mathbf{N}_i$ to is seen to be
$$\mathbf{N}_i\ =\ -\mathbf{T}_i\ =\ -\mathbf{T}_{\overline{p}}+i^*\mathbf{T}_{p}\ =\ -\mathbf{T}_{\overline{p}}+i^*q^*\mathbf{T}_{\underline{p}}.$$
In examples $\mathbf{T}_{\overline{p}}$ is fairly trivial, and $\mathbf{T}_{\underline{p}}=a^*\theta$, so the Euler class $e(\mathbf{N}_i)$ is easy enough to compute also.

\subsubsection{} In practice we usually instead use a generic cocharacter $\mathbf{G}_m\to T$ and consider the induced $\mathbf{G}_m$ action, which has the same fixed locus.

\subsubsection{}\label{localisedcoha} We describe the maps $\overline{q}$ and $\overline{p}$ more precisely. The space $\text{Ext}^s$ classifies short exact sequences
$$0\ \longrightarrow\ \textstyle{\bigoplus}\mathcal{E}_{1,i}\ \stackrel{\oplus\alpha_i}{\longrightarrow}\ \textstyle{\bigoplus} \mathcal{E}_{2,i}\ \stackrel{\oplus\beta_i}{\longrightarrow}\ \textstyle{\bigoplus}\mathcal{E}_{3,i}\ \longrightarrow\ 0.$$
The map $\overline{p}$ sends this sequence to its middle term, and $\overline{q}$ sends it to the two outer terms. In particular the map $\overline{q}$ lifts 
 \begin{center}
\begin{tikzcd}
X^s\times X^s\arrow[d]&&\text{Ext}^s\arrow[lld,bend right = 25, "\overline{q}",swap]\arrow[ll,bend right = 25,swap,"\widetilde{q}",dashed]& \\
\textcolor{white}{a}X\times X\textcolor{white}{a}&&&\\
\end{tikzcd}
\end{center}
This means the usual CoHA map lifts to a \textit{localised product}
$$\text{H}^\cdot(X^s)\otimes \text{H}^\cdot(X^s)\ \longrightarrow\ \text{H}^\cdot(X^s)(t)$$
where restricting to $\text{H}^\cdot(X)$ eliminates the dependence on $t$. Compare with Theorem $2$ of \cite{KS}, and the localised coproduct in \cite{Da}. The right hand side is $\text{H}^\cdot(X^s)\otimes \text{Frac}\text{H}^\cdot(\text{B}\mathbf{G}_m)$.

\subsection{The main theorem}  \label{main} We are now in a place to prove our main result.

\begin{theorem} \label{mainthm}
The CoHA product and braided vertex coalgebra structures on $\text{H}^\cdot(X)$ form a braided vertex algebra.
\end{theorem}
\begin{proof}To do this we use the split locus of objects which are a direct sum of two subobjects:
$$\underline{q} \ : \ X^2 \ \longrightarrow \ X,$$
and apply the method of section \ref{heuristic}. We have a diagram
\begin{equation}\label{diag2}
\begin{tikzcd}
&&\text{Ext}^2\arrow[d,"i"]\arrow[rdd,bend right = -30, "\overline{p}"]\arrow[llddd,bend right = 40, "\overline{q}",swap]&&\\
&&\text{Ext}\times_X X^2\arrow[rd,"p"]\arrow[ld,"q",swap]&&\\
&\text{Ext}\arrow[ld,"a",swap]\arrow[rd,"\underline{p}"]&& X^2\arrow[ld,"\underline{q}",swap]\\
\textcolor{white}{a}X\times X\textcolor{white}{a}&&X&\\
\end{tikzcd}
\end{equation}
and we can use abelian localisation to compute
\begin{equation}\label{eqn5}
\underline{q}^*(\text{CoHA product}) \ = \ \underline{q}^*\underline{p}_*a^*\ =\ \overline{p}_*\frac{\overline{q}^*(-)}{e(\mathbf{N}_i)}.
\end{equation}
Concretely, we have applied Proposition \ref{cohaproposition}, the mild condition that $i$ be homotopically of finite presentation being satisfied in our case.

First we will prove the theorem for the braided vertex algebra structure where all orientations are trivial. Now, to understand diagram (\ref{diag4}) we should first expand it out:
\begin{center}
\begin{tikzcd}
\text{H}^\cdot(X\times X)\arrow[r,"\overline{q}^*"]\arrow[d,"\underline{p}_*a^*"] & \text{H}^\cdot(X^2\times X^2)\arrow[r,"\Psi( e((a^*\theta)^2) \cdot - )"] \textcolor{white}{ll}&\textcolor{white}{ll} \textcolor{white}{l}\text{H}^\cdot(X^2\times X^2)((z))\arrow[r,"S(z)\sigma^*"]& \text{H}^\cdot(X^2\times X^2)((z))\arrow[d," \overline{p}_* "]\\
\text{H}^\cdot(X)\arrow[r,"\underline{q}^*"] &\text{H}^\cdot(X^2)\arrow[rr,"\Psi( e(\theta) \cdot - )"] && \text{H}^\cdot(X^2)((z))
\end{tikzcd}
\end{center}
Throughout we have identified 
$$\text{H}^\cdot(\text{Ext}^2) \ = \ \text{H}^\cdot(X^2\times X^2).$$
We have also written $(a^*\theta)^2=a^*\theta\boxplus a^*\theta$, and denoted $\Psi=act^*$ where $act$ is either the action of $\text{B}\mathbf{G}_m$ on the first factor of $X^2$, or the action on the first and third factors of $X^2\times X^2$. But notice that we can slightly rephrase the above diagram
\begin{equation}\label{diag3}
\begin{tikzcd}
\text{H}^\cdot(X\times X)\arrow[r,"\overline{q}^*"]\arrow[dr,"p_*q^*a^*",swap] & \text{H}^\cdot(\text{Ext}^2)\arrow[r,"\Psi( e((a^*\theta)^2) \cdot - )"]  \textcolor{white}{ll}&\textcolor{white}{ll} \textcolor{white}{l}\text{H}^\cdot(\text{Ext}^2)((z))\arrow[r,"S(z)\sigma^*"]& \text{H}^\cdot(\text{Ext}^2)((z))\arrow[d," \overline{p}_* "]\\
 &\text{H}^\cdot(X^2)\arrow[rr,"\Psi( e(\theta) \cdot - )"] && \text{H}^\cdot(X^2)((z))
\end{tikzcd}
\end{equation}
so that it only involved spaces with trivial torus actions. In particular, the diagram on localised $\mathbf{G}_m$ equivariant cohomology
\begin{center}
\begin{tikzcd}
\text{H}^\cdot_{\mathbf{G}_m}(X\times X)_{loc}\arrow[r,"\overline{q}^*"]\arrow[dr,"p_*q^*a^*",swap] & \text{H}^\cdot_{\mathbf{G}_m}(\text{Ext}^2)_{loc}\arrow[r,"\Psi( e((a^*\theta)^2) \cdot - )"] \arrow[d,dashed,"\alpha"] \textcolor{white}{ll}&\textcolor{white}{ll} \textcolor{white}{l}\text{H}^\cdot_{\mathbf{G}_m}(\text{Ext}^2)_{loc}((z))\arrow[r,"S(z)\sigma^*"]& \text{H}^\cdot_{\mathbf{G}_m}(\text{Ext}^2)_{loc}((z))\arrow[d," \overline{p}_* "]\\
 &\text{H}^\cdot_{\mathbf{G}_m}(X^2)_{loc}\arrow[rr,"\Psi( e(\theta) \cdot - )"] && \text{H}^\cdot_{\mathbf{G}_m}(X^2)_{loc}((z))
\end{tikzcd}
\end{center}
is simply (\ref{diag3})$\otimes k(t)$, so it is enough to show that this commutes. Notice that the arrow
$$\alpha\ =\ \overline{p}_*\frac{(-)}{e(\mathbf{N}_i)}$$
makes the left triangle commute. We are therefore left with showing that
\begin{equation}\label{eqn9}
\Psi \left(e(\overline{p}^*\theta)\cdot \frac{(-)}{e(\mathbf{N}_i)} \right)\ = \ S(z)\cdot \sigma^*\Psi(e((a^*\theta)^2)\cdot -)
\end{equation}
This defines $S(z)$ uniquely, what remains is to compute it. 
\begin{lem}\label{lemmacheck}
The element $S(z)$ defined by (\ref{eqn9}) is precisely our Yang-Baxter matrix of Theorem \ref{joycethm},
$$S(z) \ = \ \Psi(\theta)/\Psi(\sigma^*\theta).$$
\end{lem}
\begin{proof}The direct sum map
$$\oplus \ : \ X^2\times X^2 \ \longrightarrow \ \text{Ext}^2$$
gives an isomorphism on cohomology, so it is enough to work with $X^2\times X^2$. In what follows we will suppress $\oplus$ and $\oplus^*$ from the notation.

First begin with computing the K theory class $[\mathbf{N}_i]=-[\mathbf{T}_i]$. By repeatedly applying 
$$[\mathbf{T}_{X/Z}] \ = \ [\mathbf{T}_{X/Y}] +[\mathbf{T}_{Y/Z}]$$ 
we arrive at
$$[\mathbf{N}_i]\ = \ [i^*q^*\mathbf{T}_{\text{Ext}/X}]-[\mathbf{T}_{\text{Ext}^2/X^2}]\ = \ [\overline{q}^*\theta] - [a^*\theta\boxplus a^*\theta].$$
Now label the connected components of $X^2\times X^2$ by quadruples of components of $X$:
$$X^2\times X^2 \ = \ \coprod (X_{\alpha_1}\times X_{\beta_1})\times (X_{\alpha_2}\times X_{\beta_2}).$$
Denote by $\theta_{\alpha_1, \beta_2}$ for the pullback of $\theta$ under the projection
$$(X_{\alpha_1}\times X_{\beta_1})\times (X_{\alpha_2}\times X_{\beta_2}) \ \longrightarrow \ X_{\alpha_1}\times X_{\beta_2},$$ 
and similarly for other indices. We can now describe each element of (\ref{eqn9}) in turn.
\begin{enumerate}[label = \arabic*)]
\item $\overline{p}$ is the direct sum map
$$\overline{p}\ :\ (X_{\alpha_1}\times X_{\beta_1})\times (X_{\alpha_2}\times X_{\beta_2})\ \longrightarrow\ X_{\alpha_1+\beta_1}\times X_{\alpha_2+\beta_2}$$
and so
$$\overline{p}^*\theta\ =\ \theta_{\alpha_1,\alpha_2}\oplus \theta_{\alpha_1,\beta_2}\oplus \theta_{\beta_1,\alpha_2}\oplus \theta_{\beta_1,\beta_2}$$

\item $\overline{q}$ is another direct sum map 
$$\overline{q}\ :\ (X_{\alpha_1}\times X_{\beta_1})\times (X_{\alpha_2}\times X_{\beta_2})\ \longrightarrow\ X_{\alpha_1+\alpha_2}\times X_{\beta_1+\beta_2}$$
and so
$$\overline{q}^*\theta\ =\ \theta_{\alpha_1,\beta_1}\oplus \theta_{\alpha_1,\beta_2}\oplus \theta_{\alpha_2,\beta_1}\oplus \theta_{\alpha_2,\beta_2}.$$
\item Likewise, $a^*\theta\boxplus a^*\theta = \theta_{\alpha_1,\beta_1}\oplus \theta_{\alpha_2,\beta_2}$.
\end{enumerate}
It follows that $[\mathbf{N}_i]=[\theta_{\alpha_1,\beta_2}]+[ \theta_{\alpha_2,\beta_1}]$. Since all of the above complexes have nonzero $\text{B}\mathbf{G}_m$ weight, $\Psi(e(-))$ of them is a unit. Equation (\ref{eqn9}) is thus equivalent to
$$S(z)\ =\  \Psi(e(\theta_{\beta_1,\alpha_2}))/ \Psi(e(\theta_{\alpha_2,\beta_1})).$$
\end{proof}

\end{proof}

\section{Explicit computations} \label{explicit}

\subsection{} We explained a heuristic for computing CoHA products in section \ref{heuristic}, which was used to prove Theorem \ref{main}. We now use it to write down explicit formulas for CoHA products for representations of a quiver, which replicates the formulas in \cite{KS}, and coherent sheaves on a curve. To do this, we consider the correspondence
\begin{equation}\label{diag14}
\begin{tikzcd}
&\text{Ext}^s\arrow[rd,"\overline{p}"]\arrow[ld,"\widetilde{q}",swap]&\\
X^s\times X^s&&X^s
\end{tikzcd}
\end{equation}
to compute the localised CoHA map 
\begin{equation}\label{eqn17}
\overline{p}_*\frac{\widetilde{q}^*(-)}{e(\mathbf{N}_i)}
\end{equation}
as in section \ref{localisedcoha}. Recall that this is defined on $\text{H}^\cdot(X^s)(t)$, and restricting to $\text{H}^\cdot(X)$ gives the CoHA product.

\subsection{Quiver representations} Take as the split locus the moduli space parametrising direct sums of finitely many rank one representations
$$X^s \ =\ \coprod_{J} \left(\prod_{j\in J}X_1\right),$$
where $X_1$ is the moduli space of rank one representations. The connected components are labelled by finite sets $J$. Similarly, connected components of
$$\text{Ext}^s\ =\ \coprod_{J=J_1\cup J_2} \left(\prod_{j_1\in J_1}\text{Ext}_{1,0}\times \prod_{j_2\in J_2}\text{Ext}_{0,1}\right)$$
are labelled by pairs of finite sets $J_1,J_2$. Here $\text{Ext}_{1,0}$  parametrises extensions of a rank zero object by a rank one object (and vice-versa for $\text{Ext}_{0,1}$). Of course there are no nontrivial such extensions, and so
$$\text{Ext}^s\ =\ \coprod_{J=J_1\cup J_2} \left(\prod_{j_1\in J_1}X_1\times \prod_{j_2\in J_2}X_1\right).$$
The correspondence (\ref{diag14}) is 
\begin{center}
\begin{tikzcd}[bo column sep]
&[35pt]\coprod_{J=J_1\cup J_2} \left(\prod_{j_1\in J_1}X_1\times \prod_{j_2\in J_2}X_1\right)\arrow[rd,"\overline{p}"]\arrow[ld,swap,"\widetilde{q}"]\arrow[opacity=0]{ld}[opacity=1]{\sim}&[35pt]\\
\coprod_{J_1} \left(\prod_{j_1\in J_1}X_1\right)\times \coprod_{J_2} \left(\prod_{j_2\in J_2}X_1\right)&&\coprod_{J} \left(\prod_{j\in J}X_1\right)
\end{tikzcd}
\end{center}
Note that $\widetilde{q}$ is the identity and $\overline{p}$ is the identity on each connected component, so the localised CoHA can be understood purely combinatorially. Writing $V=\text{H}^\cdot(X_1)$, it is
$$\overline{p}_*\ :\ \bigoplus_{J_1, J_2}V^{\otimes|J_1|}\otimes V^{\otimes |J_2|}\ \longrightarrow\ \bigoplus_{J}V^{\otimes |J|}.$$
To finish we note that since $\overline{p}$ has degree zero its cotangent complex vanishes, and so $[\mathbf{N}_i]= [\overline{q}^*\theta]$ which recovers the explicit formula for the CoHA product in Theorem $2$ of \cite{KS}.

\subsection{Coherent sheaves on curves} This case is trickier because there are nontrivial rank zero objects. Take as split locus the moduli space parametrising finite direct sums of rank zero and one coherent sheaves. Its connected components are
$$(\text{Coh}_0^{e_1}\times\cdots\times \text{Coh}_0^{e_n})\times(\text{Coh}_1^{d_1}\times\cdots\times \text{Coh}_1^{d_m})$$
for integers $d_i$ and positive integers $e_i$ (see section \ref{description} for notation). In other words,
$$X^s\ =\ \coprod_{I\to\mathbf{N}}\left(\prod_{i\in I} \text{Coh}_0^{e_i}\right)\times \coprod_{J\to \mathbf{Z}}\left( \prod_{j\in J} \text{Coh}_1^{d_j}\right)$$
where the union is over all finite sets $I,J$ and functions $e:I\to\mathbf{N}$ and $d:J\to\mathbf{Z}$. 

\subsubsection{} Likewise, $\text{Ext}^s$ classifies direct sums of short exact sequences of coherent sheaves whose middle term has rank zero or one. Thus its connected components are products with multiplicity of three types of extension moduli spaces
$$\text{Coh}_{0,0}^{e,e'}, \ \ \text{Coh}_{1,0}^{d,e}, \ \ \text{Coh}_{0,1}^{e,d}.$$
 In formulae this reads
$$\text{Ext}^s\ = \ \coprod_{K\rightrightarrows\mathbf{N}}\left(\prod_{k\in K}\text{Coh}_{0,0}^{e_k,e_k'}\right)\times \coprod_{J=J_1\amalg J_2\to\mathbf{Z}\times\mathbf{N}}\left( \prod_{j_1\in J_1}\text{Coh}_{1,0}^{d_{j_1},e_{j_1}}\times  \prod_{j_2\in J_2}\text{Coh}_{0,1}^{e_{j_2},d_{j_2}}  \right).$$
The first union is over all finite sets $K$ and pairs of functions $e,e':K\to\mathbf{N}$. The second is over all finite sets $J$ with a partition into two $J=J_1\amalg J_2$, and functions $d:J\to \mathbf{Z}$ and $e:J\to\mathbf{N}$.

\subsection{} \label{sect1} We proceed with computing the CoHA product. The correspondence (\ref{diag14}) can be understood in terms of the following three very simple correspondences
\begin{center}
\begin{tikzcd}[bo column sep]
&\text{Coh}_{0,0}^{e,e'}\arrow[rd,"\alpha"]\arrow[swap,"q",ld]& &[-30pt]& &\text{Coh}_{0,1}^{e,d}\arrow[rd,"\beta"]\arrow[swap,"q",ld] && [-30pt]& &\text{Coh}_{1,0}^{d,e}\arrow[rd,"\gamma"]\arrow[swap,"q",ld]&\\
\text{Coh}_0^e\times \text{Coh}_0^{e'}&&\text{Coh}^{e+e'}_0&& \text{Coh}_0^e\times \text{Coh}_1^{d}&&\text{Coh}^{e+d}_1& & \text{Coh}_1^d\times \text{Coh}_0^{e}&&\text{Coh}^{d+e}_1\\
\end{tikzcd}
\end{center}
More precisely, the restriction of the correspondence (\ref{diag14}) to a connected component of $\text{Ext}^s$ can be schematically written as
\begin{center}
\begin{tikzcd}[bo column sep]
&[40pt](\text{Coh}_{0,0}^{?,?})^{|K|}\times (\text{Coh}_{1,0}^{?,?})^{|J_1|}\times (\text{Coh}_{0,1}^{?,?})^{|J_2|}\arrow[rd,"\overline{p}"]\arrow[ldd,"\widetilde{q}",swap]&[40pt] \\
&&  (\text{Coh}_{0}^{?})^{|K|}\times (\text{Coh}_{1}^{?})^{|J_1|}\times (\text{Coh}_{1}^{?})^{|J_2|}\\
\left((\text{Coh}_{0}^{?})^{|K|}\times (\text{Coh}_{1}^{?})^{|J_1|}\times (\text{Coh}_{0}^{?})^{|J_2|}\right) \times \left((\text{Coh}_{0}^{?})^{|K|}\times (\text{Coh}_{0}^{?})^{|J_1|}\times (\text{Coh}_{1}^{?})^{|J_2|}\right) &&\\
\end{tikzcd}
\end{center}
Here we have suppressed the degree of coherent sheaves from the notation, replacing them with $?$'s, so that
$$(\text{Coh}_{1,0}^{?,?})^{|J_1|}\ = \ \prod_{j_1\in J_1} \text{Coh}_{1,0}^{d_{j_1},e_{j_1}},$$
on the left side
$$(\text{Coh}_{1}^{?})^{|J_1|}\ =\ \prod_{j_1\in J_1} \text{Coh}_1^{d_{j_1}},\ \ (\text{Coh}_{0}^{?})^{|J_1|}\ =\ \prod_{j_1\in J_1} \text{Coh}_0^{e_{j_1}},$$
on the right side
$$(\text{Coh}_{1}^{?})^{|J_1|}\ =\ \prod_{j_1\in J_1} \text{Coh}_1^{d_{j_1}+e_{j_1}},$$
and so on.

\subsubsection{} \label{choice} The correspondence (\ref{diag14}) acts on the finite sets attached to the connected components by
\begin{center}
\begin{tikzcd}[bo column sep]
&[25pt](K,J_1,J_2)\arrow[rd,"\overline{p}",|->]\arrow[ld,"\widetilde{q}",|->,swap]&[25pt]\\
((K\amalg J_2, J_1),(K\amalg J_1, J_2))&& (K,J_1\amalg J_2)\\
\end{tikzcd}
\end{center}
In particular, given finite sets 
$$((I_1,J_1),(I_1,J_2))$$
attached to a connected component of $X^s\times X^s$, the map $\widetilde{q}$ depends on a non-canonical choice of decomposition
$$I_1\ =\ K\amalg J_2, \ \ \ I_2\ =\ K\amalg J_1.$$
Recall from section \ref{localisedcoha} that $\overline{q}$ is defined canonically, but not in general $\widetilde{q}$. Thus $\widetilde{q}^*$ and so the localised CoHA map may depend on this choice, however the restriction to $\text{H}^\cdot(X)^{\otimes 2}$ will be independent of this choice.

\subsubsection{} The first element of (\ref{eqn17}) we compute is the Euler class of the normal bundle to $i$. As we have seen in section \ref{mainsect}, as K theory classes
$$[\mathbf{N}_i]\ =\ -[\mathbf{T}_{\overline{p}}]\ +\ [\overline{q}^*\theta].$$
To begin with,
$$\mathbf{T}_{\overline{p}}\ =\ \mathbf{T}_\alpha^{\boxplus |K|}\boxplus \mathbf{T}_\beta^{\boxplus |J_1|}\boxplus\mathbf{T}_\gamma^{\boxplus |J_2|}.$$
 Next, writing $\oplus: X^s\times X^s\to X\times X$, we have
$$\overline{q}^*\theta\ =\ \widetilde{q}^*\oplus^*\theta.$$
Now since $\theta$ respects the direct sum structure on the moduli space (condition (1) in section \ref{general}), we can understand the pullback $\oplus^*\theta$ as
$$\oplus^*\theta \ = \ \bigoplus_{x,y\in K\amalg J_1\amalg J_2}\theta_{x,y}.$$
Here we have projected
$$\left((\text{Coh}_{0}^{?})^{|K|}\times (\text{Coh}_{1}^{?})^{|J_1|}\times (\text{Coh}_{0}^{?})^{|J_2|}\right) \times \left((\text{Coh}_{0}^{?})^{|K|}\times (\text{Coh}_{0}^{?})^{|J_1|}\times (\text{Coh}_{1}^{?})^{|J_2|}\right)\ \to\ \left(\text{Coh}_r^?\right)\times \left(\text{Coh}_{r'}^?\right)$$
onto the $(x,y)$th factor, and $\theta_{x,y}$ is the pullback of $\theta$ by this map. Moreover, notice
$$\bigoplus_{x=y} \widetilde{q}^*\theta_{x,y}\ =\ \mathbf{T}_\alpha^{\boxplus |K|}\boxplus \mathbf{T}_\beta^{\boxplus |J_1|}\boxplus\mathbf{T}_\gamma^{\boxplus |J_2|}\ =\ \mathbf{T}_{\overline{p}}$$
so it follows that
\begin{lem}
$[\mathbf{N}_i] =\bigoplus_{x\ne y} [\widetilde{q}^*\theta_{x,y}]$.
\end{lem}

\subsubsection{} We make a note about the cohomology
$$\text{H}^\cdot(X^s\times X^s)\ =\ \left( \bigoplus_{I_1\to\mathbf{N}} \otimes_{i_1\in I_1} \text{V}_e^{i_1}\otimes \bigoplus_{J_1\to\mathbf{Z}} \otimes_{j_1\in J_1} \text{W}^{j_1}_d\right) \otimes 
\left( \bigoplus_{I_2\to\mathbf{N}} \otimes_{i_2\in I_2} \text{V}^{i_2}_{e'}\otimes \bigoplus_{J_2\to\mathbf{Z}} \otimes_{j_2\in J_2} \text{W}_d^{j_2}\right)$$
where we have used the shorthand  
$$\text{V}^i_e\ =\ \text{H}^\cdot(\text{Coh}^{e_i}_0)\ \simeq \ \text{Sym}^{e_i}(\text{V}_C)\ \ \ \text{ and }\ \ \ \text{W}^j_d\ =\ \text{H}^\cdot(\text{Coh}_1^d)\ \simeq\ \text{Sym}(\text{W}_{C,1}).$$
Recall that $\text{V}_C=\text{H}^\cdot(C\times \text{B}\mathbf{G}_m)$ and $\text{W}_{C,1}$ is the super vector space spanned by the $\text{H}^\cdot(\text{Coh}_d^1)$ coefficients of the Chern classes of the tautological coherent sheaf on $C\times \text{Coh}_d^1$.

\subsubsection{} Recall also that the connected components of $X^s$ are labelled by pairs of finite sets decorated with functions capturing the information about the degree of the sheaves:
$$\mathcal{C}\ =\ \{(I,J, \ I\stackrel{e}{\to}\mathbf{N},\ J\stackrel{d}{\to}\mathbf{Z})\}$$
and so the components of $X^s\times X^s$ are labelled by
$$\mathcal{C} \times\mathcal{C} \ =\ \{(I_1,J_1,I_2,J_2,\ I_1\stackrel{e}{\to}\mathbf{N}, \ I_2\stackrel{e'}{\to}\mathbf{N}, \ J_1\amalg J_2\stackrel{d}{\to}\mathbf{Z})\}.$$

\subsubsection{} The maps $\alpha,\beta$ and $\gamma$ induce pushforward maps, which we call
$$\alpha_*^k\ :\ \text{V}_e^k\otimes \text{V}_{e'}^{k}\ \longrightarrow\ \text{V}_{e+e'}^{k},$$
$$\beta_*^{j_1}\ :\ \text{V}_e^{j_1}\otimes \text{W}_{d}^{j_1}\ \longrightarrow\ \text{W}_{e+d}^{j_1},$$
$$\gamma_*^{j_2}\ :\ \text{W}_d^{j_2}\otimes \text{V}_{e}^{j_2}\ \longrightarrow\ \text{W}_{d+e}^{j_2}.$$
It follows that
$$\overline{p}_*\ =\ \bigotimes_{(k,j_1,j_2)\in (K,J_1,J_2)}\alpha_*^k\otimes \beta_*^{j_1}\otimes \gamma_*^{j_2}.$$

\subsubsection{} Fix a connected component of $X^s\times X^s$
$$c\ =\ (I,J_1,I',J_2,\ I_1\stackrel{e}{\to}\mathbf{N}, \ I_2f\stackrel{e'}{\to}\mathbf{N}, \ J_1\amalg J_2\stackrel{d}{\to}\mathbf{Z})\ \in\ \mathcal{C}\times\mathcal{C}$$ 
To define the localised CoHA map we need to for each $c$ make a non-canonical choice
$$I_1\ \simeq \ K\amalg J_2, \ \ \ I_2\ \simeq\ K\amalg J_1$$
for another finite set $K$. The restriction of the localised CoHA map to $\text{H}(X)^{\otimes 2}$ will be independent of these choices. If no such $K$ exists, the image of $\widetilde{q}$ will not hit this component and the localised CoHA map on this component will vanish. The localised CoHA map then sends $c$ to the connected component 
$$ (K,J_1\amalg J_2,\ K\stackrel{e\vert_K}{\to}\mathbf{N}, \ J_1\amalg J_2\stackrel{d}{\to}\mathbf{Z})\ \in\ \mathcal{C}.$$
We can now state the result.

\begin{theorem}\label{explicitthm} The localised CoHA map on this component of $X^s\times X^s$ 
$$\left( \otimes_{i\in I} \text{V}_e^i\otimes \otimes_{j_1\in J_1} \text{W}^{j_1}_d\right) \otimes 
\left( \otimes_{i'\in I'} \text{V}^{i'}_{e'}\otimes  \otimes_{j_2\in J_2} \text{W}_d^{j_2}\right)\ \longrightarrow\ \left( \otimes_{k\in K} \text{V}_e^k\otimes \otimes_{j\in J_1\cup J_2} \text{W}^{j}_d\right)(t) $$
is given by
$$\eta\ \longmapsto \ \left(\bigotimes_{(k,j_1,j_2)\in (K,J_1,J_2)}\alpha_*^k\otimes \beta_*^{j_1}\otimes \gamma_*^{j_2}\right)\left(\frac{\eta }{ \prod_{x\ne y }e(\theta_{x,y})}\right).$$
The inner product is over elements $x,y\in K\amalg J_1\amalg J_2$.
\end{theorem}

\subsubsection{} The Euler class denominator is trivial when $|K\amalg J_1\amalg J_2|=1$ and the above becomes a genuine (non-localised) algebra product. In this case the localised CoHA correspondence becomes one of the three simple CoHA correspondences in section \ref{sect1}.

\subsection{} We conclude by discussing the pushforwards $\alpha_*$, $\beta_*$ and $\gamma_*$, and their cotangent complexes, which are pullbacks of the complexes $\theta_{x,y}$.

\subsubsection{} \label{coha00} First consider the rank zero correspondence
\begin{center}
\begin{tikzcd}[bo column sep]
&\text{Coh}_{0,0}^{e,e'}\arrow[rd,"\alpha"]\arrow[ld,swap,"q"]& \\
\text{Coh}_0^e\times \text{Coh}_0^{e'}&&\text{Coh}^{e+e'}_0
\end{tikzcd}
\end{center}
The work \cite{He} of Heinloth can be easily adapted to show that the CoHA product $\alpha_*q^*$ is thus the usual algebra structure on
$$\Sym(\text{V}_C)\ =\ \bigoplus_{e\ge 0} \text{H}^\cdot(\text{Coh}_0^e).$$
Since $\alpha$ is generically finite, $\mathbf{T}_\alpha$ generically vanishes.

\subsubsection{} Turn secondly to  
\begin{center}
\begin{tikzcd}[bo column sep]
 &\text{Coh}_{0,1}^{e,d}\arrow[rd,"\beta"]\arrow[ld,swap,"q"]&\\
 \text{Coh}_0^e\times \text{Coh}_1^{d}&&\text{Coh}^{e+d}_1 
\end{tikzcd}
\end{center}
Whilst $\beta_*$ is complicated to work out, its stratified pieces with respect to the stratification $\text{Coh}_1^{d+e,\ell}\subseteq \text{Coh}_1^{d+e}$ by length $\ell$ of torsion subsheaf are easy to compute. Setting $f=d+e-\ell$, we have
\begin{center}
\begin{tikzcd}
\text{Coh}_0^e\times (\text{Coh}_0^{\ell-e}\times\text{Pic}^f)\arrow[d]& \text{Coh}_{0,0}^{e,\ell-e}\times \text{Pic}^f\arrow[r]\arrow[l,swap]\arrow[d]&\text{Coh}_0^\ell\times \text{Pic}^f\arrow[d]\\

\text{Coh}_0^e\times\text{Coh}_1^{d,\ell-e}\arrow[d]& \text{Coh}_{0,1}^{e,d,\ell}\arrow[r,"\beta_\ell"]\arrow[l,swap,"q_\ell"]\arrow[d]&\text{Coh}_1^{d+e,\ell}\arrow[d]\\

\text{Coh}_0^e\times \text{Coh}_1^{d}&\text{Coh}_{0,1}^{e,d}\arrow[r,"\beta"]\arrow[l,"q",swap]&\text{Coh}_1^{d+e}
\end{tikzcd}
\end{center}
We have defined $\text{Coh}_{0,1}^{d,e,\ell}$ so the lower right square is Cartesian. The upper right square is Cartesian because there are no nonzero maps from a torsion sheaf into a line bundle.
 The top row of vertical arrows are all vector bundles, so that 
$$\beta_{\ell,*}q_\ell^*\ =\ \text{rank zero CoHA product}\ \otimes \ \text{id}_{\text{H}^\cdot(\text{Pic}^f)}.$$
 Now, since $q^*$ and $\beta_*$ are uniquely determined by their restriction to the strata (see appendix \ref{stratifications}), the above uniquely determines the first postive rank CoHA product $\beta_*q^*$. Moreover, by the above $\mathbf{T}_{\beta_\ell}$ are given in terms of $\mathbf{T}_\alpha$.

\subsubsection{}\label{kevin} Finally we consider the last positive rank case
\begin{center}
\begin{tikzcd}[bo column sep]
&\text{Coh}_{1,0}^{d,e}\arrow[rd,"\gamma"]\arrow[ld,swap,"q"]& \\
\text{Coh}_1^d\times \text{Coh}_0^{e}&&\text{Coh}^{d+e}_1
\end{tikzcd}
\end{center}
Partial information about the CoHA product can be computed by stratifying the base of $\gamma$:
\begin{center}
\begin{tikzcd}
& \text{Coh}_{1,0}^{d,e,\ell}\arrow[r,"\gamma_\ell"]\arrow[d]&\text{Coh}_1^{d+e,\ell}\arrow[d]\\

 \text{Coh}_1^{d}\times \text{Coh}_0^e&\text{Coh}_{1,0}^{d,e}\arrow[r,"\gamma"]\arrow[l,"q",swap]&\text{Coh}_1^{d+e}
\end{tikzcd}
\end{center}
However, in contrast to last section, $\text{Coh}_{1,0}^{d,e,\ell}$ is more complicated because there exist nonzero maps from line bundles into torsion sheaves. To proceed in computing $\gamma_{\ell*}$ we apply an argument suggested to us by Kevin Lin. First fix some notation:

\begin{enumerate}[label = \arabic*)]
\item $\text{Coh}^{d+e,\ell}_1$ classifies rank one degree $d+e$ coherent sheaves $\mathcal{E}$ whose torsion part $\mathcal{T}$ has length $\ell$. That is to say, it classifies short exact sequences
$$\mathcal{T}\ \longrightarrow\ \mathcal{E}\ \longrightarrow\ \mathcal{Q}$$
of a degree $\ell$ torsion sheaf $\mathcal{T}$ by a degree $f=d+e-\ell$ line bundle $\mathcal{Q}$.

\item $\text{Coh}_{1,0}^{d,e}$ classifies extensions
$$0\ \longrightarrow\ \mathcal{E}'\ \longrightarrow \ \mathcal{E}\ \longrightarrow\ \mathcal{E}''\ \longrightarrow\ 0$$
where $\mathcal{E}'$ has rank one and degree $d$, and $\mathcal{E}''$ has rank zero and degree $e$.
\item \label{extensionspace} The pullback $\text{Coh}_{1,0}^{d,e,\ell}$ classifies
\begin{center}
\begin{tikzcd}

\mathcal{T}'\arrow[r]\arrow[d]& \mathcal{E}'\arrow[r]\arrow[d] & \mathcal{Q}'\arrow[d] & & ? & d&?\\

\mathcal{T}\arrow[r]\arrow[d]& \mathcal{E}\arrow[r]\arrow[d] & \mathcal{Q}\arrow[d] & & \ell& d+e&f\\

\mathcal{T}''\arrow[r]& \mathcal{E}''\arrow[r] & \mathcal{Q}'' & & ? & e& ? 

\end{tikzcd}
\end{center}
where all rows and columns are exact sequences, the left horizontal arrows are maximal torsion subsheaves, and the degrees are indicated on the right.

\end{enumerate}
Note that fixing one of the $?$'s determines the rest. Thus there is a stratification $\text{Coh}_{1,0}^{d,e,\ell,\ell'}\subseteq \text{Coh}_{1,0}^{d,e,\ell}$ given by bounding the length of $\mathcal{T}'$, classifying the same data as above, except that the degrees are fixed:
\begin{center}
\begin{tikzcd}

\mathcal{T}'\arrow[r]\arrow[d]& \mathcal{E}'\arrow[r]\arrow[d] & \mathcal{Q}'\arrow[d] & & \le \ell' & d&\ge f'\\

\mathcal{T}\arrow[r]\arrow[d]& \mathcal{E}\arrow[r]\arrow[d] & \mathcal{Q}\arrow[d] & & \ell& d+e&f\\

\mathcal{T}''\arrow[r]& \mathcal{E}''\arrow[r] & \mathcal{Q}'' & & \ge \ell'' & e& \le f''

\end{tikzcd}
\end{center}
Notice that the strata are labelled by $0\le \ell'\le \ell$, so in particular there are finitely many strata. 

\begin{enumerate}[label = \arabic*)] \setcounter{enumi}{2}
\item Write $\mathcal{M}$ for the space classifying
\begin{center}
\begin{tikzcd}

\mathcal{T}'\arrow[d]&  & \mathcal{Q}'\arrow[d] & & \ell' & &f'\\

\mathcal{T}\arrow[r]\arrow[d]& \mathcal{E}\arrow[r] & \mathcal{Q}\arrow[d] & & \ell& d+e&f\\

\mathcal{T}''&  & \mathcal{Q}'' & & \ell'' & & f''

\end{tikzcd}
\end{center}
with notation as in (\ref{extensionspace}) above. Similarly, write $\widetilde{\mathcal{M}}$ for the space classifying
\begin{center}
\begin{tikzcd}

\mathcal{T}'\arrow[d]\arrow[r]&  \mathcal{T}'\oplus\mathcal{Q}'\arrow[r]\arrow[d]& \mathcal{Q}'\arrow[d] & & \ell' & d&f'\\

\mathcal{T}\arrow[r]\arrow[d]& \mathcal{E}\arrow[r]\arrow[d] & \mathcal{Q}\arrow[d] & & \ell& d+e&f\\

\mathcal{T}''\arrow[r]& \mathcal{T}''\oplus\mathcal{Q}''\arrow[r] & \mathcal{Q}'' & & \ell'' & e& f''

\end{tikzcd}
\end{center}
\end{enumerate}
The point of considering $\mathcal{M}$ is that we have the pullback
\begin{center}
\begin{tikzcd}
\text{Coh}_{0,0}^{\ell'',\ell''}\times \text{Pic}_{1,0}^{f',f''}\arrow[r]\arrow[d] & \mathcal{M}\arrow[d,"\pi"]\\
\text{Coh}^\ell_0\times \text{Pic}^f  \arrow[r] & \text{Coh}_1^{d+e,\ell}
\end{tikzcd}
\end{center}
and the horizontal arrows give isomorphisms on cohomology, so  $\pi_*$ is easy to compute. This can be used to gain information about $\gamma_{\ell*}$, using the diagram
\begin{center}
\begin{tikzcd}
&\widetilde{\mathcal{M}}\arrow[ddd,"\widetilde{\pi}"]\arrow[rd, bend left = 10]\arrow[ld,  bend right = 10]&\\
\mathcal{M}\arrow[rdd, bend right = 25,"\pi",swap]&&\text{Coh}_{1,0}^{d,e,\ell,\ell'}\arrow[d,"j_{\ell'}"]\\
&& \text{Coh}_{1,0}^{d,e,\ell}\arrow[ld, "\gamma_\ell"]\\
&\text{Coh}_1^{d+e,\ell}&
\end{tikzcd}
\end{center}
The bottom two maps are proper, and the top two are affine space fibrations. Applying cohomology to $\widetilde{\pi}_!\widetilde{\pi}^!k\to k$ thus gives
\begin{center}
\begin{tikzcd}
&\text{H}^\cdot(\widetilde{\pi}_!k)\arrow[rd, bend left = 10,"\sim"]\arrow[ld,  bend right = 10,"\sim",swap]&\\
\text{H}^\cdot(\mathcal{M})\arrow[rdd, bend right = 25,"\pi_*",swap]&&\text{H}^\cdot(\text{Coh}_{1,0}^{d,e,\ell,\ell'},j_{\ell'!}k)\arrow[d]\\
&& \text{H}^\cdot(\text{Coh}_{1,0}^{d,e,\ell,\ell'})\arrow[ld, "\gamma_{\ell*}"]\\
&\text{H}^\cdot(\text{Coh}_1^{d+e,\ell})&
\end{tikzcd}
\end{center}
where we have omitted grading shifts from the notation. This determines what $\gamma_{\ell*}$ is on the image of $\text{H}^\cdot(\text{Coh}_{1,0}^{d,e,\ell,\ell'},j_{\ell'!}k)$.

\newpage

\appendix

\section{Spaces and sheaves} \label{sheaves}

\subsection{} Grothendieck's \textit{six functor formalism} is an invaluable enhancement of the notion of cohomology. Standard properties of cohomology are lifted to the category $\text{Sh}(X)$ of \textit{sheaves} on the space $X$.

\subsection{} Fix a category $\mathcal{C}$ of \textit{spaces}. Examples will include topological spaces, schemes, stacks, or derived stacks. Following [CD], a \textbf{sheaf theory} on $\mathcal{C}$ is an assignment
$$X\in\mathcal{C}\ \ \rightsquigarrow \ \ \text{Sh}(X)$$ of a triangulated category $\text{Sh}(X)$ to every space $X\in\mathcal{C}$, and an assignment to every map in $\mathcal{C}$ two pairs of adjoint triangulated functors $(f^*,f_*),(f_!,f^!)$
\begin{center}
\begin{tikzcd}
 X\stackrel{f}{\to} Y &[-20pt] \rightsquigarrow &[-20pt] \text{Sh}(X)\arrow[r,shift left=1,"f_*\text{,}f^!"]& \text{Sh}(Y) \arrow[l,shift left=1,"f^*\text{,}f^!"]
\end{tikzcd}
\end{center}
The assignment $f\mapsto f_*,f_!$ (resp. $f^*,f^!$) should define covariant (resp. contravariant) $2$-functors. The final piece of data two bifunctors $\otimes$ and $\mathcal{H}\text{om}$ defined on $\text{Sh}(X)$, satisfying tensor-hom adjunction:
$$\text{Hom}(\mathcal{A}\otimes\mathcal{B},\mathcal{C})\ \simeq\ \text{Hom}(\mathcal{A},\mathcal{H}\text{om}(\mathcal{B},\mathcal{C}),$$
and such that $f^*$ is monoidal with respect to $\otimes$.

\subsubsection{} These are the \textit{six functors}, and are required to satisfy the following:

\begin{enumerate}[label = \arabic*)]
\item \label{one} There is a natural transformation $f_!\to f_*$, which is an isomorphism if $f$ is proper. If $j$ is an open embedding, $j^*=j^!$. If $f$ is smooth of dimension $d$ then $f^!=f^*[2d]$.\footnote{In this paper we will ignore all Tate twists.}
\item \textit{Kashiwara's theorem}. If $i$ is a closed embedding then $i_*i^*=\text{id}$, so
$$i_*\ :\ \text{Sh}(X)\ \longrightarrow\ \text{Sh}(Y)$$
is a fully faithful map. It induces an equivalence $\text{Sh}(X)\simeq\text{Sh}_X(Y)$ to sheaves on $Y$ supported along $X$. In particular, $\text{Sh}(X)=\text{Sh}(X_{red})$ only depends on the reduced structure of $X$, and if $\mathcal{C}$ is the category of dg stacks or schemes,
$$\text{Sh}(X)\ =\ \text{Sh}(X_{cl}).$$
\item \textit{Mayer-Vietoris}. If $i$ and $j$ are complementary closed and open embeddings, there is a  distinguished triangle
$$i_*i^!\ \longrightarrow\ \text{id}\ \longrightarrow\ j_*j^!\ \stackrel{+1}{\longrightarrow}$$
called the \textit{Mayer-Vietoris} triangle. Applying $i^*$ to this then gives what we shall call the \textit{Gysin} triangle
$$i^!\ \longrightarrow\ i^*\ \longrightarrow\ i^*j_*j^*\ \stackrel{+1}{\longrightarrow}$$
\item \textit{Base change}. Take a Cartesian square
\begin{center}
\begin{tikzcd}
\overline{X}\arrow[r,"\overline{f}"]\arrow[d,"\overline{g}"]&\overline{Y}\arrow[d,"g"]\\
X\arrow[r,"f"]& Y
\end{tikzcd}
\end{center}
Then if $f$ is separated of finite type, there are natural isomorphisms 
$$g^*f_!\ \stackrel{\sim}{\longrightarrow}\ \overline{f}_!\overline{g}^*\ \ \ \text{ and }\ \ \ \overline{g}_*\overline{f}^!\ \stackrel{\sim}{\longrightarrow}\ f^!g_*.$$
Combining this with (\ref{one}) gives the standard statements of smooth base change and proper base change.
\item Sheaves on a point is the category of dg vector spaces, $\text{Sh}(\text{pt})=\text{Vect}_k$.
\item If $f$ is separated of finite type, then there are natural isomorphisms
$$\textcolor{white}{aaa}f_!A\otimes_YB\ \stackrel{\sim}{\longrightarrow}\ f_!(A\otimes_X f^*B),$$
$$\textcolor{white}{aaaaaa} \mathcal{H}\text{om}_Y(f_!A,B)\ \stackrel{\sim}{\longrightarrow}\ f_*\mathcal{H}\text{om}_X(A,f^!B), \ \ \ f^!\mathcal{H}\text{om}_Y(B,B')\ \stackrel{\sim}{\longrightarrow}\ \mathcal{H}\text{om}_X(f^*B,f^*B').$$
\end{enumerate}

\subsubsection{} Almost all our arguments work for general sheaf theories. Our last requirement will be that $\text{H}^\cdot(\mathbf{A}^1) = k$, that the cohomology of the affine line is trivial. This is because we will need that
$$\text{H}^\cdot(\text{B}\mathbf{G}_m)\ =\ k[t],$$
is freely generated by a degree two generator. This can be deduced from the fact that $\text{H}^\cdot(\mathbf{G}_m)=k[u]/u^2$ using the arguments section \ref{shclass}, which in turn follows from the triviality of $\text{H}^\cdot(\mathbf{A}^1)$ by Mayer-Vietoris.

\subsubsection{} There is a distinguished sheaf on the point, the one dimensional vector space in degree zero $k\in\text{Sh}(\text{pt})$, and so pulling back by the projection $p:X\to\text{pt}$ gives two sheaves on any space $k_X = p^*k$ and $\omega_X =p^!k$, the so-called \textit{constant} and \textit{dualising} sheaves. They are functorial with respect to maps as $f^*k_Y= k_X$ and $f^!\omega_X= \omega_Y$ and  if $X$ is smooth then $\omega_X\simeq k_X[2d_X]$ by condition (\ref{one}). Next, for any space we can use the dualising sheaf to define the \textit{Verdier duality} functor
$$\mathbf{D}_X\ :\ \text{Sh}(X)\ \stackrel{\sim}{\longrightarrow}\ \text{Sh}(X)^{op}\textcolor{white}{aaaaaaaa} A\ \longmapsto\ \mathcal{H}\text{om}_X(A,\omega_X).$$
It is self-dual, and exchanges the four functors $f_!=  \mathbf{D}_Yf_*\mathbf{D}_X$ and $f^!=\mathbf{D}_Xf^*\mathbf{D}_Y$. On $X=\text{pt}$ Verdier duality simply takes dual vector spaces.

\subsection{} Cohomology is easily recoverable. The \textit{cohomology} 
$$\text{H}^\cdot\ :\ \text{Sh}(X)\ \longrightarrow\ \text{Sh}(\text{pt})\ =\ \text{Vect}$$
is the pushforward $p_*$ along $p:X\to \text{pt}$. By adjunction, this is the same thing as $\text{Maps}(k_X,-)$. Compactly supported cohomology $\text{H}^\cdot_c$ is likewise defined as $p_!$, and Verdier duality gives
$$\text{H}^\cdot(X,A) \ =\ \text{H}^\cdot_c(X,\mathbf{D}_XA)^\vee.$$
\subsubsection{}If $X$ is smooth and $A$ is the constant sheaf, this recovers Poincar{\'e} duality. Finally, if a group $G$ acts on a space $X$
$$\text{H}^\cdot_G(X,A)\ =\ \text{H}^\cdot(X/G,A)$$
is called the \textit{equivariant cohomology} of the sheaf  $A$ on the stack quotient $X/G$. 

\subsubsection{}To give a flavour of how the structures on cohomology are recovered on the level of sheaves, note that the unit $A\to f_*f^*A$ of the adjunction induces the pullback map on cohomology
$$f^*\ :\ \text{H}^\cdot(Y,A)\ \longrightarrow\ \text{H}^\cdot(X,f^*A)$$
and so in this sense cohomology is functorial. The counit $f_!f^!A\to A$ gives a pushforward map if $f$ is smooth and proper
$$f_*\ :\ \text{H}^{*+2\dim f}(X,f^*A)\ \longrightarrow\ \text{H}^\cdot(Y,A).$$
See \ref{umkehr} for a further discussion of this map. Similarly, the unit $A\boxtimes A'\to \Delta_*\Delta^*A\boxtimes A'$ of the diagonal map $\Delta:X\to X\times X$ gives the cup product
$$\cup\ :\ \text{H}^\cdot(X,A)\otimes \text{H}^\cdot(X,A')\ \longrightarrow \ \text{H}^\cdot(X,A\otimes A').$$
In particular it follows that the cohomology of every sheaf is a module over the algebra $\text{H}^\cdot(X)=\text{H}^\cdot(X,k_X)$.

\subsection{Examples} Historically the first example was the category of $\ell$-adic sheaves. Fix a base field whose characteristic is prime to $\ell$. Then one can consider
$$\text{Sh}(X)\ =\ \text{D}_c^b(X,\mathbf{Q}_\ell)$$
the bounded derived category of constructible $\ell$-adic sheaves. Grothendieck and others initially developed this theory for $X$ a finite type separated scheme, but recently in [LZ12] and [LZ14] the formalism has been extended to higher Artin stacks. The functors $f_*,f^*$ exist for arbitrary maps $f$, and $f_!,f^!$ exist when the spaces are locally of finite type. The theory extends to derived higher Artin stacks using $\text{Sh}(X)=\text{Sh}(X_{cl})$.

\subsubsection{}Associated to any variety over a field of characteristic zero is the derived category $D(X)$ of its coherent $D$-modules. There is a subcategory
$$\text{Sh}(X)\ =\ D_{hol}(X)$$
of \textit{holonomic} D modules, which defines a sheaf theory. See for instance \cite{Be} or \cite{HTT}. For more general spaces $X$ the category $D(X)$ is well-studied, as in [GR11] and \cite{BD}, but the literature on $D_{hol}(X)$ is sparser.

\subsubsection{} Other sheaf theories admitting the six functor formalism include topological spaces with
$$\text{Sh}(X)\ =\ \text{D}^+_{\text{Ab}}(X)$$
 simply the (bounded below) derived category of sheaves of abelian groups on $X$. See \cite{I} for an account of this. There is also arithmetic $D$-modules algebraic stacks (giving $p$-adic cohomology, see [A]), the category of mixed Hodge modules (see \cite{Sa}), and a universal version for motivic homotopy theory (see \cite{DGal}).

\subsection{Bivariant homology} \label{biv} Fulton and MacPherson discovered a modification of usual cohomology which is much better suited to the study of singular spaces. In the more modern language of the six functors, \cite{Deg} defines the \textit{bivariant homology} of a map $f:X\to S$ to be
$$\text{H}^\cdot(X/S)\ =\ \text{H}^\cdot(X,f^!k_S).$$
Equivalently, it consists of maps $k_X\to f^!k_S$ in $\text{Sh}(X)$. Bivariant homology interpolates between ordinary cohomology and Borel-Moore homology
$$\text{H}^\cdot(X/X)\ =\ \text{H}^\cdot(X), \textcolor{white}{aaaaaaaaa} \text{H}^\cdot(X/\text{pt})\ =\ \text{H}^\cdot_{BM}(X).$$
It carries three basic operations:
\begin{enumerate} [label = \arabic*)]
\item A \textit{product} map 
$$\cdot\ :\ \text{H}^\cdot(Z/Y)\otimes \text{H}^\cdot(Y/X)\ \longrightarrow\ \text{H}^\cdot(Z/X)$$
for any morphisms $Z\stackrel{f}{\to}Y\stackrel{g}{\to} X$. The product  $\alpha\cdot\beta$ is $k_Z\stackrel{\alpha}{\to}f^!k_Y\stackrel{f^!\beta}{\to} f^!g^!k_X$ viewed as a bivariant homology class.

\item A \textit{proper pushforward} map. In in the above $f$ is proper, we get
$$f_*\ :\ \text{H}^\cdot(Z/X)\ \longrightarrow\ \text{H}^\cdot(Y/X)$$
given by $f_*f^!g^!k_Z=f_!f^!g^!k_Z\to g^!k_Z$.

\item A \textit{pullback} map 
$$g^*\ :\ \text{H}^\cdot(Y/X)\ \longrightarrow\ \text{H}^\cdot(Y'/X')$$ associated to any a Cartesian diagram
\begin{center}
\begin{tikzcd}
Y'\arrow[r,"f'"]\arrow[d,"g'"]& X'\arrow[d,"g"]\\
Y\arrow[r,"f"]& X
\end{tikzcd}
\end{center}
given by $f^!k_{X}\to f^!g_*g^*k_{X} = g'_*{f'}^!k_{X'}.$
\end{enumerate}
Moreover it is a \textit{bivariant theory} as per section $7.4$ of \cite{FM}, in the sense that the three operations satisfy the following conditions:
\begin{enumerate}
\item[$A_1$)] The product is associative: $(\alpha\cdot \beta)\cdot \gamma=\alpha\cdot (\beta\cdot\gamma)$.
\item[$A_2$)] Pushforward is functorial: if $f_1,f_2$ are composable proper maps, $f_{1*}(f_{2*}(\alpha))=(f_1f_2)_*(\alpha)$.
\item[$A_3$)] Pullback is functorial: if $g_1,g_2$ are composable maps, $g_1^*(g_2^*(\alpha))=(g_1g_2)^*(\alpha)$.
\item[$A_{13}$)] Product and pullback commute: $g^*(\alpha\cdot \beta)=g^*(\alpha)\cdot g^*(\beta)$.
\item[$A_{23}$)] Pushforward and pullback commute: $g^*f_*\alpha=f'_*g^*\alpha$ for any Cartesian diagrams
\begin{center}
\begin{tikzcd}
Z'\arrow[r,"f'"]\arrow[d] & Y'\arrow[r]\arrow[d]& X'\arrow[d,"g"]\\
Z\arrow[r,"f"]& Y\arrow[r]& X
\end{tikzcd}
\end{center}
where $f$ is proper, and any class $\alpha\in \text{H}^\cdot(Z/X)$.
\item[$A_{123}$)] Projection formula: $\beta\cdot f_*\alpha= f'_*({g'}^*\beta\cdot \alpha)$ for any Cartesian diagram
\begin{center}
\begin{tikzcd}
Z'\arrow[r,"f'"]\arrow[d] & Y'\arrow[d,"g'"]\\
Z\arrow[r,"f"]& Y\arrow[r]& X
\end{tikzcd}
\end{center}
where $f$ is proper, and classes $\alpha\in \text{H}^\cdot(Z/X)$ and $\beta\in \text{H}^\cdot(Y'/Y)$.

\item[$C$)] Skew-commutativity: $g^*\alpha\cdot \beta=(-1)^{\deg(\alpha)\cdot \deg(\beta)}f^*(\beta)\cdot \alpha$ for any Cartesian diagram 
\begin{center}
\begin{tikzcd}
Y'\arrow[r,"f'"]\arrow[d,"g'"]& X'\arrow[d,"g"]\\
Y\arrow[r,"f"]& X
\end{tikzcd}
\end{center}
and classes $\alpha\in \text{H}^\cdot(Y/X)$ and $\beta\in \text{H}^\cdot(X'/X)$.
\end{enumerate}

\section{Total spaces of perfect complexes} \label{tot}
\subsection{} The \textit{total space} of a locally free sheaf $E$ (i.e. vector bundle) on a derived scheme $X$ is 
$$\Spec _{\mathcal{O}_X}\Sym E^\vee,$$
which is usually also denoted by $E$. This construction makes sense for any sheaf of $\mathcal{O}_X$ modules. This construction can also extended for $E$ a general quasicoherent complex, obtaining an $\infty$-functor
$$\text{Tot}\ :\ \{\text{perfect complexes over }X\}\ \longrightarrow\ \{\text{derived stacks over }X\}.$$
For details see subsection $3.3$ of \cite{T}, where they denote $\text{Tot}(E)=\mathbf{V}(E^\vee)$, or section $2$ of \cite{KV}. If $E$ is concentrated in the ``derived direction'' (has tor-amplitude in $[0,\infty)$) then 
$$\text{Tot}(E)\ =\ \Spec _{\mathcal{O}_X}\Sym E^\vee$$
is a derived scheme over $X$. If $E$ is concentrated in the ``stacky direction'' (has tor-amplitude in $(-\infty,0]$) then $E\to X$ is smooth of dimension $\text{rk}E$ and its fibres are classical higher stacks. For instance, if $E=(E_{-1}\to E_0)$ is a complex of vector bundles in degrees $-1$ and $0$ then
$$\text{Tot}(E)\ = \ [E_0/E_{-1}]$$
is a quotient stack by the linear action induced by $E_{-1}\to E_0$.

\subsection{} Let $E$ be a perfect complex, which from now on we conflate with its total space. Then the (co)tangent complexes of the projection $\pi:E\to X$ are
$$\mathbf{L}_\pi\ =\ \pi^*E^\vee, \ \ \ \mathbf{T}_\pi\ =\ \pi^*E.$$
It follows that $E\to X$ is quasismooth if and only if $E$ is concentrated in tor-amplitude $(-\infty,1]$. One can show that the zero section $X\to E$ is quasismooth if and only if $E$ has tor-amplitude in $(-\infty,0]$.

\subsection{} Associated to any distinguished triangle $E_1\to E_2\to E_3\stackrel{+1}{\to}$ of perfect complexes is a diagram of infinitely many squares of derived stacks
\begin{center}
\begin{tikzcd}[bo column sep]
&X\arrow[rd]&&X\arrow[rd]&&X\arrow[rd]&\\
\cdots &&E_1\arrow[ru]\arrow[rd]&&E_3\arrow[ru]\arrow[rd]&&E_2[1]\\
&E_3[-1]\arrow[ru]\arrow[rd]&&E_2\arrow[ru]\arrow[rd]&&E_1[1]\arrow[ru]\arrow[rd]&&\cdots \\
&&X\arrow[ru]&&X\arrow[ru]&&X
\end{tikzcd}
\end{center}
which are each both pushouts and pullbacks.

\section{Stratifications} \label{stratifications}

\subsection{} \label{closedstrat} Let $X$ be a smooth finite type space with a decreasing union of closed subspaces
$$X\ =\ Z_0\ \supseteq  \ Z_1\ \supseteq \ \cdots \textcolor{white}{aaaaaaaaa}\text{ with }\ \ \bigcap Z_k\ =\ \emptyset.$$
Thus $X$ is stratified by locally closed subspaces 
$$X_k\ =\ Z_{k-1}\setminus Z_{k}.$$
Assume that the $X_k$ are smooth, and that the Euler class of the normal bundle to $X_k$ inside $X$ is not a zero divisor. 

\begin{lem}
For $X$ as above, the restriction map $\text{H}^\cdot(X)\to \bigoplus \text{H}^\cdot(X_k)$ is injective.
\end{lem}
\begin{proof}
By the finite type assumption, the codimension of $X_k$ in $X$ tends to infinity as $k\to\infty$. It follows that
$$\text{H}^\cdot(X)\ =\ \varprojlim \text{H}^\cdot(X\setminus Z_k)$$
so it suffices for us to prove that
$$\text{H}^\cdot(X\setminus Z_k)\ \longrightarrow\ \bigoplus_{n=0}^{k-1}\text{H}^\cdot(X_n)$$
is injective. It is the identity when $k=1$. More generally, the Gysin sequence gives
$$\cdots\ \longrightarrow \ \text{H}^{\cdot+2c_{k-1}}(X_{k-1})\ \longrightarrow\ \text{H}^{\cdot}(X\setminus Z_{k})\ \longrightarrow\ \text{H}^{\cdot}(X\setminus Z_{k-1})\ \longrightarrow\ \cdots$$
By the assumption on the Euler class, this becomes:
$$0\ \longrightarrow \ \text{H}^{\cdot+2c_{k-1}}(X_{k-1})\ \longrightarrow\ \text{H}^{\cdot}(X\setminus Z_{k})\ \stackrel{\beta}{\longrightarrow}\ \text{H}^{\cdot}(X\setminus Z_{k-1})\ \longrightarrow\ 0.$$
A class in $\text{H}^\cdot(X\setminus Z_k)$ restricts to zero on the strata $X_0,...,X_{k-2}$ precisely if it lies in the kernel of $\beta$, i.e. is the pushforward of a class $\alpha$ from $X_{k-1}$. Restricting this class back to $X_{k-1}$ gives $\alpha$ times the Euler class, so is zero if and only if $\alpha=0$. Thus we are done: a class in $\text{H}^\cdot(X\setminus Z_k)$ restricts to zero on $X_0,...,X_{k-2},X_{k-1}$ only if it is zero.
\end{proof}

\subsection{} Let $f:X\to Y$ be a map between two stratified spaces as above. If it preserves the strata, i.e. induces maps
\begin{equation}\label{diag5}
\begin{tikzcd}
X_n\arrow[r,"\overline{\jmath}_n"]\arrow[d,"f_n"] & X\arrow[d,"f"]\\
Y_n\arrow[r,"j_n"]& Y
\end{tikzcd}
\end{equation}
We have commuting diagrams
\begin{center}
\begin{tikzcd}
\text{H}^\cdot(Y)\arrow[d,"f^*"]\arrow[r]& \bigoplus \text{H}^\cdot(Y_n)\arrow[d,"f_{n}^*"]&& \text{H}^\cdot(X)\arrow[d,"f_*"]\arrow[r]& \bigoplus \text{H}^\cdot(X_n)\arrow[d,"f_{n*}"]\\
\text{H}^\cdot(X)\arrow[r]& \bigoplus \text{H}^\cdot(X_n)&& \text{H}^{\cdot-2d}(Y)\arrow[r]& \bigoplus \text{H}^{\cdot-2d}(Y_n)
\end{tikzcd}
\end{center}
the right existing under the additional assumption that $f$ is proper, and transverse to the strata in the sense that (\ref{diag5}) is a pullback of \textit{derived} stacks. We have written $d=\dim X-\dim Y$ for its relative dimension. It follows from the previous lemma that 
\begin{cor}
$f^*$ and $f_*$ are uniquely determined by their restrictions $f_n^*$ and $f_{n*}$ to the strata.
\end{cor}

\section{Stabiliser groups}  \label{stab}

\subsection{} \label{stabsch} We define stabiliser subgroups of quotient stacks of schemes, recalling section $8.4.1$ of \cite{O}. Let $X\to Y$ be a map of schemes, and $\text{G}/Y$ a smooth group scheme acting on $X$. Fix a point
$$x\ :\ \Spec k\ \longrightarrow\ X$$
where $k$ is algebraically closed, and let $y$ be its image in $Y$. Note that there is a map
$$\lambda\ :\ \text{G}_y\ \longrightarrow\ X_y\textcolor{white}{aaaaaaaaaa} g\ \longrightarrow\ g\cdot x.$$
We now write $\text{G}_x$ for the fibre product
\begin{center}
\begin{tikzcd}
\text{G}_x\arrow[r]\arrow[d] & \Spec k\arrow[d,"x"]\\
\text{G}_y\arrow[r,"\lambda"]& X_y
\end{tikzcd}
\end{center}
It defines a subgroup $\text{G}_x\subseteq \text{G}_y$. Moreover, the groups $\text{G}_x$ and $\text{G}_{g\cdot x}$ are conjugate, so that to the point
$$\overline{x}\ :\ \Spec k\ \longrightarrow\ X/\text{G}$$
we may associate a group scheme $\text{G}_{\overline{x}}$, the \textbf{stabiliser subgroup} of the the point $\overline{x}$.

\begin{prop}\label{dmprop}
The stack $X/\text{G}$ is Deligne-Mumford if and only if for every point $\overline{x}:\Spec k\to X/\text{G}$ with $k$ algebraically closed, $\text{G}_{\overline{x}}$ is etale over $\overline{x}$.
\end{prop}
\begin{proof}
This is corollary $8.4.2$ of \cite{O}.
\end{proof}

\subsection{} For our purposes we will need a relative version of the above. Let $X\to Y$ now be a representable map of derived Artin stacks and $\text{G}$ a smooth group scheme over $Y$ acting on $X$. So above each scheme $Y'\to Y$ we are in the situation of \ref{stabsch}. Because $X_y$ is a scheme, the discussion in \ref{stabsch} can be repeated verbatim, defining the \textbf{relative stabiliser subgroup} $\text{G}_{\overline{x}}$ associated to a point $\overline{x}:\Spec k\to X/\text{G}$.

Say that a map of stacks $X\to Y$ is \textit{relative Deligne-Mumford} if for any any Delgine Mumford stack $V$ mapping into $Y$, the pullback 
\begin{center}
\begin{tikzcd}
U\arrow[r]\arrow[d]& V\arrow[d]\\
X\arrow[r]& Y
\end{tikzcd}
\end{center}
is a Deligne Mumford stack. Or equivalently the above, but only only using $V$ a test \textit{scheme} rather than Deligne Mumford stack.

\begin{prop} \label{dmrelprop}
If all the relative stabiliser groups $\text{G}_{\overline{x}}$ are etale then $X/\text{G}\to Y$ is relative Deligne Mumford. 
\end{prop}
\begin{proof}
Let $V$ be a scheme admitting a map $V\to Y$, and $\widetilde{\text{G}}$ the pullback of $\text{G}$ by this map. It is a group scheme over $V$, and acts on $U$. There is a pullback
\begin{center}
\begin{tikzcd}
U/\widetilde{\text{G}}\arrow[d]\arrow[r]& V\arrow[d]\\
X/\text{G}\arrow[r]& Y
\end{tikzcd}
\end{center}
It is enough then to show that $U/\widetilde{\text{G}}$ is a Deligne Mumford stack. Note that $U$ is a scheme. Let
$$u\ :\ \Spec k\ \longrightarrow\ U$$
be a point in $U$, with $k$ algebraically closed, and $x,\overline{u},v,y$ its images in $X,U/\widetilde{\text{G}},V,Y$. To finish it is enough to show that $\widetilde{\text{G}}_{\overline{u}}$ is etale over $\overline{u}$. Note that
$$X_y \ =\ U_v$$
so that the two diagrams
\begin{center}
\begin{tikzcd}
\text{G}_x\arrow[r]\arrow[d] & \Spec k\arrow[d,"x"]&& \widetilde{\text{G}}_u\arrow[r]\arrow[d] & \Spec k\arrow[d,"u"]\\
\text{G}_{y}\arrow[r]& X_{y}&& \widetilde{\text{G}}_{v}\arrow[r]& U_{v}
\end{tikzcd}
\end{center}
are isomorphic. Thus as by assumption $\text{G}_x$ is etale over the point, it follows that so is $\widetilde{\text{G}}_u$ and so by Proposition \ref{dmprop} the quotient $U/\widetilde{\text{G}}$ is Deligne Mumford.
\end{proof}

\newpage

\label{references}

\end{document}